\documentclass[bst/sn-mathphys]{sn-jnl}

\usepackage{multicol}
\usepackage{amsmath}

\usepackage{caption}
\usepackage{subcaption}

\jyear{2023}%

\theoremstyle{thmstyleone}%
\newtheorem{theorem}{Theorem}
\newtheorem{proposition}[theorem]{Proposition}%

\newtheorem{lemma}[theorem]{Lemma}%

\newtheorem{corollary}[theorem]{Corollary}%
\newtheorem{assumption}{Assumption}%

\theoremstyle{thmstyletwo}%

\theoremstyle{thmstylethree}%
\newtheorem{definition}{Definition}%

\begin{document}

\title[Inexact Policy Iteration Methods]{Inexact Policy Iteration Methods for Large-Scale Markov Decision Processes}


\author*[1]{\fnm{Matilde} \sur{Gargiani}}\email{gmatilde@ethz.ch}

\author[1]{\fnm{Robin} \sur{Sieber}}\email{rosieber@student.ethz.ch}

\author[2]{\fnm{Efe} \sur{Balta}}\email{efe.balta@inspire.ch}

\author[3]{\fnm{Dominic} \sur{Liao-McPherson}}\email{dliaomcp@mech.ubc.ca}

\author[1]{\fnm{John} \sur{Lygeros}}\email{jlygeros@ethz.ch}

\affil*[1]{\orgdiv{Automatic Control Laboratory}, \orgname{ETH}, \orgaddress{\street{Physikstrasse 3}, \city{8092 Zurich}, \country{Switzerland}}}

\affil[2]{\orgdiv{Control and Automation Group}, \orgname{Inspire AG}, \orgaddress{\city{Zurich}, \country{Switzerland}}}

\affil[3]{\orgdiv{Department of Mechanical Engineering}, \orgname{UBC}, \orgaddress{\street{2054-6250 Applied Science Lane}, \city{Vancouver}, \state{British Columbia}, \country{Canada}}}



\abstract{We consider inexact policy iteration (iPI) methods for large-scale infinite-horizon discounted Markov decision processes (MDPs) with finite state and action spaces, a variant of policy iteration where the policy evaluation step is implemented inexactly using an iterative solver for linear systems. 
In the classical dynamic programming literature, a similar principle is deployed in optimistic policy iteration, where an a-priori fixed-number of iterations of value iteration is used to inexactly solve the policy evaluation step.
Inspired by the connection between policy iteration and semismooth Newton's method, we investigate a class of iPI methods that mimic the inexact variants of semismooth Newton's method by adopting a parametric stopping condition to regulate the level of inexactness of the policy evaluation step.
For this class of methods we discuss local and global convergence properties and derive a practical range of values for the stopping-condition parameter that provide contraction guarantees. We also show that, when the iterative solver for policy evaluation enjoys linear contraction, the stopping criterion is verified in a finite-number of iterations.
Our analysis is general and therefore encompasses a variety of iterative solvers for policy evaluation, including the standard value iteration as well as more sophisticated ones such as GMRES.
As underlined by our analysis, the selection of the inner solver is of fundamental importance for the performance of the overall method. We therefore consider different iterative methods to solve the policy evaluation step and analyse their applicability and contraction properties when used for policy evaluation. 
In particular, we study Richardson's method, steepest descent, the minimal residual method and GMRES for policy evaluation. 
We show that the contraction properties of these methods tend to be enhanced by the specific structure of policy evaluation and that there is margin for substantial improvement over value iteration in terms of convergence rate.  
Finally, we study the numerical performance of different instances of inexact policy iteration on large-scale MDPs for the design of health policies to control the spread of infectious diseases in epidemiology.
}




\maketitle

\section{Introduction}\label{sec1}

Markov decision processes (MDPs) are a mathematical model for dynamic decision-making problems under uncertainty~\cite{puterman_mdp} and have been used over the years in a variety of different applications, from agriculture~\cite{agriculture_MDP} to epidemiology~\cite{sis_model} and finance~\cite{finance_MDP}, to name a few. We refer to~\cite{White1993ASO} for a comprehensive survey on applications of MDPs. 

The size of these problems is usually very large. For systems modeled with a set of propositional
state variables, the state space grows exponentially with the
number of variables. This problem becomes even more
relevant for MDPs with continuous state-spaces. If the continuous space
is discretized to find a solution, the discretization causes
another level of exponential blow up. This phenomenon is known as the \textit{curse of dimensionality}~\cite{Bellman:1957} and limits the exact solution methods to problems with a compact representation which fits into the memory RAM of the hardware at hand. For these scenarios, it is fundamental to design exact algorithms that scale well with the dimensionality of the problem. This is the main motivation behind optimistic policy iteration (OPI), a variant of policy iteration (PI) which economizes in computation by approximating the policy evaluation step using a fixed finite number of value iteration updates~\cite[Chapter 2]{DB_book}.
Indeed, even if it locally enjoys a quadratic rate of convergence~\cite{gargiani_2022}, PI has expensive iterations. In the large-scale setting, the computational costs per-iteration are mostly due to the policy evaluation step, which requires the exact solution of a linear system with size equal to the state space of the MDP. Consequently, even though PI converges in very few
iterations, its time performance is not scalable as it degrades rapidly with the size of the state space of the MDP. 

A similar problem arises in Newton's method and its semismooth variants, where inexactness is introduced in the solution of the Newtonian linear systems in order to reduce the computational costs and, in general, the level of inexactness is regulated via a parametric stopping condition that depends on locally available information~\cite{inexact_Newton, MARTINEZ1995127}. Different iterative methods for linear systems can be deployed to produce the approximate solution and generally the choice is dictated by the specific structure of the linear systems~\cite{inexactnewtonMR, inexactgmresnewton, inexactnewtoncg}. Local convergence results which are independent from the selected iterative method are available in the literature~\cite{inexact_Newton, MARTINEZ1995127}. These inexact variants are collectively known as inexact Newton methods and are often deployed in large-scale scenarios as their performance is more scalable than that of Newton's method~\cite{DEUFLHARD1991366}. 

In light of the connection between PI and semismooth Newton's method~\cite{gargiani_2022}, we analyse the class of inexact policy iteration methods, which are characterized by an approximation of the solution of the policy evaluation step and where the level of inexactness is regulated by a parametric stopping condition that is borrowed from the inexact variants of Newton's method. This class of methods has been originally proposed in~\cite{gargiani_2023}, where the authors focus on the deployment of GMRES as inner solver, showcasing the promising performance of the resulting inexact scheme on large-scale artificial MDPs.   
Building upon these results, we extend the analysis of inexact policy iteration methods.
In particular, the main contributions of this paper can be summarized as follow:
\begin{itemize}
    \item a convergence analysis of inexact policy iteration methods which encompasses the choice of the inner solver, 
    \item a characterization of the convergence properties of different iterative methods for policy evaluation, which proves that value iteration for policy evaluation is not optimal in the sense of contraction rate for a certain class of MDPs and underlines the advantages of deploying different iterative methods than value iteration as inner solvers, 
    \item an empirical analysis of the performance of inexact policy iteration methods on MDPs arising from a dynamic extension of SIS models for epidemiology.
\end{itemize}

The paper is organized as follows. In Section~\ref{sec2} we describe the problem setting and propose a classification of MDPs based on the properties of the transition probability matrices of their induced Markov chains. We also briefly discuss the classic dynamic programming methods and their convergence and scalability properties. In Section~\ref{sec3} we discuss the connection between semismooth Newton's method and policy iteration, summarizing the main results of~\cite{gargiani_2022} and~\cite{gargiani_2023}. In addition, we provide an upper-bound on the global semismooth constant at the root of the Bellman residual function (Lemma~\ref{lemma: strong_semismoothness}) and derive an estimate for the region of attraction of semismooth Newton-type methods for DP (Theorem~\ref{th: semismooth_newton_convergence}). Section~\ref{sec4} is dedicated to inexact policy iteration methods. After a description of the methods in this class, we provide local and global convergence guarantees.  We also prove that, under some general assumptions on the inner solver contraction properties, the overall scheme is well-posed as the stopping condition is always met in a finite number of inner iterations when the parameter in the stopping condition is constant. This result stresses the importance of selecting a well-performing inner solver for the approximate policy evaluation step. We therefore proceed by characterizing the convergence properties of different iterative methods for linear systems when deployed for policy evaluation. In particular, we analyse Richardson's method~\ref{subsec:richardson}, the steepest descent method~\ref{subsec:steepest_descent}, the minimal residual method~\ref{subsec:minres} and GMRES~\ref{subsec:gmres}. 
Finally, in Section~\ref{sec6} we propose a dynamic extension of SIS models for epidemiology and illustrate how to reformulate them as MDPs. We then study the empirical performance of different instances of inexact policy iteration on MDPs constructed from dynamic SIS models. Our benchmarks showcase the advantages of inexact policy iteration methods on large-scale MDPs over policy iteration as well as the importance of a careful selection of the inner solver. 

\subsection{Notation}
We denote with $\mathbb{S}^{n\times n}\subset \mathbb{R}^{n \times n}$ the set of $n\times n$ row-stochastic matrices. Given an $n\times n$ real matrix $A$ with $W\leq n$ distinct eigenvalues, we use $A\succ 0$ to indicate that the matrix is symmetric positive definite, $\rho(A)$ to denote its spectral radius, $\kappa(A)$ for its condition number, $\mu_A(x)$ for its minimal polynomial, $A_s$ for its symmetric part, \textit{i.e.}, $A_s = \frac{1}{2}\left( A + A^{\top}\right)$ ,  $\Lambda(A)$ for its spectrum and  $\lambda_1, \lambda_2, \dots, \lambda_W$ for its distinct eigenvalues, where 
\begin{equation*}
    \vert \lambda_1 \vert \geq \vert \lambda_2 \vert \geq \dots \geq \vert \lambda_W \vert \,.
\end{equation*}
 We use $a(\lambda)$ to denote the algebraic multiplicity of $\lambda \in \Lambda(A)$, and $b(\lambda)$ for its associated degree in the minimal polynomial, \textit{i.e.},
\begin{equation*}
\mu_A(x) = \Pi_{\lambda \in \Lambda(A)} (x - \lambda)^{b(\lambda)}\,.
\end{equation*}
If $A$ is symmetric, then we use $\lambda_{\max}(A) = \max \left\{\lambda_1, \,\dots,\,\lambda_W\right\}$ and $\lambda_{\min}(A) = \min \left\{\lambda_1,\,\dots,\,\lambda_W\right\}$.
In addition, we use $a_{ij}$ and $a^{(q)}_{ij}$ for the $(i,j)$-element of $A$ and $A^q$, respectively. Analogously, given a vector $v \in \mathbb{R}^n$, we denote with $v_i$ its $i$-th entry. We call non-negative (positive) a square matrix whose elements are non-negative (positive). Given $\delta>0$ and $x^*\in \mathbb{R}^n$, we use $\mathcal{B}(x^*, \delta)$ to denote the set of points $x\in\mathbb{R}^n$ such that $\Vert x- x^*\Vert_{\infty} < \delta$.

\section{Infinite-Horizon Discounted Markov Decision Processes with Finite Spaces}\label{sec2}

A discounted MDP with finite spaces is a 5-tuple $\left\{ \mathcal{S}, \, \mathcal{A},\, P, \, g,\, \gamma \right\}$, where $\mathcal{S}$ and $\mathcal{A}$ are the finite sets of states and actions, respectively; $P: \mathcal{S}\times \mathcal{A} \times \mathcal{S} \rightarrow [0,1]$ is the transition probability function, where $P(i,a,j)$ defines the probability of transitioning to state $j$ when applying action $a$ in state $i$; $g:\mathcal{S}\times \mathcal{A}\rightarrow [-R, \, R]$ is the stage-cost function, which associates to each state-action pair a bounded cost; and $\gamma\in(0,1)$ is a discount factor. Without loss of generality, we consider $\mathcal{S} = \left\{ 1, \dots, n \right\}$ and $\mathcal{A} = \left\{ 1,\dots, m \right\}$.
Because of the presence of actions, MDPs can be regarded as an extension of Markov chains, where the actions are generally selected according to a given criterion called policy. There exists different types of policies, \textit{e.g.}, stochastic, non-stationary etc.... In this work we consider the subclass of deterministic stationary control policies. In particular, a deterministic stationary control policy $\pi : \mathcal{S} \rightarrow \mathcal{A}$ is a function that maps states to actions. We denote with $\Pi$ the finite set comprising all the deterministic stationary control policies, from now on simply called policies. 

Consider now a generic discounted MDP with finite spaces. At time step $t$ of the decision process under policy $\pi\in\Pi$, the system is in some state $s_t$ and the action $a_t = \pi(s_t)$ is applied. The discounted cost $\gamma^t g(s_t, a_t)$ is accrued and the system transitions to a state $s_{t+1}$ according to the probability distribution $P(s_t, a_t, \cdot)$. Starting from state $i$, this process is repeated over an infinite-horizon, leading to the following cumulative discounted cost
\begin{equation}\label{eq:cumulative_cost}
V^{\pi}(i) = \lim_{T\rightarrow \infty} \mathbb{E} \left[ \sum_{t=0}^{T-1} \gamma^t g(s_t, \pi(s_t))\, \Big\vert \, s_0 = i \right]\,,  
\end{equation} 
where $\left\{ s_0, \pi(s_0), s_1, \dots \right\}$ is the state-action sequence generated by the MDP under policy $\pi$ and with initial state $s_0$. The expected value is taken with respect to the corresponding transition probability measure over the space of sequences. The optimal cost is defined as
\begin{equation}\label{eq:optimal_cost}
V^*(i) = \min_{\pi\in\Pi} V^{\pi}(i) \quad \forall\,  i \in \mathcal{S}\,.
\end{equation}
Any policy $\pi^*\in\Pi$ that attains the optimal cost is called \textit{optimal policy}. Notice that we restrict our attention to deterministic stationary control policies as in the considered setting there provably exists a policy in this class that attains~\eqref{eq:optimal_cost}~\cite[Chapter 1]{DB_book}. 

Given the cost function $V:\mathcal{S}\rightarrow \mathbb{R}$, we call \textit{greedy} with respect to the cost $V$ any policy that satisfies the following equation
\begin{equation}
\pi(i) \in \arg\min\left\{ g(i,\pi(i)) + \gamma\, \mathbb{E}_{s'\sim P(i,\pi(i),\cdot)} \left[ V(s') \right] \right\} \quad \forall \, i \in \mathcal{S}.
\end{equation}
We use $\text{GreedyPolicy}(V)$ to denote the operator which extracts a greedy policy associated with $V$. 

Equations~\eqref{eq:cumulative_cost} and~\eqref{eq:optimal_cost} admit recursive formulations which are also known as the \textit{Bellman equations}. In particular, given $\pi\in\Pi$
\begin{equation}\label{eq:Bellman_eq_pi}
V^{\pi}(i) = g(i, \pi(i)) + \gamma \, \mathbb{E}_{s' \sim P(i, \pi(i), \cdot)}\left[ V^{\pi}(s') \right] \quad \forall\, i \in \mathcal{S}
\end{equation} 
is the Bellman equation associated with policy $\pi$, and
\begin{equation}\label{eq:Bellman_eq_pi*}
V^*(i) = \min_{\pi \in \Pi} \left\{ g(i, \pi(i)) + \gamma \, \mathbb{E}_{s' \sim P(i, \pi(i), \cdot)}\left[ V^{*}(s') \right] \right\} \quad \forall\, i \in \mathcal{S}
\end{equation}
is the Bellman equation associated with the optimal cost. Let $\mathcal{F}(\mathcal{S},\, \mathbb{R})$ denote the space of functions mapping from $\mathcal{S}$ to $\mathbb{R}$. Starting from the Bellman equations, we can define the mappings $T_{\pi}:\mathcal{F}(\mathcal{S},\, \mathbb{R})\rightarrow \mathcal{F}(\mathcal{S},\, \mathbb{R})$ with
\begin{equation}\label{eq:Tpi_operator}
(T_{\pi}V)(i) =  g(i, \pi(i)) + \gamma \, \mathbb{E}_{s' \sim P(i, \pi(i), \cdot)}\left[ V(s') \right] \quad i \in \mathcal{S} \,,
\end{equation}
and $T:\mathcal{F}(\mathcal{S},\, \mathbb{R}) \rightarrow \mathcal{F}(\mathcal{S},\, \mathbb{R})$ with
\begin{equation}\label{eq:T_operator}
(TV)(i) = \min_{\pi \in \Pi} \left\{ g(i, \pi(i)) + \gamma \, \mathbb{E}_{s' \sim P(i, \pi(i), \cdot)}\left[ V(s') \right] \right\} \quad i \in \mathcal{S} \,.
\end{equation}
These mappings are known as the \textit{Bellman operators} and are very practical for shorthand notation and algorithmic description~\cite[Chapter 1]{DB_book}. The Bellman operators enjoy three fundamental properties: monotonicity, $\gamma$-contractivity in the infinity-norm and shift-invariance. For completeness, we report them for the $T$-operator in the following propositions and we refer to~\cite[Chapter 1]{DB_book} for the analogous properties for the $T_{\pi}$-operator and for the technical proofs.
\begin{proposition}[Monotonicity]\label{prop: monotonicity}
    For any $V:\mathcal{S}\rightarrow \mathbb{R}$ and $V':\mathcal{S}\rightarrow \mathbb{R}$, and for $k=0,1,\dots$, then
    \begin{equation}
        \left((T^k V)(i) - (T^k V')(i)\right) \, \cdot \, \left(  V(i) - V'(i)  \right) \geq 0 \quad  i \in \mathcal{S}\,.
    \end{equation}
\end{proposition}

\begin{proposition}[$\gamma$-Contractivity]\label{prop: gamma_contractivity}
        For any $V:\mathcal{S}\rightarrow \mathbb{R}$ and $V':\mathcal{S}\rightarrow \mathbb{R}$, and for $k=0,1,\dots$, then
    \begin{equation}
        \max_{i \in \mathcal{S}} \, \vert (T^k V)(i) - (T^k V')(i) \vert \leq \gamma^k \max_{i \in \mathcal{S}} \, \vert V(i) - V'(i)\vert\,.
    \end{equation}
\end{proposition}

\begin{proposition}[Shift-Invariance]\label{prop: shift_invariance}
    For any $V:\mathcal{S}\rightarrow \mathbb{R}$, $r\in \mathbb{R}$ and $k=0,1,\dots$, then
    \begin{equation}
    (T^k(V + re))(i) = (T^k V)(i) + \gamma^k r\quad i \in \mathcal{S}\,,   
    \end{equation}
    where $e : \mathcal{S} \rightarrow \mathbb{R}$ is the unit function that takes value 1 identically on $\mathcal{S}$.
\end{proposition}
From their definitions we evince that the Bellman operators~\eqref{eq:Tpi_operator} and~\eqref{eq:T_operator} admit~\eqref{eq:cumulative_cost} and~\eqref{eq:optimal_cost} as fixed-points, respectively. Finally, uniqueness of their fixed-points follows directly from the Banach's fixed-point theorem~\cite{rockafellar_1976}.

Because we are dealing with finite state and action sets, any function $V$ on $\mathcal{S}$, including the mappings $T_{\pi}V$ and $TV$, can be represented as  an $n$-dimensional vector. With a slight abuse of notation, this leads to the following vector-notation
\begin{equation}\label{eq:compact_notation}
V = 
\begin{bmatrix}
V(1)\\
\vdots\\
V(n)
\end{bmatrix}\,,\quad
T_{\pi}V = 
\begin{bmatrix}
(T_{\pi}V)(1)\\
\vdots\\
(T_{\pi}V)(n)
\end{bmatrix}\,, \quad
TV = 
\begin{bmatrix}
(TV)(1)\\
\vdots\\
(TV)(n)
\end{bmatrix}\,, 
\end{equation}
which allows for a compact formulation of the Bellman equations, \textit{i.e.}, $V^{\pi} = T_{\pi}V^{\pi}$ and $V^{*} = TV^*$.
Finally, given a policy $\pi\in \Pi$, we can represent its associated transition probability distributions and stage-costs in matrix and vector-form, respectively. In particular, we denote with $P^{\pi} \in \mathbb{R}^{n\times n}$ its associated transition probability matrix, where $p^{\pi}_{ij} = P^{\pi}(i, \pi(i), j)$; and
with $g^{\pi}\in\mathbb{R}^{n}$ its associated cost-vector, where $g^{\pi}_{i} = g^{\pi}(i, \pi(i))$.
Notice that, for any $\pi\in \Pi$, $P^{\pi} \in \mathbb{S}^{n\times n}$.

Row-stochastic matrices are an important subset of non-negative matrices~\cite[Chapter 2]{nonnegative_matrices}, \textit{i.e.}, square matrices all of whose elements are non-negative. We now recall some fundamental definitions and properties of non-negative matrices that will be later used to produce a classification of discounted MDPs with finite spaces. We refer to~\cite[Chapter 2]{nonnegative_matrices} for a more thorough review on non-negative matrices as well as for the technical proofs of the following propositions.

\begin{proposition}\label{prop: rho_eigenvalue}
    If $A$ is non-negative, then $(\rho(A), 0)$ is an eigenvalue. 
\end{proposition}
We refer to the proof of statement (a) in Theorem 1.1 in~\cite[Chapter 2]{nonnegative_matrices} for a proof of Proposition~\ref{prop: rho_eigenvalue}. 

Since $\rho(A) = 1$ when $A$ is stochastic, from the previous proposition we can deduce that stochastic matrices always have $(1, 0)$ as an eigenvalue. In addition, by the Gershgorin circle theorem~\cite{gerschgorin_31}, the spectrum of stochastic matrices is contained in the circle of radius $1$ and center $(0,0)$ in the complex plane. Combining these two properties, we obtain that $\lambda_1 =(1, 0)$ (see Theorem 5.3 in~\cite[Chapter 2]{nonnegative_matrices} for a formal proof). 

The following definition characterizes the subset of irreducible matrices.
\begin{definition}[Irreducible Non-Negative Matrices]\label{def:irreducible_matrices}
A non-negative matrix $A$ is \textit{irreducible} if and only if for every $(i,j)$ there exists a natural number $q$ such that $a^{(q)}_{ij}>0$.
\end{definition}
As we can infer from the following proposition, irreducibility of a non-negative matrix has important implications on the distribution of its eigenvalues.
\begin{proposition}\label{prop: irreducible_matrix_simple}
    If $A$ is irreducible, then $\rho(A)$ is a simple eigenvalue and any other eigenvalue of $A$ with modulus $\rho(A)$ is also simple.
\end{proposition}
We refer to the proof of statement (b) in Theorem 1.4 in~\cite[Chapter 2]{nonnegative_matrices} for a proof of Proposition~\ref{prop: rhoA_greater1}. 

Within the set of irreducible matrices it is possible to define the subset of primitive matrices, which verify a more stringent condition than the one in Definition~\ref{def:irreducible_matrices}. We formalize this subset with the following definition.
\begin{definition}[Primitive Non-Negative Matrices]
An irreducible matrix $A$ is \textit{primitive} if and only if there exists a natural number $n_p$ such that $A^{n_p}$ is positive.
\end{definition}
From the previous definition, it is clear that every primitive matrix is also irreducible, but not viceversa. For instance, it is easy to verify that the following matrix is irreducible but not primitive 
\begin{equation}\label{eq:example_A}
    A = 
    \begin{bmatrix}
        0 & 1\\
        1 & 0
    \end{bmatrix}\,.
\end{equation}

The following proposition characterizes primitiveness of an irreducible matrix via a necessary and sufficient condition on its spectrum.
\begin{proposition}\label{prop: rhoA_greater1}
    $A$ is primitive if and only if $\rho(A)$ is greater in magnitude than any other eigenvalue.
\end{proposition}
We refer to the proof of statement (a) in Theorem 1.7 in~\cite[Chapter 2]{nonnegative_matrices} for a proof of Proposition~\ref{prop: rhoA_greater1}.

Irreducibility and primitiveness have an intuitive graph theoretical interpretation. In particular, denoting with $G(A)$ the graph associated with the non-negative matrix $A$, we can conclude that $A$ is irreducible if and only if $G(A)$ is strongly connected, \textit{i.e.}, there exists a finite sequence of edges which connects any two nodes in the graph. Primitiveness instead translates into the stronger requirement of reachability of any node from any other node after a sequence of exactly $n_p$ edges.

We conclude this brief recap on non-negative matrices and some of their fundamental properties with the following definition, which formalizes the \textit{index of cyclicity} of an irreducible matrix.
\begin{definition}[Index of Cyclicity]
    Let $A$ be an irreducible matrix. We call \textit{index of cyclicity} of $A$, and we denote it with $h(A)$, the number of eigenvalues of $A$ with modulus equal to $\rho(A)$.
\end{definition}
Going back to~\eqref{eq:example_A} as example, $h(A) = 2$ since $\Lambda(A) = \left\{ +1, \, -1 \right\}$.
In addition, notice that, because of Proposition~\ref{prop: rhoA_greater1}, $h(A) = 1$ if and only if $A$ is primitive.

As mentioned, row-stochastic matrices are an important subset of non-negative matrices. We can therefore deploy the structural properties of the latter to further characterize discounted MDPs with finite spaces. In particular, we refer to any discounted MDP with finite spaces as \textit{general} and distinguish two sub-classes, \textit{ergodic} and \textit{regular} discounted MDPs with finite spaces, based on the structural properties of $P^{\pi}$ for all $\pi \in \Pi$. We formalize this classification with the following definition.

\begin{definition}[Ergodic \& Regular MDPs]
    A discounted MDP with finite spaces is called \textit{ergodic} if and only if the transition probability matrix corresponding to every policy is \textit{irreducible}. If, in addition, the transition probability matrix corresponding to every policy is \textit{primitive}, then it is called \textit{regular}.
\end{definition}
A similar principle for the classification of finite MDPs is adopted by Puterman in~\cite[Chapter 8]{puterman_mdp}.

By relying on the discussed properties of irreducible matrices, we can conclude that, for ergodic MDPs, $P^{\pi}$ has $h(P^{\pi}) \geq 1$ simple eigenvalues with modulus 1 for all $\pi\in \Pi$, while all the other eigenvalues are contained strictly inside the unitary circle in the complex plane. In addition, $h(P^{\pi}) = 1$ for all $\pi\in \Pi$ if and only if the MDP is regular. To keep the notation compact, in the following we will use $h_{\pi}$ in place of $h(P^{\pi})$.
\subsection{Dynamic Programming}
Dynamic programming (DP) comprises the methods to solve the Bellman equations. In particular, from now on we refer to the solution of~\eqref{eq:Bellman_eq_pi} for a given policy $\pi\in\Pi$ as \textit{policy evaluation}. In general terms, we can recognize three fundamental DP methods: value iteration (VI), policy iteration (PI) and the linear programming approach~\cite[Chapter 2]{DB_book}. In this work we only focus on the first two methods and some of their variants. 

VI, also known as the \textit{successive approximation method}, is an iterative method to solve the Bellman equations. In particular, VI is based on repeated applications of the Bellman operators starting from an arbitrary cost-vector $V_0 \in \mathbb{R}^n$, where the operator in~\eqref{eq:Tpi_operator} is used to solve~\eqref{eq:Bellman_eq_pi} and the operator in~\eqref{eq:T_operator} to solve~\eqref{eq:Bellman_eq_pi*}. 
In the first case, we refer to it as VI \textit{for policy evaluation}. Algorithm~\ref{alg:VI} provides an algorithmic description of VI. The following proposition characterizes the convergence properties of VI and VI for policy evaluation when the infinity-norm is considered to measure the distance of the generated iterates from the solution.

\begin{algorithm}
\caption{Value Iteration}\label{alg:VI}
\begin{algorithmic}[1]
\Require $V_0 \in \mathbb{R}^n,\, \textit{tol}>0$
\State $k \leftarrow 0$
\State $V_1 \leftarrow TV_0$
\While {$\Vert V_{k+1} - V_k \Vert_{\infty} > tol$}
        \State $k \leftarrow k+1$
        \State $V_{k+1} \leftarrow TV_k$
\EndWhile
\end{algorithmic}
\end{algorithm}

\begin{proposition}\label{prop: convergence_VI}
    Consider a generic infinite-horizon discounted MDP with finite spaces and any policy $\pi\in \Pi$ and let $\left\{ V_k\right\}$ and $\left\{ V^{\pi}_k\right\}$ denote the sequence generated by VI and VI for policy evaluation, respectively. Then for any $k\geq 0$
    \vspace{-0.9cm}
    \begin{multicols}{2}
      \begin{equation}
        \Vert V_{k+1} - V^* \Vert_{\infty} \leq \gamma\Vert V_{k} - V^* \Vert_{\infty}
      \end{equation}\break
      \begin{equation}
        \Vert V^{\pi}_{k+1} - V^{\pi} \Vert_{\infty} \leq \gamma\Vert V^{\pi}_{k} - V^{\pi} \Vert_{\infty}
      \end{equation}
    \end{multicols}
    \vspace{-1cm}

\end{proposition}
\begin{proof}
We report the proof for VI. Analogous steps can be done to prove the result for VI for policy evaluation.

We consider an arbitrary starting point $V_0 \in \mathbb{R}^n$ and we prove the final result by upper-bounding the infinity-norm of the distance between the $(k+1)$-iterate and the solution as follow
    \begin{equation}
        \begin{aligned}
            \Vert V_{k+1} - V^* \Vert_{\infty}  &=  \Vert TV_k - TV^* \Vert_{\infty}\\
            &\leq\gamma\,\Vert V_k - V^* \Vert_{\infty} \quad k=0,1,\dots\,,
        \end{aligned}
    \end{equation}
    where the equality is obtained by using the definition of VI iterate and the fact that the $T$-operator has $V^*$ as unique fixed-point, and the inequality follows from Proposition~\ref{prop: gamma_contractivity}.
\end{proof}
Extensive numerical evidence shows that the results of Proposition~\ref{prop: convergence_VI} are tight, which translates in an exacerbating slow convergence for values of $\gamma$ close to 1.
By using monotonic error bounds in some cases it is possible to improve the convergence rate of VI to $\gamma \,\vert \lambda_2 \vert$, where $\lambda_2$ denotes the subdominant eigenvalue of $P^{\pi^*}$~\cite[Chapter 2]{DB_book}. Unfortunately, there are often situations where $\vert \lambda_2 \vert = 1$, \textit{e.g.}, ergodic but not regular MDPs with $h(P^{\pi^*})>1$, which results in VI with monotonic error bounds having the same contraction as the original VI.  

On a final note, VI only produces a sequence of cost-iterates, but it is possible to produce a sequence of policies too by extracting at every iteration a greedy policy associated with the current cost-iterate.

Equation~\eqref{eq:Bellman_eq_pi*} can also be solved with PI, an iterative method that alternates two steps: policy evaluation and policy improvement. In particular, we start by extracting a greedy policy associated with an arbitrary cost $\tilde{V}\in\mathbb{R}^n$ and, until convergence, we first evaluate the cost associated with the current policy (\textit{policy evaluation step}) and use this value to update the cost-iterate, and then we update the current policy with a greedy policy associated with the current cost-iterate (\textit{policy improvement step}). See Algorithm~\ref{alg:PI} for an algorithmic description of PI. Even if PI with exact arithmetic is guaranteed to converge in a finite-number of iterations, the upper bound on the number of iterations is exponential in the number of states and it is therefore important to characterize its improvement-per-iteration. The following proposition characterizes the convergence properties of PI.
\begin{proposition}\label{prop: PI_convergence}
    Consider a generic infinite-horizon discounted MDP with finite spaces and let $\left\{ V_k\right\}$ denote the sequence of iterates generated by PI. Then for any $k\geq 0$
    \begin{equation}\label{eq: PI_improvement_per_iteration}
        \Vert V_{k+1} - V^* \Vert_{\infty} \leq \gamma\Vert V_k - V^* \Vert_{\infty}\,.
    \end{equation}
    In addition, PI with exact arithmetic converges to the optimal policy in at most $m^n$ iterations.
\end{proposition}
\begin{proof}
    We start by proving the second statement. By exploiting the definition of greedy policy and the properties of the Bellman operators, it is possible to show that $V_{k+1} \leq V_k$ for $k=0,1,\dots$. For a step-by-step proof we refer to ~\cite[Proposition 2.3.1, Chapter 2]{DB_book}. The final result follows directly by considering that there only exists $m^n$ policies.
    
    Since $V_{k+1} \leq V_k$ for $k=0,1,\dots$, by exploiting the monotonicity of $T_{\pi_{k+1}}$ and the fact that $V_{k+1}$ is its unique fixed-point, we obtain the following upper-bound on $V_{k+1}$
    \begin{equation}
        \begin{aligned}
            V_{k+1} &= T_{\pi_{k+1}} V_{k+1}
            \\
            & \leq T_{\pi_{k+1}} V_{k}\\
            &= T V_k\,,
        \end{aligned}
    \end{equation}
    where the last equality follows from the fact that $\pi_{k+1}$ is a greedy policy associated to $V_k$, therefore $T_{\pi_{k+1}}V_k = T V_k$. Considering the previous derivations and since, by definition, $V^*\leq V^{\pi}$ for any $\pi\in\Pi$, we obtain that
    \begin{equation}\label{eq: positive_difference}
        0 \leq V_{k+1} - V^* \leq TV_{k} - TV^* \,.
    \end{equation}
    Finally, by taking the infinity-norm on both sides of~\eqref{eq: positive_difference} and considering that $V^* = TV^*$, we obtain the final result
    \begin{equation}
        \Vert V_{k+1} - V^* \Vert_{\infty} \leq \Vert TV_k  - TV^* \Vert_{\infty} \leq \gamma \, \Vert V_k - V^*\Vert_{\infty}\,,
    \end{equation}
    where the last inequality follows from Proposition~\ref{prop: gamma_contractivity}.
\end{proof}

\begin{algorithm}
\caption{Policy Iteration}\label{alg:PI}
\begin{algorithmic}[1]
\Require $\tilde{V} \in \mathbb{R}^n,\, \textit{tol}>0$
\State $k \leftarrow 0$
\State $\pi_0 \leftarrow \text{GreedyPolicy}(\tilde{V})$
\State $V_0 \leftarrow V^{\pi_0}$
\State $\pi_1 \leftarrow \text{GreedyPolicy}(V_0)$
\State $V_1 \leftarrow V^{\pi_1}$
\While {$\Vert V_{k+1} - V_{k} \Vert_{\infty} > tol$}
        \State $k \leftarrow k+1$
        \State $\pi_{k+1} \leftarrow \text{GreedyPolicy}(V_{k})$
        \State $V_{k+1} \leftarrow V^{\pi_{k+1}}$
\EndWhile
\end{algorithmic}
\end{algorithm}
Despite the results in Proposition~\ref{prop: PI_convergence}, Figure~\ref{fig:PIvsVI_iterations} and extensive empirical observations suggest that PI enjoys a qualitatively better rate of convergence than VI. These observations are also confirmed by the Newton-based theoretical analysis of the local contraction properties of DP methods, which is discussed in Section~\ref{sec3}. This leads to conclude that, at least locally, the results of Proposition~\ref{prop: PI_convergence} are loose and not able to describe properly the practical convergence behaviour of PI-iterates. 

Regarding the computational complextiy of PI, the policy evaluation (step 3, 5 and 9 in Algorithm~\ref{alg:PI}) comprises the exact solution of an $n$-dimensional linear system. In particular, given a policy $\pi\in\Pi$, $V^{\pi}$ is obtained by solving exactly the following linear system
\begin{equation}\label{eq: linear_system_PE}
    \left( I - \gamma P^{\pi} \right) V^{\pi} = g^{\pi}\,.
\end{equation}
In general, when the number of states is large, the exact solution of~\eqref{eq: linear_system_PE} is computationally expensive, \textit{e.g.}, $\frac{2}{3} n^3 + \mathcal{O}(n^2)$ with LU-factorization. Despite its fast rate of convergence, this results in an overall poor computational scalability of PI, as also confirmed by the results reported in Table~\ref{table: PIvsVI}.

\begin{figure}
    \centering
    \includegraphics[scale=0.6]{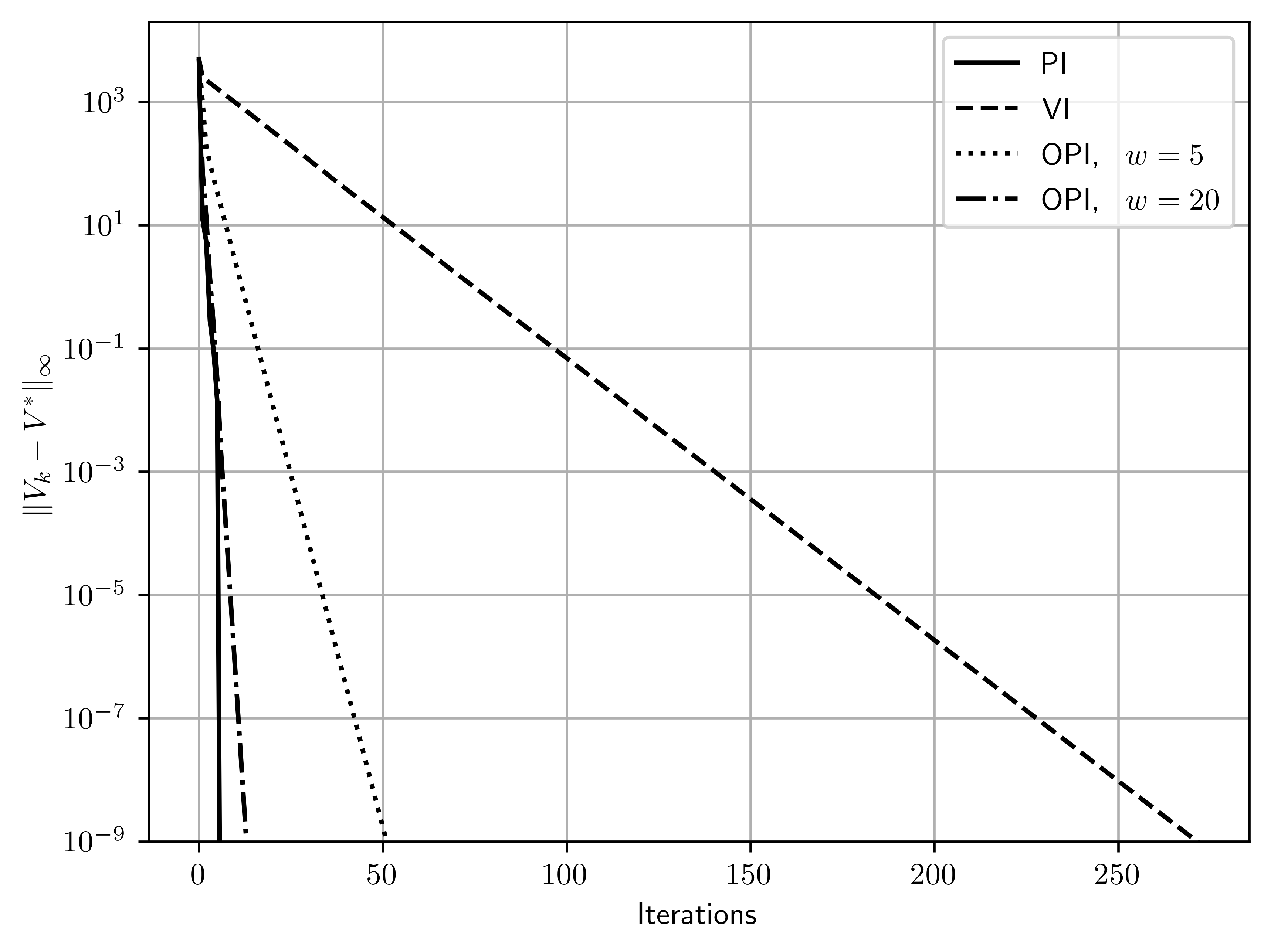}
    \caption{We consider an infinite-horizon discounted MDP with $\gamma=0.9$, $n=10000$ and $m=40$ and we plot the distance in infinity-norm of the iterates from the solution versus number of iterations for policy iteration, value iteration and optimistic policy iteration with different values of inner iterations.}
    \label{fig:PIvsVI_iterations}
\end{figure}

A better trade-off between computational complexity and convergence rate is achieved by \textit{optimistic policy iteration} (OPI)~\cite[Chapter 2]{DB_book}, also known in the literature as \textit{modified policy iteration}~\cite[Chapter 7]{puterman_mdp}. This variant of PI consists in substituting the policy evaluation step with the inexact solution of~\eqref{eq: linear_system_PE} obtained with a finite number $w \in \mathbb{N}$ of iterations of VI for policy evaluation starting from the current iterate. This results in the following cost-iterates 
\begin{equation}
    V_{k+1} = T_{\pi_{k+1}}^w V_k\quad k=0,1,\dots\,,
\end{equation}
\vspace{0.05cm}

\noindent where $\pi_{k+1} = \text{GreedyPolicy}(V_k)$ and $V_0\in \mathbb{R}^n$ is the initial iterate.
Clearly, when $w=1$ we recover VI and when $w = \infty$ we recover PI. In practice, this value is adjusted heuristically in order to produce the best trade-off between computational complexity and convergence rate for the specific problem at hand. Indeed, the more inner iterations are run and the faster is the resulting convergence in terms of iterations (see Figure~\ref{fig:PIvsVI_iterations}), but the more computationally expensive is the overall method. The convergence properties of OPI are discussed in the following proposition. 
\begin{proposition}\label{prop: convergence_OPI}
    Consider a general infinite-horizon discounted MDP with finite spaces and let $\left\{ V_k\right\}$ denote the sequence of iterates generated by OPI with $V_0\in\mathbb{R}^n$.
    Then for any $k\geq 0$
    \begin{equation}
        \Vert V_k - V^*  \Vert_{\infty} \leq \gamma^k\Vert V_0 - V^* \Vert_{\infty} + \gamma^k b\,.
    \end{equation}
    In addition, if $TV_0 \leq V_0$, then for any $k\geq 0$
    \begin{equation}
        \Vert V_{k+1} - V^*  \Vert_{\infty} \leq \gamma \Vert V_k - V^* \Vert_{\infty}\,.
    \end{equation}
\end{proposition}
\vspace{-0.5cm}
\begin{proof}
    We start by defining the auxiliary bar-iterates, where $\bar{V}_0 = V_0 + be$ and $b \in \mathbb{R}_{+}$. In particular, we set $b=0$ if $TV_0 \leq V_0$ and $b = \frac{1}{(1-\gamma)}\,\max_{i \in \mathcal{S}}\, (TV_0 - V_0)(i)$ otherwise. From the definition of $b$ it follows directly that $T\bar{V}_0 \leq \bar{V}_0$.
    Indeed, if $b=0$ then the bar-quantities coincide with the original iterates and $TV_0 \leq V_0$ by assumption. If $b>0$, then $TV_0 - V_0 \leq (1-\gamma) be$. By adding $\gamma be$ on both sides of the previous inequality and by exploiting Proposition~\ref{prop: shift_invariance} and the definition of the bar-quantities, we obtain that $TV_0 + \gamma be = T\bar{V}_0 \leq V_0 + be = \bar{V}_0$.
    Since for any $k \geq 0$ $\bar{V}_{k} = V_k + \gamma^{kw} be$, the bar-quantities produce the same sequence of greedy-policies $\left\{ \pi_k\right\}$ as the original iterates. We can therefore proceed with studying the convergence in terms of the bar-iterates. 

    By following the same steps as in the proof of Proposition 2.3.2 in~\cite[Chapter 2]{DB_book}, we are able to conclude that the following inequalities hold for any $k\geq 0$
    \vspace{-0.9cm}
    \begin{multicols}{2}
      \begin{equation}\label{eq:opi_ineq0}
      \bar{V}_{k+1} \leq T\bar{V}_k
      \end{equation}\break
      \begin{equation}
      \label{eq:opi_inequalities}
        V^* \leq \bar{V}_{k} \leq T^k \bar{V}_0\,.
      \end{equation}
    \end{multicols}
    
By using the fact that $\bar{V}_k = V_k + \gamma^{kw}be$ and re-arranging the terms in~\eqref{eq:opi_inequalities}, we obtain that for any $k\geq 0$ the following inequalities hold
\begin{equation}
    \begin{aligned}
    -\gamma^{kw}be  &\leq V_k - V^* \leq T^k V_0 - V^* + \gamma^kbe-\gamma^{kw}be\,. 
    \end{aligned}
\end{equation}
Since $\gamma^k be \geq \gamma^{kw}be\geq 0$ for any $w\in\mathbb{N}$ and by exploiting the fact that $V^*$ is the unique fixed-point of the $T$-operator, we obtain that the following inequalities hold for any $k\geq 0$ 
\begin{equation}\label{eq: sandwitch_OPI_step1}
    \begin{aligned}
    -\gamma^k\Vert V_0 - V^* \Vert_{\infty} e-\gamma^{k}be  &\leq V_k - V^* \leq T^k V_0 - T^k V^* + \gamma^k be\,. 
    \end{aligned}
\end{equation}
Starting from~\eqref{eq: sandwitch_OPI_step1} and exploiting the fact that $T^k V_0 - T^k V^* \leq \Vert T^k V_0 - T^k V^* \Vert_{\infty}e \leq \gamma^k \Vert V_k - V^*\Vert_{\infty}e$, where the last inequality follows from Proposition~\ref{prop: gamma_contractivity}, and since $\Vert V_0 - V^*\Vert_{\infty} \geq 0$, we obtain that the following inequalities hold for any $k\geq 0$
\begin{equation}\label{eq: sandwitch_OPI}
    \begin{aligned}
    -\gamma^k\Vert V_0 - V^* \Vert_{\infty} e-\gamma^{k}be  &\leq V_k - V^* \leq \gamma^k\Vert V_0 - V^* \Vert_{\infty} e + \gamma^k be\,. 
    \end{aligned}
\end{equation}
From~\eqref{eq: sandwitch_OPI}, we can conclude the final result
\begin{equation}
    \Vert V_k - V^* \Vert_{\infty} \leq \gamma^k\Vert V_0 - V^* \Vert_{\infty}  + \gamma^k b \,.
\end{equation}
In addition, if $b=0$, then, from~\eqref{eq:opi_ineq0} and~\eqref{eq:opi_inequalities}, we can conclude that for any $k\geq 0$ the following holds
\begin{equation}
    0 \leq V_{k+1} - V^* \leq TV_{k} - V^* = TV_{k} - TV^*\,,
\end{equation}
where the equality follows from the fact that $V^*$ is the unique fixed-point of the $T$-operator.
By switching to the infinity-norm and exploiting Proposition~\ref{prop: gamma_contractivity}, we obtain the final result
\begin{equation}
    \Vert V_{k+1} - V^* \Vert_{\infty} \leq \gamma \Vert V_{k} - V^* \Vert_{\infty}\,.
\end{equation}
\end{proof}
Results on the convergence of OPI can also be found in~\cite[Proposition 2.3.2]{DB_book}. Our results give also an upper-bound on the improvement-per-iteration.
Unfortunately, Proposition~\ref{prop: convergence_OPI} as well as Proposition 2.3.2 in~\cite{DB_book} fail in capturing the improvement in terms of convergence rate that is observed when increasing the number of inner iterations $w$, and therefore the level of accuracy that is used to solve the policy evaluation tasks.

\begin{table}[h]
\begin{center}
\begin{minipage}{\textwidth}
\caption{We consider an infinite-horizon discounted MDPs with $\gamma=0.9$, $m=40$ and different sizes of the state space and we report the total number of iterations to reach an accuracy of $10^{-9}$ and the cost-per-iteration in terms of CPU time for policy iteration and value iteration. The benchmarks are run on Intel(R) Core(TM) i7-10750H CPU @ 2.60GHz
enforcing single core execution.}\label{table: PIvsVI}
\begin{tabular*}{\textwidth}{@{\extracolsep{\fill}}lcccccc@{\extracolsep{\fill}}}
\toprule%
& \multicolumn{2}{@{}c@{}}{Policy Iteration} & \multicolumn{2}{@{}c@{}}{Value Iteration} \\\cmidrule{2-3}\cmidrule{4-5}%
$\gamma=0.9$ & Iterations & CPU Time & Iterations & CPU Time \\
\midrule
$n=100$    & 6  & $3.67 \times 10^{-4}$  & 270 & $1.96 \times 10^{-4}$  \\
$n=1000$   & 6  & $4.55 \times 10^{-2}$   & 271  & $1.81 \times 10^{-2}$\\
$n=10000$  & 7  & $17.07$                  & 273  & $1.84$ \\
\botrule
\end{tabular*}
\end{minipage}
\end{center}
\end{table}

\section{The Bellman Residual Function}\label{sec3}

Solving the Bellman equation~\eqref{eq:Bellman_eq_pi*} can be recast as a nonlinear non-smooth root-finding problem, where the residual function is the so-called \textit{Bellman residual function} $r:\mathbb{R}^n \rightarrow \mathbb{R}^n$, with $r(V) = V - TV$~\cite{gargiani_2022}. Because of the relation between $T$ and $T_{\pi}$, we can also rewrite the Bellman residual function as follows
\begin{equation}\label{eq: bellman_residual_function}
    r(V) = V - \min_{\pi \in \Pi}\left\{ T_{\pi}V \right\} = V - \min_{\pi\in \Pi} \left\{ g^{\pi} + \gamma P^{\pi}V\right\}\,.
\end{equation}
The function in~\eqref{eq: bellman_residual_function} is piecewise affine with at most $\vert \Pi \vert$ selection functions. In particular, we denote with $\tilde{\Pi}\subseteq \Pi$ the non-empty subset of policies that are associated to the selection functions of $r$, \textit{i.e.} for any $\pi \in \tilde{\Pi}$ there exists $V\in \mathbb{R}^n$ such that $r(V) = V - T_{\pi}V$, and we use $r^{\pi}(V) = V - T_{\pi}V$ for any $\pi\in\Pi$. For a general piecewise affine function, it is always possible to add extra \textit{spurious} selection functions to the minimal set of effective ones~\cite[Chapter 4]{facchinei_bookI}. The following assumption excludes the presence of spurious selection functions by requiring all the policies in $\tilde{\Pi}$ to be associated to effective selection functions.
\begin{assumption}\label{assumption: spurious_assumption}
Consider a discounted MDP with finite spaces.
Let $\tilde{\Pi}\subseteq \Pi$ denote the non-empty subset of policies associated to the selection functions of the Bellman residual function. We assume that for any $\pi \in \tilde{\Pi}$ then
\begin{equation}
    \text{int} \left(\{ V \in \mathbb{R}^n \, \, \, \text{s.t.} \, \, \, r(V) = V - T_{\pi}V \} \right) \neq \emptyset \,.
\end{equation}
\end{assumption}
Under Assumption~\ref{assumption: spurious_assumption}, the generalized Jacobian of~\eqref{eq: bellman_residual_function} at $V\in \mathbb{R}^n$ is given by the following set
\begin{equation}
    \partial r(V) = \text{Conv}\left(\{ I - \gamma P^{\pi}, \,\text{where }\pi\in \text{GreedyPolicy}(V)\}\right)\,.
\end{equation}
Since for any $V\in\mathbb{R}^n$ every $J\in \partial r(V)$ is non-singular,~\eqref{eq: bellman_residual_function} is globally CD-regular~\cite[Proposition 2]{gargiani_2022}. In addition, as shown in~\cite[Lemma 3]{gargiani_2023}, for any $\pi \in \Pi$  
\begin{equation}
\Vert (I - \gamma P^{\pi})^{-1}\Vert_{\infty} \leq 1/(1 - \gamma)\,. 
\end{equation}
Because of its piecewise affine structure, the Bellman residual function is globally $(1+\gamma)$-Lipschitz continuous in the infinity-norm~\cite[Lemma 2]{gargiani_2023} and strongly semismooth everywhere~\cite[Proposition 7.4.7]{facchinei_bookII}. In particular, we are interested in strong semismoothness at the root $V^*$.
\begin{lemma}\label{lemma: strong_semismoothness}
Consider a general infinite-horizon discounted MDP with finite spaces and let $r:\mathbb{R}^n\rightarrow \mathbb{R}^n$ denote its associated Bellman residual function as defined in~\eqref{eq: bellman_residual_function}. Then there exists $\delta > 0$ such that, for any $C \geq  \frac{2R}{(1-\gamma)\delta^2}$, the following inequality holds for all $V\in\mathbb{R}^n$ 
\begin{equation}\label{eq: strong_semismoothenss}
    \Vert r(V) - r(V^*) - \left( I - \gamma P^{\pi}\right)\left( V - V^* \right) \Vert_{\infty} \leq C\,\Vert V - V^* \Vert^2_{\infty}\,,
\end{equation}
where $\pi \in \text{GreedyPolicy}(V)$. We call $C_{\min} = \frac{2R}{(1-\gamma)\delta^2}$ the global strong semismoothness constant of $r$ at $V^*$.
\end{lemma}
\begin{proof}
We start the proof by upper-bounding the LHS in~\eqref{eq: strong_semismoothenss}
\begin{equation}\label{eq: strong_semismoothnessLHS}
    \begin{aligned}
    \Vert r(V) - r(V^*) - (I-\gamma P^{\pi})(V- V^*) \Vert_{\infty} &= \Vert (I - \gamma P^{\pi})V^* - g^{\pi}  \Vert_{\infty} \\ 
    &\leq (1+\gamma)\Vert V^* \Vert_{\infty} + R \\
    & \leq \frac{2}{1-\gamma}R\,,
    \end{aligned}
\end{equation}
where the last inequality follows from the definition of $V^*$ and the limit of the geometric series.
We start by considering the case where $V$ is ``close-enough'' to $V^*$. In particular, there exists $\delta>0$ such that, for all $V\in\mathcal{B}(V^*, \delta)$, then $\pi$ is an optimal policy, since by continuity the selection functions that are active at points in this neighborhood are a subset of the selection functions active at the root (a similar argument is also used in the proof of Theorem 7.2.15 in~\cite{facchinei_bookII}). Consequently, whenever $V \in \mathcal{B}(V^*, \delta)$, then the LHS in~\eqref{eq: strong_semismoothnessLHS} is equal to zero and~\eqref{eq: strong_semismoothenss} trivially holds for any $C\geq 0$. Otherwise, we can always select $C \geq  2R/((1-\gamma)\delta^2)$ and~\eqref{eq: strong_semismoothenss} is verified since $\Vert V - V^* \Vert_{\infty} \geq \delta$.  
\end{proof}

\subsection{Semismooth Newton-Type Methods for DP}
The semismooth adaptation of Newton's method and its variants constitute the main class of algorithms to solve general nonlinear non-smooth root finding problems. If the residual function enjoys certain structural properties, \textit{i.e.}, Lipschitz continuity and strong semismoothness in a neighborhood of the root, the semismooth variant of Newton's method maintains the fast local quadratic convergence typical of Newton's method~\cite[Chapter 2]{izmailov_book}. 

Instead of solving the original problem directly, semismooth Newton method starts from an initial guess $V_0\in \mathbb{R}^n$  and, for every $k\geq 0$, solves the following linear approximation of~\eqref{eq: bellman_residual_function} at the current iterate $V_{k}$
\begin{equation}\label{eq:Newtonian_LS}
    r(V_k) + J_k (V_{k+1} - V_k) = 0\,,
\end{equation}
where $J_k \in \partial r(V_k)$.
The iterate is then updated with the derived solution of~\eqref{eq:Newtonian_LS} as follows 
\begin{equation}\label{eq:Newton_update}
    V_{k+1} = V_k - J_k^{-1}r(V_k)\,.
\end{equation}
As proved in~\cite{gargiani_2022}, under Assumption~\ref{assumption: spurious_assumption} policy iteration is an instance of semismooth Newton method. By selecting $J_k  = I - \gamma P^{\pi_{k+1}}$, with $\pi_{k+1} = \text{GreedyPolicy}(V_k)$, and plugging these quantities back into~\eqref{eq:Newton_update}, we obtain the following expression for the update of semismooth Newton method
\begin{equation}
    V_{k+1} = (I-\gamma P^{\pi_{k+1}})^{-1}g^{\pi_{k+1}}\,,
\end{equation}
which is equivalent to the policy iteration cost update. In light of this connection and because of the structural properties of the Bellman residual function, it is possible to conclude local quadratic contraction for policy iteration~\cite[Proposition 3.5]{gargiani_2022}. 

The semismooth Newton-type variants are obtained by replacing $J_k \in \partial r(V_k)$ in~\eqref{eq:Newtonian_LS} with a non-singular matrix $B_k \in \mathbb{R}^{n\times n}$. In our setting, the main reason behind this approximation of the original method are the computational costs. Indeed, whenever $B_k$ is selected with a favorable sparsity structure, factorization and inversion operations are generally less expensive, leading to lower computational costs for the iteration update. In addition, whenever $B_k$ approximates ``well-enough'' an element in $\partial r(V_k)$ for all $k\geq 0$, then the method preserves local linear contraction guarantees~\cite{gargiani_2022}. As shown in~\cite{gargiani_2022}, VI can be interpreted as a semismooth Newton-type method, where $B_k = I$ for all $k\geq 0$. A simple variant of VI, called $\alpha$-VI, is obtained by approximating $J_k\in \partial r(V_k)$ with a scaled identity matrix for all $k\geq 0$, where $\alpha>0$ is the scaling parameter and plays a fundamental role in determining the convergence properties of the method~\cite{gargiani_2022}.

The following proposition summarizes the local convergence properties of semismooth Newton-type methods when used to find the root of the Bellman residual function. 

\begin{theorem}\label{th: semismooth_newton_convergence}
    Consider a general infinite-horizon discounted MDP with finite spaces and let $r:\mathbb{R}^n \rightarrow \mathbb{R}^n$ denote its associated Bellman residual function as defined in~\eqref{eq: bellman_residual_function}. In addition, let $V_0\in\mathbb{R}^n$, $L>0$ and $\kappa \in [0,1)$ be a constant. For any sequence of nonsingular matrices $\left\{ B_k \right\} \subseteq \mathbb{R}^{n\times n}$, such that, for all $k\geq 0$, $\Vert B_k^{-1} \Vert_{\infty} \leq L$ and $\exists\,\, J_k \in \partial r(V_k)$ for which the kappa-condition
    \begin{equation}\label{eq:kappa_condition}
        \Vert B_k^{-1}(B_k - J_k)  \Vert_{\infty} \leq \kappa_k \leq \kappa
    \end{equation}
     is verified, the sequence $\left\{ V_k\right\}\subseteq \mathbb{R}^n$, with
     \begin{equation}\label{eq:NewtonType_iteration}
         V_{k+1} = V_k - B_k^{-1}r(V_k)\,,
     \end{equation}
     verifies the following inequality
     \begin{equation}\label{eq: main_inequality_newtontype}
         \Vert V_{k+1} - V^* \Vert_{\infty} \leq \kappa_k \Vert V_k - V^* \Vert_{\infty} + L C_{\min}\,\Vert V_k - V^*\Vert^2_{\infty},
     \end{equation}
     where $C_{\min}$ is the global strong semismooth constant of $r$ at $V^*$.
     Therefore, if $V_0\in\mathcal{B}\left(V^*, \frac{1-\kappa}{LC_{\min}}\right)$, the sequence of iterates generated by~\eqref{eq:NewtonType_iteration} is Q-linearly converging to $V^*$ in the infinity-norm with rate at least $\kappa + LC_{\min}\, \Vert V_0 - V^*\Vert_{\infty}$. 
\end{theorem}
\begin{proof}
A proof of~\eqref{eq: main_inequality_newtontype} can be derived by following similar steps as in~\cite[Theorem 2.6]{gargiani_extended_2022}.

We derive the first result by upper-bounding the infinity-norm of the difference between the $(k+1)$-iterate and the root $V^*$ as follow
\begin{equation*}
    \begin{aligned}
        &\Vert V_{k+1} - V^* \Vert_{\infty} = \Vert V_k - B_{k}^{-1} r(V_k) - V^* \Vert_{\infty}\\
        &= \Vert B_k^{-1}\left( B_k - J_k \right) \left( V_k - V^* \right) - B_k^{-1}\left( r(V_k) - r(V^*) - J_k\left( V_k - V^* \right) \right) \Vert_{\infty}\\
        &\overset{(a)}{\leq} \Vert B_k^{-1}\left( B_k - J_k \right)\Vert_{\infty} \Vert V_k - V^* \Vert_{\infty} + \Vert B_k^{-1}\Vert_{\infty} \Vert  r(V_k) - r(V^*) - J_k\left( V_k - V^* \right)  \Vert_{\infty}\\
        &\overset{(b)}{\leq} \kappa_k \Vert V_k - V^* \Vert_{\infty} + LC_{\min}\, \Vert   V_k - V^* \Vert^2_{\infty} \,,
    \end{aligned}
\end{equation*}
where inequality $(a)$ follows from the fundamental properties of the infinity-norm and inequality $(b)$ from the assumptions on the sequence $\left\{ B_k\right\}$ and Lemma~\ref{lemma: strong_semismoothness}.
From the previous derivations we can conclude that
\begin{equation}\label{eq:inequality_contraction}
    \Vert V_{k+1} - V^* \Vert_{\infty} \leq \left( \kappa_k + LC_{\min} \Vert V_{k} - V^*\Vert_{\infty}\right)\Vert  V_{k} - V^*\Vert_{\infty}\,.
\end{equation}
Finally, we can use~\eqref{eq:inequality_contraction} to obtain a sufficient condition on the initial iterate $V_0$ such that the sequence $\left\{ V_k\right\}$ enjoys local Q-linear convergence to $V^*$. In particular, starting from~\eqref{eq:inequality_contraction}, we obtain that if
\begin{equation}
    \Vert V_0 - V^* \Vert_{\infty}   < \frac{1-\kappa}{LC_{\min}}\,,
\end{equation}
then the iterates are Q-linearly contracting to $V^*$ in the infinity-norm with rate at least $\kappa + LC_{\min}\, \Vert V_0 - V^* \Vert_{\infty}$.
\end{proof}

Notice that for exact semismooth Newton method for all $k\geq 0$ $\kappa_k = 0$ since $B_k \in \partial r(V_k)$. Consequently, the generated sequence of iterates enjoys local Q-quadratic convergence to $V^*$ in the infinity-norm as~\eqref{eq: main_inequality_newtontype} reduces to 
\begin{equation}
    \Vert V_{k+1} - V^* \Vert_{\infty} \leq LC_{\min}\, \Vert V_k - V^* \Vert^2_{\infty}\,.
\end{equation}
This consideration, together with the fact that, under Assumption~\ref{assumption: spurious_assumption}, PI is an instance of semismooth Newton method, leads to the results discussed in the following corollary.
\begin{corollary}
Consider a general infinite-horizon discounted MDP with finite spaces. Under Assumption~\ref{assumption: spurious_assumption}, the sequence $\left\{ V_k\right\}\subseteq \mathbb{R}^n$ generated by PI enjoys local Q-quadratic convergence to $V^*$.
\end{corollary}

\section{Inexact Policy Iteration Methods}\label{sec4}
Even if, as proved in Theorem~\ref{th: semismooth_newton_convergence}, the semismooth Newton method enjoys a fast
rate of convergence, computing the exact solution of~\eqref{eq:Newtonian_LS}
using a direct method can be expensive when the number
of unknowns is large. A more computationally efficient
solution in the large-scale case consists in solving~\eqref{eq:Newtonian_LS}
only approximately with some iterative solver for linear systems and
using a certain stopping rule. These are the principles
behind the inexact variants of semismooth Newton method~\cite{MARTINEZ1995127}. In particular, $V_{k+1}$
is no longer required to exactly solve~\eqref{eq:Newtonian_LS}, but only to satisfy
\begin{equation}\label{eq:stopping_cond}
    \Vert r(V_k) + J_{k}(V_{k+1} - V_k) \Vert \leq \alpha_k\Vert r(V_k) \Vert\,,
\end{equation}
for some $\alpha_k \in [0,1)$. The sequence $\left\{\alpha_k\right\}$ is called \textit{forcing
sequence} and it greatly affects both local convergence
properties and robustness of the method~\cite{MARTINEZ1995127}. Different iterative solvers for linear systems can be
used to approximately solve~\eqref{eq:Newtonian_LS}. Often
Krylov subspace methods, such as the generalized minimal
residual method (GMRES)~\cite{gmres}, are
deployed in large-scale scenarios.
Based on these observations, we define a novel variant of PI for large-scale scenarios,
which we call inexact policy iteration (iPI) methods. This class of methods is based on approximately solving the policy evaluation step with an iterative solver. The methods in this class start with an initial guess of the optimal cost $V_0\in\mathbb{R}^n$ and then at every iteration extract a greedy policy associated with the current iterate $V_k \in \mathbb{R}^n$, which is used
to compute an element in Clarke’s generalized Jacobian.
The next iterate $V_{k+1}\in\mathbb{R}^n$ is selected as an approximate solution of the Newtonian linear system
\begin{equation}\label{eq:PE_LS}
    (I - \gamma P^{\pi_{k+1}}) V = g^{\pi_{k+1}}
\end{equation}
which verifies the stopping condition in~\eqref{eq:stopping_cond} with the
infinity-norm. Because of the specific structure of the
Bellman residual function,~\eqref{eq:stopping_cond} simplifies to
\begin{equation}\label{eq:simplified_stopping_condition}
    \Vert g^{\pi_{k+1}} - (I - \gamma P^{\pi_{k+1}})V_{k+1} \Vert \leq \alpha_k \Vert g^{\pi_{k+1}} - (I - \gamma P^{\pi_{k+1}})V_{k} \Vert\,.
\end{equation}
In principle, any iterative solver for linear systems with
non-singular coefficient matrices can be used to generate
an approximate solution of~\eqref{eq:PE_LS}, such as VI, its mini-batch version~\cite{gargiani_minibatch} and GMRES. Notice
that, when VI is deployed as inner solver, we obtain a
variant of OPI where the number of inner iterations is not
selected a priori, but dictated by the stopping condition.
The parametric stopping condition allows for a dynamic selection of the number of inner iterations based on locally available information. Finally, as we discuss in Section~\ref{sec: iterative_LS} and~\ref{subsec: benchmarks}, deploying a different inner solver than VI for policy evaluation may lead to a superior convergence performance.

See Algorithm~\ref{alg:iPI} for a pseudocode description of a general iPI method with a constant $\alpha$-parameter and where $\text{IterativeSolver}(A, b, \tilde{x})$ is used to denote the operator that applies one iteration with a well-defined iterative solver for linear systems with non-singular coefficient matrices to the linear system $Ax = b$ with $A$ being non-singular and starting from $\tilde{x}$. 



\begin{algorithm}
\caption{Inexact Policy Iteration}\label{alg:iPI}
\begin{algorithmic}[1]
\Require $V_{0} \in \mathbb{R}^n,\, \textit{tol}>0, \alpha \in (0,1)$
\State $k \leftarrow 0$
\While {$\Vert r(V_{k})\Vert_{\infty} > tol$}
        \State $\pi_{k+1} \leftarrow \text{GreedyPolicy}(V_{k})$
        \State $J_k \leftarrow I - \gamma P^{\pi_{k+1}}$
        \State $\theta_{0}\leftarrow V_k$
        \State $i\leftarrow 0$
        \While{$\Vert g^{\pi_{k+1}} - J_k {\theta}_{i} \Vert_{\infty} \leq \alpha\Vert g^{\pi_{k+1}} - J_k V_{k}\Vert_{\infty}$
        }

        \State ${\theta}_{i+1}\leftarrow \text{IterativeSolver}(J_k, g^{\pi_{k+1}}, {\theta}_{i})$
        \State $i \leftarrow i+1$
        \EndWhile
        \State $V_{k+1}\leftarrow \theta_i$
        \State $k \leftarrow k+1$
\EndWhile
\end{algorithmic}
\end{algorithm}

\subsection{Local and Global Convergence}

This section is dedicated to the analysis of the local and global convergence properties of inexact policy iteration methods. Local results were originally studied in~\cite{gargiani_2023}, but we will report them for completeness. 
On that side, we complement the analysis by adding an estimate for the region of attraction. The latter is later exploited in conjunction with additional results to show global convergence of a general iPI method.

\begin{theorem}[local convergence]\label{th: iPI_localconvergence}
Consider a general infinite-horizon discounted MDP with finite spaces and let $\left\{ V_k \right\}$ be the sequence of iterates generated by a general inexact policy iteration method as described in Algorithm~\ref{alg:iPI} with $V_0 \in \mathbb{R}^n$. In addition, let $\omega = \frac{(1-\gamma) - (1+\gamma)\alpha}{C_{\min}}$. Under Assumption~\ref{assumption: spurious_assumption}, if $\left\{ \alpha_k\right\} \subseteq [0,\,\alpha]$ with $\alpha \in (0, \frac{1-\gamma}{1+\gamma})$ and $V_0 \in \mathcal{B}(V^*, \omega)$, then  $\left\{ \Vert V_k - V^* \Vert_{\infty}\right\}$ is Q-linearly converging to zero with rate $\frac{1+\gamma}{1-\gamma}\alpha$. If, in addition, $\lim_{k\rightarrow \infty} \alpha_k = 0$, then $\left\{ \Vert V_k - V^* \Vert_{\infty}\right\}$ is converging to zero Q-superlinearly.
\end{theorem}

\begin{proof}
We refer to the proof of Theorem 4 in~\cite{gargiani_2023} for the derivation of the following inequality
\begin{equation}\label{eq: ineq_TH4}
\Vert V_{k+1} - V^* \Vert_{\infty} \leq \frac{1}{1-\gamma}\left[\Vert r(V_k) - r(V^*) - J_k (V_k - V^*) \Vert_{\infty} \!+ (1+\gamma)\alpha_k \Vert V_k - V^*\Vert_{\infty} \!\right]\!.
\end{equation}
Making use of Lemma~\ref{lemma: strong_semismoothness}, we can upper-bound the first term in the RHS of~\eqref{eq: ineq_TH4} and obtain the following inequality
\begin{equation}\label{eq: ineq_0delta}
\begin{aligned}
\Vert V_{k+1} - V^* \Vert_{\infty} &\leq \frac{1}{1-\gamma}\left[C_{\min} \Vert V_k - V^* \Vert^2_{\infty} + (1+\gamma)\alpha_k \Vert V_k - V^*\Vert_{\infty} \right]\\
&=  \left[\frac{C_{\min}}{1-\gamma} \Vert V_k - V^* \Vert_{\infty} + \frac{(1+\gamma)}{(1-\gamma)}\alpha_k \right] \Vert V_k - V^* \Vert_{\infty}\,. \\
\end{aligned}
\end{equation}
Finally, since $\alpha_k \leq \alpha$, we obtain the following upper-bound
\begin{equation}\label{eq: ineq_delta}
\begin{aligned}
\Vert V_{k+1} - V^* \Vert_{\infty}&\leq \underbrace{\left[\frac{C_{\min}}{1-\gamma} \Vert V_k - V^* \Vert_{\infty} + \frac{(1+\gamma)}{(1-\gamma)}\alpha \right]}_{:=\delta_k} \Vert V_k - V^* \Vert_{\infty}\,. 
\end{aligned}
\end{equation}
Convergence of the sequence $\left\{ \Vert V_k - V^*\Vert_{\infty} \right\}$ to zero follows from the fact that $\Vert V_{1} - V^* \Vert_{\infty} \,\leq \,\delta_0\,\Vert V_{0} - V^* \Vert_{\infty}$ and $\delta_0 < 1$, which is an implication of the assumption that $V_0 \in \mathcal{B}(V^*, \omega)$. Starting from~\eqref{eq: ineq_delta} and using these inequalities recursively lead to conclude that
\begin{equation}\label{eq: linear_contraction}
\Vert V_{k} - V^*\Vert_{\infty} \leq \,\delta_0^k\,\Vert V_0 - V^*\Vert_{\infty}\,.
\end{equation}
This means that, under the considered assumptions, the sequence $\left\{ \Vert V_k - V^*\Vert_{\infty} \right\}$ converges to zero at least linearly with rate $\delta_0$. To obtain the local contraction rate we take the limit for $k\rightarrow \infty$ of~\eqref{eq: ineq_delta} as follow
\begin{equation}
\begin{aligned}
\lim_{k \rightarrow \infty} \frac{\Vert V_{k+1} - V^* \Vert_{\infty}}{\Vert V_{k} - V^* \Vert_{\infty}} = \lim_{k \rightarrow \infty} \frac{C_{\min}}{1 - \gamma} \Vert V_k - V^* \Vert_{\infty} + \frac{(1+\gamma)}{(1-\gamma)}\alpha 
= \frac{(1+\gamma)}{(1-\gamma)}\alpha < 1\,,
\end{aligned}
\end{equation}
where the second equality follows from~\eqref{eq: linear_contraction} and the strict inequality follows from the assumption that $\alpha \in (0, \frac{1-\gamma}{1+\gamma})$. If, in addition, $\lim_{k\rightarrow \infty} \alpha_k = 0$, then we can repeat similar steps starting from~\eqref{eq: ineq_0delta} and conclude local Q-superlinear convergence.
\end{proof}
While the convergence results of OPI (see for instance Proposition~\ref{prop: convergence_OPI}) fails in capturing the improvement in convergence rate that results from running more inner iterations, and, consequently, from solving with higher accuracy the policy evaluation step, the results of Theorem~\ref{th: iPI_localconvergence} show that smaller values of the $\alpha$-parameter, which translates into solving with a higher accuracy the policy evaluation step, lead to a better local convergence rate of the method.

While it is possible to show local contraction with relatively simple steps, proving global convergence of a general iPI method is not equally simple. Intuitively, from the convergence results of OPI in Proposition~\ref{prop: convergence_OPI} (or also Proposition  in~\cite{DB_book}), one may expect that, with no additional assumptions on the structural properties of the underlying MDP, for any $\gamma\in (0,1)$  there exists a non-empty set of positive $\alpha$ for which we can get global convergence guarantees for a general iPI method. As we discuss in the following, despite its intuitiveness, proving that is quite challenging.
Differently from the global convergence proof of OPI (see Proposition~\ref{prop: convergence_OPI}), we can not exploit the fundamental properties of the Bellman operators tout court as, depending on the specific selection, the inner solver may not be a contraction in the infinity-norm. 

We start the discussion on global convergence of iPI methods by illustrating possible technical complications. Then, by introducing additional assumptions on the structure of the underlying problem, we prove global convergence of a general iPI method. 

\begin{proposition}[]\label{prop: upper_bound}
Consider a general infinite-horizon discounted MDP with finite spaces and let $\left\{ V_k \right\}$ and $\left\{ \pi_{k+1} \right\}$ be the sequences generated by a general iPI method as described in Algorithm~\ref{alg:iPI}, where $V_0 \in \mathbb{R}^n$. Then for any $k\geq 0$
\begin{equation}\label{eq: ineq_iPI_wrong}
\Vert V_{k+1} - V^* \Vert_{\infty} \leq \frac{(1+\gamma)}{(1-\gamma)}\alpha \Vert V_k - V^* \Vert_{\infty} + \Vert V^{\pi_{k+1}} - V^* \Vert_{\infty}\,.
\end{equation}
\end{proposition}

\begin{proof}
In the following, we use $J_k = I - \gamma P^{\pi_{k+1}}$ and $\Delta V_k = V_{k+1} - V_k$. We start by algebraically manipulating $V_{k+1} - V^*$ and then we use the fundamental properties of the infinity-norm, the structural properties of the Bellman residual function and the stopping condition to obtain the final result as follow
\begin{equation*}\label{eq:inequality_globalconvergence}
\begin{aligned}
\Vert V_{k+1} - V^* \Vert_{\infty} &= \Vert V_{k} + \Delta V_k - V^* \Vert_{\infty}\\
&= \Vert V_{k} + \Delta V_k - V^{\pi_{k+1}} + V^{\pi_{k+1}}  - V^* \Vert_{\infty}\\
&\leq \Vert V_{k} + \Delta V_k  - V^{\pi_{k+1}} \Vert_{\infty} + \Vert V^{\pi_{k+1}}  - V^* \Vert_{\infty}\\
&= \Vert J_k^{-1}g^{\pi_{k+1}}\! - J_k^{-1}g^{\pi_{k+1}} + J_k^{-1}J_kV_{k} + J_k^{-1}J_k\Delta V_k  - J_k^{-1}J_kV^{\pi_{k+1}} \Vert_{\infty} \\
&\phantom{=\,} + \Vert V^{\pi_{k+1}}  - V^* \Vert_{\infty}\\
&= \Vert J_k^{-1}\left( g^{\pi_{k+1}}\! - J_k V^{\pi_{k+1}} \right) + J_k^{-1}\left( J_k V_k - g^{\pi_{k+1}} + J_k \Delta V_k \right) \Vert_{\infty}\\
&\phantom{=\,} + \Vert V^{\pi_{k+1}}  - V^* \Vert_{\infty}\\
&= \Vert J_k^{-1}\left( J_k V_k - g^{\pi_{k+1}} + J_k \Delta V_k \right) \Vert_{\infty} + \Vert V^{\pi_{k+1}}  - V^* \Vert_{\infty}\\
&\leq \Vert J_k^{-1} \Vert_{\infty} \Vert J_k V_k - g^{\pi_{k+1}} + J_k \Delta V_k \Vert_{\infty} + \Vert V^{\pi_{k+1}}  - V^* \Vert_{\infty}\\
&\overset{(a)}{\leq} \frac{1}{1-\gamma} \Vert J_k V_k - g^{\pi_{k+1}} + J_k \Delta V_k \Vert_{\infty} + \Vert V^{\pi_{k+1}}  - V^* \Vert_{\infty} \\
&\leq \frac{1}{1-\gamma}\alpha \Vert r(V_k) - r(V^*) \Vert_{\infty} + \Vert V^{\pi_{k+1}}  - V^* \Vert_{\infty} \\
&\overset{(b)}{\leq} \frac{1+\gamma}{1-\gamma}\alpha \Vert V_k - V^* \Vert_{\infty} + \Vert V^{\pi_{k+1}}  - V^* \Vert_{\infty}\,,  
\end{aligned}
\end{equation*}
where $(a)$ and $(b)$ are obtained by deploying Lemma 3 and Lemma 2 in~\cite{gargiani_2023}, respectively.
\end{proof}

Starting from~\eqref{eq: ineq_iPI_wrong} it is very tempting to proceed by deploying the convergence results from policy iteration and, in particular, Inequality~\eqref{eq: PI_improvement_per_iteration} in Proposition~\ref{prop: PI_convergence}, as $V^{\pi_{k+1}}$ is obtained by applying a full policy iteration step (policy improvement and policy evaluation) starting from $V_k$. That would allow one to conclude a sufficient condition on $\alpha$, similar to the one in Proposition~\ref{th: iPI_localconvergence}, and that guarantees global convergence of iPI methods.  
Unfortunately, without additional assumptions, for the following upper-bound to hold
\begin{equation}\label{eq:wrong_ineqPI}
\Vert V^{\pi_{k+1}} - V^* \Vert_{\infty} \leq \gamma\,\Vert V_{k} - V^* \Vert_{\infty}\,,
\end{equation}
$V_k$ would need to be the cost associated to a policy $\pi \in \Pi$. Because of the inexactness in the policy evaluation step, this condition is not necessarily verified, therefore we can not deploy~\eqref{eq:wrong_ineqPI}.

Without introducing additional assumptions, a tight upper-bound on $\Vert V^{\pi_{k+1}} - V^* \Vert_{\infty}$ is given by Proposition 2.3.3 in~\cite{DB_book}. In particular, by deploying
\begin{equation}\label{eq:2.3.3Bertsekas}
\Vert V^{\pi_{k+1}} - V^* \Vert_{\infty} \leq \frac{2\gamma}{1-\gamma}\Vert V_k - V^* \Vert_{\infty}\,
\end{equation} 
in~\eqref{eq: ineq_iPI_wrong}, we obtain the following upper-bound
\begin{equation}
\Vert V_{k+1} - V^* \Vert_{\infty} \leq \frac{(1+\gamma)\alpha + 2\gamma}{1 - \gamma} \Vert V_{k} - V^* \Vert_{\infty}
\end{equation}
which allows one to conclude global convergence if $\gamma < 1/3$ and $\alpha \in (0, (1 - 3\gamma)/(1 + \gamma))$. These results are unsatisfactory as they are only valid for a limited range of discount factor values. By exploiting also the contractivity of the Bellman operator, in the next proposition we provide a similar sufficient condition to guarantee global convergence of a general iPI method, but for a slightly wider range of discount factor values.
\begin{proposition}\label{prop: global_gamma}
Consider a general infinite-horizon discounted MDP with finite spaces and let $\left\{ V_k \right\}$ be the sequence of iterates generated by a general iPI method as described in Algorithm~\ref{alg:iPI}. If $\gamma \in (0, \, -1+\sqrt{2})$ and $\alpha \in (0, \, (1 - 2\gamma - \gamma^2)/(1 + \gamma))$, then $(\alpha + \gamma)(1 + \gamma) + \gamma<1$ and for any $k\geq 0$
\begin{equation}
    \Vert V_{k+1} - V^* \Vert_{\infty} \leq ((\alpha + \gamma)(1+\gamma) + \gamma)\Vert  V_k - V^*\Vert_{\infty}\,.
\end{equation}
\end{proposition}
\begin{proof}
We derive an upper-bound on $\Vert V_{k+1} - V^* \Vert_{\infty}$ by deploying the properties of the infinity-norm, the stopping condition and the structural properties of the Bellman residual function and the Bellman operators. This leads to the following upper-bound
\begin{equation}
\begin{aligned}
&\Vert V_{k+1} - V^* \Vert_{\infty} = \Vert V_{k+1}  - V_k + V_k - TV_k + TV_k - V^* \Vert_{\infty}\\
&= \Vert \Delta V_k + r(V_k) + TV_k - V^* \Vert_{\infty}\\
&= \Vert \Delta V_k + \left( I - \gamma P^{\pi_{k+1}} \right)r(V_k) + \gamma P^{\pi_{k+1}}r(V_k) + TV_k - V^* \Vert_{\infty}\\
&\leq \Vert \Delta V_k + \left( I - \gamma P^{\pi_{k+1}} \right)r(V_k) \Vert_{\infty} +  \Vert \gamma P^{\pi_{k+1}}r(V_k) \Vert_{\infty} +  \Vert TV_k - V^* \Vert_{\infty}\\
&\overset{(a)}{\leq} \alpha\Vert r(V_k) \Vert_{\infty} +  \Vert \gamma P^{\pi_{k+1}}\left(r(V_k) - r(V^*)\right) \Vert_{\infty} +  \Vert TV_k - TV^* \Vert_{\infty}\\
&\overset{\phantom{(c)}}{\leq} \alpha\Vert r(V_k) \Vert_{\infty} +  \gamma\Vert  P^{\pi_{k+1}} \Vert_{\infty} \Vert\left(r(V_k) - r(V^*)\right) \Vert_{\infty} +  \Vert TV_k - TV^* \Vert_{\infty}\\
&\overset{(b)}{\leq} \alpha(1+\gamma)\Vert V_k - V^* \Vert_{\infty} +  \gamma(1+\gamma)\Vert V_k - V^* \Vert_{\infty}+  \gamma\Vert V_k - V^* \Vert_{\infty}\\
&= ((\alpha + \gamma)(1+\gamma) + \gamma)\Vert V_k - V^* \Vert_{\infty} \,,
\end{aligned}
\end{equation}
where $(a)$ follows from the stopping condition and $(b)$ from Lemma 2 in~\cite{gargiani_2023} and Proposition~\ref{prop: gamma_contractivity}. The final result follows from observing that $(\alpha + \gamma)(1+\gamma) + \gamma \in (0,1)$ if and only if $\gamma \in (0, -1 + \sqrt{2})$ and $\alpha \in (0, (1 - 2\gamma - \gamma^2)/(1+\gamma))$.
\end{proof}

Going back to Inequality~\eqref{eq: ineq_iPI_wrong}, we now assume that the sequence of iterates $\left\{ V_k \right\}$ generated by a general iPI method is such that $V_k \geq TV_k$ for $k \geq 0$.
This requirement has a strong relation with the concept of \textit{region of decreasing} as introduced in~\cite[Section 5.2]{yuchao_phd}. In particular, it corresponds to assuming that $\left\{ V_k \right\}$ is contained in this subset. 
The assumption is also comparable to some sort of policy improvement requirement as it implies that $V_k \geq V^{\pi_{k+1}}$. 
Indeed, by definition of greedy-policy, we have that $TV_k = T_{\pi_{k+1}}V_k$. Because of the monotonicity of the $T_{\pi}$-operator for any $\pi\in \Pi$, we have that $V_k\geq T_{\pi_{k+1}}V_k$ implies that $T_{\pi_{k+1}}V_k \geq V^{\pi_{k+1}}$. In addition, by definition of optimal cost, $V^{\pi}\geq V^*$ for any $\pi\in\Pi$. Putting these pieces together we have that 
\begin{equation}
TV_k \geq V^{\pi_{k+1}} \geq V^*\,.
\end{equation}
By subtracting the term $V^*$ using the fact that that $V^* = TV^*$, we obtain that
\begin{equation}\label{eq:ineq_TVk}
TV_k  - TV^* \geq V^{\pi_{k+1}} - V^* \geq 0\,,
\end{equation}
which implies the following
\begin{equation}\label{eq:ineq_TVk}
\Vert TV_k  - TV^* \Vert_{\infty} \geq \Vert V^{\pi_{k+1}} - V^* \Vert_{\infty}\,.
\end{equation}
By using~\eqref{eq:ineq_TVk} in~\eqref{eq: ineq_iPI_wrong} and exploiting the fact that the $T$-operator is $\gamma$-contractive in the infinity-norm, we obtain the following upper-bound 
\begin{equation}\label{eq:ineq_conv_policyimprovement}
\Vert V_{k+1} - V^* \Vert_{\infty} \leq  \frac{(1+\gamma)\alpha}{1-\gamma} \Vert V_{k} - V^* \Vert_{\infty} + \gamma \Vert V_{k} - V^* \Vert_{\infty}\,.
\end{equation}
From~\eqref{eq:ineq_conv_policyimprovement} we can then conclude global Q-linear convergence of a general iPI method with rate $\frac{(1 + \gamma)\alpha}{1 - \gamma} + \gamma$ if $\alpha \in (0,\, (1-\gamma)^2/(1+\gamma))$. 
Having an assumption on the generated sequence of iterates is not ideal as it is not possible to verify it a priori. It would therefore be more practical to modify the algorithm in order to enforce this condition, but this would obviously radically change the general iPI scheme.

In the following, we discuss global convergence of iPI methods. In particular, first we show global convergence to a neighborhood of the solution for the same interval of $\alpha$-values for which local contraction is guaranteed. Then, by introducing an extra requirement on the stage-cost bound, we show that there exists a non-empty interval of $\alpha$-values for which global convergence to the solution is guaranteed. Differently from  Proposition~\ref{prop: global_gamma}, these results hold for any $\gamma\in(0,1)$.
\begin{theorem}[global convergence to a neighborhood of $V^*$]\label{th: global_conv_iPI_ball}
Consider a general infinite-horizon discounted MDP with finite spaces and let $\left\{ V_k \right\}$ be the sequence of iterates generated by a general iPI method as described in Algorithm~\ref{alg:iPI}. If $\alpha \in (0,\,(1-\gamma)/(1+\gamma))$, then, for any $V_0 \in \mathbb{R}^n$, $\left\{V_k\right\}$ asymptotically converges to $\mathcal{B}(V^*, 2R/((1-\gamma)(1-\zeta)))$, where $\zeta = \frac{1+\gamma}{1-\gamma}\alpha$.
\end{theorem}
\begin{proof}
In the following, we denote with $\left\{ \pi_{k+1} \right\}_{k \geq 0}$ and $\left\{ \pi^{\text{\tiny{PI}}}_k \right\}_{k\geq 0}$ the sequences of greedy policies that are generated by iPI and exact PI starting from $V_0$ and $\tilde{V}$, respectively. Let $\bar{V}\in\mathbb{R}^n$, we assume that $V_0 = \tilde{V} = \bar{V}$ and that the first policy for both sequences is initialized with the same greedy policy with respect to $\bar{V}$. Therefore, $\pi_1 = \pi^{\text{PI}}_0$ while for $k>0$ the equality does not necessarily hold since in general $V_k \neq V^{\pi_k}$. We also use $J_{\pi_{k+1}}$ to denote $I - \gamma P^{\pi_{k+1}}$ and $J_{\pi^{\text{PI}}_{k+1}}$ to denote $I - \gamma P^{\pi^{\text{PI}}_{k+1}}$. Following similar steps as in the proofs of Proposition~\ref{prop: upper_bound} and~\ref{prop: global_gamma}, we upper-bound $\Vert V_{k+1} - V^*\Vert_{\infty}$ by using the fundamental properties of the infinity-norm, the structural properties of the Bellman residual function and the stopping-condition. In particular, we obtain the following upper-bound 
\begin{equation}\label{eq: upperbound_iPI_first}
\begin{aligned}
&\Vert V_{k+1} - V^* \Vert_{\infty} = \Vert V_k - V^* + \Delta V_k \Vert_{\infty} \\
&= \Vert J_{\pi_{k+1}}^{-1}\left(J_{\pi_{k+1}} V_k - g^{\pi_{k+1}} + J_{\pi_{k+1}} \Delta V_k \right) + J_{\pi_{k+1}}^{-1}g^{\pi_{k+1}} + V^{\pi^{\text{PI}}_{k+1}} - V^* - V^{\pi^{\text{PI}}_{k+1}} \Vert_{\infty} \\
& \leq  \Vert J_{\pi_{k+1}}^{-1}\left(J_{\pi_{k+1}} V_k - g^{\pi_{k+1}} + J_{\pi_{k+1}} \Delta V_k \right) \Vert_{\infty}  + \Vert J_{\pi_{k+1}}^{-1}g^{\pi_{k+1}} - V^{\pi^{\text{PI}}_{k+1}} \Vert_{\infty} \\
& \phantom{\leq}\,\, + \Vert V^{\pi^{\text{PI}}_{k+1}} - V^*  \Vert_{\infty} \\
& =  \Vert J_{\pi_{k+1}}^{-1}\left(J_{\pi_{k+1}} V_k - g^{\pi_{k+1}} + J_{\pi_{k+1}} \Delta V_k \right) \Vert_{\infty}  + \Vert J_{\pi_{k+1}}^{-1}g^{\pi_{k+1}} - J_{\pi^{\text{PI}}_{k+1}}^{-1} g^{\pi^{\text{PI}}_{k+1}} \Vert_{\infty} \\
& \phantom{\leq}\,\, + \Vert  V^{\pi^{\text{PI}}_{k+1}} - V^*  \Vert_{\infty} \\
& \leq  \frac{1+\gamma}{1-\gamma}\alpha\Vert V_k - V^* \Vert_{\infty}  + \max_{\pi \,\,\in\,\, \left\{\pi_{k+1}, \, \pi^{\text{PI}}_{k+1}\right\}}\Vert J_{\pi}^{-1} ( g^{\pi_{k+1}} -  g^{\pi^{\text{PI}}_{k+1}}) \Vert_{\infty} \\
& \phantom{\leq}\,\, + \Vert  V^{\pi^{\text{PI}}_{k+1}} - V^*  \Vert_{\infty} \\
& \leq  \frac{1+\gamma}{1-\gamma}\alpha\Vert V_k - V^* \Vert_{\infty}  + \max_{\pi \,\,\in\,\, \left\{\pi_{k+1}, \, \pi^{\text{PI}}_{k+1}\right\}}\Vert J_{\pi}^{-1}\Vert_{\infty} \Vert  g^{\pi_{k+1}} -  g^{\pi^{\text{PI}}_{k+1}} \Vert_{\infty} \\
& \phantom{\leq}\,\, + \Vert  V^{\pi^{\text{PI}}_{k+1}} - V^*  \Vert_{\infty} \\
& \leq  \frac{1+\gamma}{1-\gamma}\alpha\Vert V_k - V^* \Vert_{\infty}  + \frac{1}{1-\gamma} \Vert  g^{\pi_{k+1}} -  g^{\pi^{\text{PI}}_{k+1}} \Vert_{\infty} + \Vert  V^{\pi^{\text{PI}}_{k+1}} - V^*  \Vert_{\infty} \\
& \overset{(a)}{\leq}  \frac{1+\gamma}{1-\gamma}\alpha\Vert V_k - V^* \Vert_{\infty}   + \Vert  V^{\pi^{\text{PI}}_{k+1}} - V^*  \Vert_{\infty} + \frac{2R}{1-\gamma}\\
&\overset{(b)}{\leq} \frac{1+\gamma}{1-\gamma}\alpha\Vert V_k - V^* \Vert_{\infty}   + \gamma^{k+1}\Vert  V^{\pi_0} - V^*  \Vert_{\infty} + \frac{2R}{1-\gamma}\,,
\end{aligned}
\end{equation}
where Inequality $(a)$ follows from the fact that, for any $\pi\in\Pi$, $g^{\pi} : \mathcal{S}\times\mathcal{A} \rightarrow [-R, \,R]$, and Inequality $(b)$ follows from Proposition~\ref{prop: PI_convergence}.

By applying Inequality~\eqref{eq: upperbound_iPI_first} recursively, we obtain the following upper-bound
\begin{equation}\label{eq: upperbound_recursive}
\begin{aligned}
\Vert V_{k+1} - V^* \Vert_{\infty} \leq\,\, & \zeta^{k+1} \Vert V_0 - V^* \Vert_{\infty} + \Vert V^{\pi_0} - V^* \Vert_{\infty}  \gamma^{k+1} \sum_{j=0}^k \left(\frac{\zeta}{\gamma} \right)^j \\
&\quad\, + \frac{2R}{1 - \gamma}\sum_{j=0}^k \zeta^j\,,
\end{aligned}
\end{equation}
where $\zeta = \frac{1+\gamma}{1-\gamma}\alpha$. 
We now take the limit for $k\rightarrow \infty$ of~\eqref{eq: upperbound_recursive}. Since $\alpha \in (0, (1-\gamma)/(1+\gamma))$, then $\zeta < 1$ and therefore $\lim_{k\rightarrow \infty} \zeta^{k+1} \Vert {V}_0 - V^* \Vert_{\infty} = 0$ and $\lim_{k\rightarrow\infty} \frac{2R}{1-\gamma}\sum_{j=0}^k \zeta^j = \frac{2R}{(1-\gamma)(1-\zeta)} $. To study the limit of the second term in the RHS of~\eqref{eq: upperbound_recursive}, we distinguish between three scenarios: $\zeta < \gamma$, $\zeta > \gamma$ and $\zeta = \gamma$. In the first scenario, by using the properties of the geometric series with common ratio strictly less than $1$, we can directly conclude that $\lim_{k\rightarrow \infty} \gamma^{k+1} \sum_{j=0}^k \left( \frac{\zeta}{\gamma}\right)^j = 0 \cdot \frac{1}{1 - \frac{\zeta}{\gamma}} = 0$. For the second scenario, we can apply a similar argument starting from the observation that $\gamma^{k+1}\sum_{j=0}^k \left( \frac{\zeta}{\gamma} \right)^j$ can be equivalently rewritten as $\gamma\zeta^k \sum_{j=0}^k \left( \frac{\gamma}{\zeta} \right)^j$. Finally, when $\zeta = \gamma$, $ \gamma^{k+1}\sum_{j=0}^k \left( \frac{\zeta}{\gamma} \right)^j = (k+1)\gamma^{k+1}$. By using l'H$\hat{\text{o}}$pital's rule, we can conclude that $\lim_{k\rightarrow \infty} (k+1)\,\gamma^{k+1} = \lim_{k\rightarrow \infty}  \frac{\gamma^{k-1}}{k} = 0$. Therefore, by combining these derivations, we obtain the following upper-bound 
\begin{equation}
\lim_{k\rightarrow \infty} \Vert V_{k+1} - V^* \Vert_{\infty} \leq  \frac{2R}{(1 - \gamma)(1-\zeta)}\,,
\end{equation}
from which we can conclude that the sequence of iterates generated by a general iPI method asymptotically converges to $\mathcal{B}(V^*, \, 2R/((1 - \gamma)(1-\zeta)) )$. 
\end{proof}

\begin{corollary}
Consider the same conditions of Theorem~\ref{th: global_conv_iPI_ball}. If $R < \frac{(1-\gamma)^{3/2}\delta}{2}$, then $\left\{ \Vert V_k - V^* \Vert_{\infty} \right\}$ asymptotically converges to zero for any $V_0 \in \mathbb{R}^n$.
\end{corollary}
\begin{proof}

We deploy the results of Theorem~\ref{th: iPI_localconvergence} and~\ref{th: global_conv_iPI_ball} and derive a condition on $\alpha$ such that the estimated radius of the global region of convergence is smaller than the radius of the local region of attraction. This boils down to solving the following inequality in $\alpha$
\begin{equation}\label{eq: ineq_balls}
\frac{2R}{(1-\gamma)(1-\zeta)} \leq \frac{(1-\gamma) - (1 + \gamma)\alpha}{C_{\min}}\,.
\end{equation}
With some algebraic manipulation, we can rewrite~\eqref{eq: ineq_balls} as follow
\begin{equation}\label{eq:ineq_alpha}
(1+\gamma)^2\alpha^2 - 2(1-\gamma)(1+\gamma)\alpha + (1-\gamma)^2 - 2RC_{\min} \geq 0\,.
\end{equation}
Inequality~\eqref{eq:ineq_alpha} is verified for $\alpha \leq \alpha_1$ and $\alpha \geq \alpha_2$, where
\begin{equation}
\alpha_{1,2} = \frac{(1-\gamma) \pm \sqrt{2RC_{\min}}}{(1+\gamma)}\,.
\end{equation}  
The second interval can not be accepted since to deploy the local contraction results we need $\alpha \in (0, (1-\gamma)/(1+\gamma))$ and $\alpha_1 \geq (1-\gamma)/(1+\gamma)$ by construction since $R, C_{\min} \geq 0$. The first interval is valid as it is always contained in $(0, (1-\gamma)/(1+\gamma))$ if $R < (1-\gamma)^{3/2}\delta/2$, where the latter requirement indeed ensures that $\sqrt{2RC_{\min}}<(1-\gamma)$. 
\end{proof}


Algorithm~\ref{alg:iPI} is well-posed if the stopping condition for the inexact policy evaluation step can be verified by the iterative linear solver in a finite number of iterations. The following proposition provides guarantees that the stopping condition is verified in a finite number of iterations whenever the iterative linear solver generates a sequence of residuals (iterates) that is linearly contractive to zero (the solution) in some norm. 

\begin{proposition}\label{prop: iPI_stoppingcondition}
    Consider a general infinite-horizon discounted MDP with finite spaces and the linear system for the evaluation of the cost associated to a policy $\pi\in\Pi$. 
Let $z,\,\alpha\in(0,1)$ and $C_*\geq 1$ be constants, $\Vert \cdot \Vert_*$ an arbitrary norm and $\Phi^{\pi}:\mathbb{R}^n \rightarrow \mathbb{R}^n$ the function with $\Phi^{\pi}(V) = g^{\pi} - \left( I - \gamma P^{\pi} \right)V $. Consider an iterative solver that generates a sequence of iterates $\left\{\theta_i\right\}$ such that
$\Vert \Phi^{\pi}(\theta_i) \Vert_* \leq z^i\Vert \Phi^{\pi}(\theta_0) \Vert_*
$ or
$\Vert \theta_i - V^{\pi} \Vert_* \leq z^i\Vert \theta_0 - V^{\pi} \Vert_*
$ for all $i\geq 0$ and $\theta_0 \in \mathbb{R}^n$.
Then 
\begin{equation}
\Vert \Phi^{\pi}(\theta_i) \Vert_{\infty} \leq \alpha\,\Vert \Phi^{\pi}(\theta_0) \Vert_{\infty} 
\end{equation}
for all $i\geq \lceil  \log_z\frac{\alpha}{C_*}\frac{1-\gamma}{1+\gamma}\rceil$.
\end{proposition}
\begin{proof}
Because of the equivalence of norms on finite-dimensional spaces, we know that there exists $C_{*}\geq 1$ such that, if $\Vert  \Phi^{\pi}(\theta_i) \Vert_{*} \leq z^i \,\Vert \Phi^{\pi}(\theta_0) \Vert_{*}$, then $\Vert \Phi^{\pi}(\theta_i)  \Vert_{\infty} \leq z^i \, C_*\,\Vert \Phi^{\pi}(\theta_0) \Vert_{\infty}$. Analogously, if $\Vert \theta_{i} - V^{\pi} \Vert_{*} \leq z^i\,\Vert \theta_{0} - V^{\pi} \Vert_{*}$, then $\Vert \theta_{i} - V^{\pi} \Vert_{\infty} \leq z^i\,C_*\,\Vert \theta_{0} - V^{\pi} \Vert_{\infty}$.
From these relations we can easily see that, in the first scenario, the condition $\Vert \Phi^{\pi}(\theta_i) \Vert_{\infty} \leq \alpha\Vert \Phi^{\pi}(\theta_0) \Vert_{\infty}$ holds if $z^{i} C_* \leq \alpha$, which leads to $i \geq \lceil \log_z \left( {\alpha}/{C_*}\right) \rceil$. To find a lower-bound on the number of iterations such that $\Vert \Phi^{\pi}(\theta_i)\Vert_{\infty} \leq \alpha  \Vert \Phi^{\pi}(\theta_0)\Vert_{\infty}$ for the second scenario, we start by upper-bounding the infinity-norm of $\Phi^{\pi}(\theta_i)$
\begin{equation}\label{eq: uppberbound_i}
\begin{aligned}
\Vert \Phi^{\pi}(\theta_i) \Vert_{\infty} &= \Vert g^{\pi} - \left( I - \gamma P^{\pi} \right)\theta_i \Vert_{\infty}\\
&= \Vert  \left( I - \gamma P^{\pi} \right)\left( \left( I - \gamma P^{\pi} \right)^{-1} g^{\pi} - \theta_i\right) \Vert_{\infty}\\
&\overset{(a)}{\leq} \Vert I - \gamma P^{\pi} \Vert_{\infty} \Vert \left( I - \gamma P^{\pi} \right)^{-1} g^{\pi} - \theta_i \Vert_{\infty}\\
&= \Vert I - \gamma P^{\pi} \Vert_{\infty} \Vert  V^{\pi} - \theta_i \Vert_{\infty}\\
&\overset{(b)}{\leq} z^i\,C_*\,\Vert I - \gamma P^{\pi} \Vert_{\infty} \Vert \theta_0 - V^{\pi} 
\Vert_{\infty}\\
&= z^i\,C_*\,\Vert I - \gamma P^{\pi} \Vert_{\infty} \Vert \left(I-\gamma P^{\pi}\right)^{-1}\left(\left(I-\gamma P^{\pi}\right)\theta_0 - g^{\pi}\right) \Vert_{\infty} \\
&\overset{(c)}{\leq}  z^i\,C_*\,\Vert I - \gamma P^{\pi} \Vert_{\infty} \Vert \left(I - \gamma P^{\pi}\right)^{-1} \Vert_{\infty} \Vert \left(I-\gamma P^{\pi}\right)\theta_0 - g^{\pi} \Vert_{\infty}  \\
&=  z^i\,C_*\,\Vert I - \gamma P^{\pi} \Vert_{\infty} \Vert \left(I - \gamma P^{\pi}\right)^{-1} \Vert_{\infty} \Vert \Phi^{\pi}(\theta_0) \Vert_{\infty}  \\
&\overset{(d)}{\leq} z^i \, C_*\, \frac{1+\gamma}{1-\gamma}\, \Vert \Phi^{\pi}(\theta_0) \Vert_{\infty} \,,
\end{aligned}
\end{equation}
where Inequalities $(a)$ and $(c)$ follow from the application of the triangle inequality, Inequality $(b)$ from the convergence properties of the iterative linear solver in the infinity-norm and Inequality $(d)$ from the upper-bounds on the infinity-norm of the matrix $I-\gamma P^{\pi}$ and its inverse. From~\eqref{eq: uppberbound_i} it follows that $\Vert \Phi^{\pi}(\theta_i) \Vert_{\infty} \leq \alpha \Vert \Phi^{\pi}(\theta_0) \Vert_{\infty}$ if $z^{i}\, C_*\, \frac{1+\gamma}{1-\gamma}\leq \alpha$, which leads to $i\geq \lceil \log_z \left( \alpha/C_* \cdot\, (1-\gamma)/(1+\gamma)\right)  \rceil$.  Therefore we can conclude that, since $(1-\gamma)/(1+\gamma)<1$ and $z<1$, in both scenarios $i\geq \lceil \log_z \left( \alpha/C_* \cdot\, (1-\gamma)/(1+\gamma)\right)  \rceil$ is a sufficient number of iterations to guarantee that $\Vert \Phi^{\pi}(\theta_i) \Vert_{\infty} \leq \alpha\Vert \Phi^{\pi}(\theta_0) \Vert_{\infty}$.
\end{proof}

In the context of iPI, as described in Algorithm~\ref{alg:iPI}, at the $k$-th iteration we consider the linear system associated to policy $\pi_{k+1}$ and $\theta_0$ is initialized with the current outer iterate $V_k$ (step 5 in Algorithm~\ref{alg:iPI}). The approximate solution $\theta_i$ generated by the inner solver after $i>0$ iterations is then used to update the outer iterate (step 11 in Algorithm~\ref{alg:iPI}). The results of Proposition~\ref{prop: iPI_stoppingcondition} ensure that iPI methods are well-posed provided an inner solver which is $z$-contractive in some arbitrary norm, since the stopping condition of the inner loop is verified in at most $\lceil  \log_z\frac{\alpha}{C_*}\frac{1-\gamma}{1+\gamma}\rceil$ inner iterations.
      
\subsection{Iterative Linear Solvers for Policy Evaluation}\label{sec: iterative_LS}

Inexact policy iteration methods require at every iteration the approximate solution of a linear system of the following form
\begin{equation}\label{eq: policy_eval}
\left( I - \gamma P^{\pi} \right) \Delta \theta = -r^{\pi}(\bar{V})\quad\quad \pi \in \Pi\,,
\end{equation}
where $\bar{V}\in\mathbb{R}^n$, $\Delta \theta = \theta - \bar{V}$ and the approximation level is regulated by the stopping condition in~\eqref{eq:simplified_stopping_condition}. Notice that, given a $\pi\in\Pi$,~\eqref{eq: policy_eval} is equivalent to $(I - \gamma P^{\pi})\theta = g^{\pi}$.

%
In the choice of the solver for a linear system of equations, the structural properties of the coefficient matrix play a fundamental role. In particular, for any $\pi\in\Pi$, the coefficient matrix in~\eqref{eq: policy_eval} has eigenvalues contained in the circle with center $(1,0)$ and radius $\gamma$ in the complex plane~\cite[Lemma 1]{gargiani_2023}. Consequently, the coefficient matrix is always non-singular, while we can not rely on symmetry unless the underlying MDP enjoys a specific structure. Its symmetric component is 
\begin{equation}\label{eq: symmetric_component}
H^{\pi, \,\gamma} = I - \frac{\gamma}{2}\left( P^{\pi} + {P^{\pi}}^{\top} \right) = I - \gamma P^{\pi}_s\,.
\end{equation}  

An approximate solution of~\eqref{eq: policy_eval} can therefore  be provided by any linear and non-linear iterative method for linear systems with a non-singular coefficient matrix~\cite{saad, iterative_solutions}.
For a general iPI method, the choice of the inner solver should though be tailored to the specific structure of~\eqref{eq: policy_eval} and, as underlined in Proposition~\ref{prop: iPI_stoppingcondition},  it is of crucial importance in order to obtain fast convergence. In this work we consider Richardson's method~\cite[Chapter 3.2.1]{iterative_solutions}, steepest descent~\cite[Chapter 9.2]{iterative_solutions}, the minimal residual method~\cite[Chapter 9.4]{iterative_solutions} and GMRES~\cite{gmres}. For each of these methods we analyze the convergence properties and computational costs when applied to solve~\eqref{eq: policy_eval}.
 


\subsubsection{Richardson's Method}\label{subsec:richardson}
The iteration of Richardson's method for~\eqref{eq: policy_eval} is 
\begin{equation}\label{eq: RICH_iteration}
\Delta {\theta}_{i+1} = \left( I - \frac{1}{\nu}\left( I - \gamma P^{\pi} \right) \right) \Delta \theta_i - \frac{1}{\nu} r^{\pi}(\bar{V})\,,
\end{equation}
with $\nu >0$ and $\Delta {\theta}_i = {\theta}_i - \bar{V}$ for any $i\geq 0$ and $\bar{V}
\in\mathbb{R}^n$. For compactness, we use $M_{\text{Rich},\,\nu} = I - \frac{1}{\nu}\left( I - \gamma P^{\pi} \right)$ to denote its iteration matrix.

\begin{proposition}\label{prop: VI_Rich}
Consider a general infinite-horizon discounted MDP with finite spaces and the linear system for the evaluation of the cost associated to a policy $\pi\in\Pi$. VI for policy evaluation is the instance of Richardson's method obtained by setting $\nu =1$.
\end{proposition}
\begin{proof}
Consider an arbitrary ${\theta}_i \in \mathbb{R}^n$ and we denote with ${\theta}_{i+1}^{\text{VI}}$ and $\Delta {\theta}_{i+1}^{\text{Rich}, \, \nu}$ the VI for policy evaluation and Richardson's updates, respectively.
Recall that the iteration of VI for policy evaluation starting from ${\theta}_i$ is
\begin{equation}
{\theta}_{i+1}^{\text{VI}} = T_{\pi} {\theta}_i\,.
\end{equation}
We now consider the Richardson's iteration starting from $\Delta {\theta}_i = {\theta}_i - \bar{V}$ with $\nu =1$ 
\begin{equation}\label{eq: proof_virich}
\begin{aligned}
\Delta {\theta}_{i+1}^{\text{Rich},\,1} &= \gamma P^{\pi} \Delta {\theta}_i - r^{\pi}(\bar{V})\\
&= g^{\pi} + \gamma P^{\pi} {\theta}_i - \bar{V}\\
&= T_{\pi}{\theta}_i - \bar{V}\,. 
\end{aligned}
\end{equation}
By simplifying the $\bar{V}$-terms on both sides of~\eqref{eq: proof_virich} we obtain that ${\theta}_{i+1}^{\text{Rich},\,1} = {\theta}_{i+1}^{\text{VI}}$. 
\end{proof}
Richardson's method can be therefore interpreted as a generalization of the value iteration method for policy evaluation. In particular, the $\nu$-parameter plays a fundamental role in determining its convergence properties.
We now study the values of $\nu$ for which Richardson's iterations enjoy monotone convergence in the infinity-norm (see Section 2.2.5 in~\cite{iterative_solutions} for a definition of monotone convergence with respect to a norm).
\begin{proposition}[monotone convergence in the infinity-norm]\label{prop: Rich_monotone_conv}
Consider a general infinite-horizon discounted MDP with finite spaces and the linear system for the evaluation of the cost associated to a policy $\pi\in\Pi$. Let $\left\{ {\theta}_{i}\right\}$ be the sequence of iterates generated by~\eqref{eq: RICH_iteration} starting from $V_0 \in \mathbb{R}^n$. If $\nu >  \frac{1+\gamma}{2}$, then for any $i\geq 0$
\begin{equation}
\Vert {\theta}_{i+1} - V^{\pi} \Vert_{\infty} \leq\, h(\nu, \,\gamma)\, \Vert {\theta}_{i} - V^{\pi} \Vert_{\infty}\,,
\end{equation}
where  $h(\nu, \gamma) = \frac{\vert \nu - 1 \vert}{\nu} + \frac{\gamma}{\nu} <1$.
\end{proposition}
\begin{proof}
A sufficient condition for monotone convergence in the infinity-norm is 
$
\Vert M_{\text{Rich},\,\nu} \Vert_{\infty} < 1\,
$~\cite[Th. 2.19]{iterative_solutions}.
We start by computing an upper-bound on the infinity-norm of the iteration matrix 
\begin{equation}\label{eq: upper-bound_CN}
\Vert M_{\text{Rich},\,\nu} \Vert_{\infty} \leq \underbrace{\frac{\vert \nu - 1 \vert}{\nu} + \frac{\gamma}{\nu}}_{:= h(\nu,\,\gamma)}\,,
\end{equation}
where the inequality follows from the application of the triangle-inequality. 
We now study the values of $\nu$ for which $h(\nu,\,\gamma)<1$. 
It is straightforward to verify that, since $\gamma\in(0,1)$, $h(\nu,\,\gamma) < 1$ if and only if $\nu > \frac{1+\gamma}{2}$. 
\end{proof}
The results of Proposition~\ref{prop: Rich_monotone_conv} allow one to have an upper-bound on the infinity-norm of the error for some fixed value of $m$. From this analysis, it is though not possible to conclude that there exists values of $\nu$ which result into a faster monotone convergence than value iteration for policy evaluation ($\nu=1$). It is also straightforward to see that 
\begin{equation}
1 = \arg\min_{\nu > \frac{1+\gamma}{2}} h(\nu,\, \gamma)\,,
\end{equation}
and $h(1,\,\gamma) = \gamma$. A necessary and sufficient condition on convergence of Richardson's method for policy evaluation can be formulated by studying the values of $\nu$ for which $\rho\left(M_{\text{Rich},\, \nu}\right)<1$~\cite[Th. 2.16]{iterative_solutions}. In the literature, the spectral radius of the iteration matrix of a linear iteration method is called \textit{convergence rate} since it determines the asymptotic convergence rate of the method. Differently from the contraction number~\cite[Section 2.2.5]{iterative_solutions}, the convergence rate only allows for asymptotic statements concerning the convergence of the error. 
In particular, $\rho(M_{\text{Rich},\,\nu})$ is a tight measure of the asymptotic convergence rate of Richardson's method. 
In the following theorem, we characterize the class of MDPs for which it is always possible to select $\nu$ in order to achieve asymptotic acceleration with respected to value iteration for policy evaluation, \textit{i.e.} $\rho(M_{\text{Rich},\,\nu}) < \gamma$, and the interval of $\nu$-values which leads to asymptotic acceleration. These results indicate that $\rho(M_{\text{Rich}\,,1}) = \gamma$ which suggests that for $\nu=1$ the upper-bound on the contraction number in~\eqref{eq: upper-bound_CN} is tight.
\begin{theorem}[accelerated convergence rate]\label{th: Rich_acceleration}
Consider a general infinite-horizon discounted MDP with finite spaces and the linear system for the evaluation of the cost associated to a policy $\pi\in\Pi$. If the MDP is regular, there always exists a non-empty interval $\left( \underline{\nu},\, 1 \right)$, with
\begin{equation}\label{eq: lower_bound_Rich_interval}
\underline{\nu} := \max_{\lambda\in\Lambda(P^{\pi})} \frac{\left(1 - \gamma \text{Re}\left\{ \lambda \right\} \right) - \gamma \sqrt{(\gamma^2 - 1)\text{Im}^2\left\{ \lambda \right\} + \left( 1 - \gamma \text{Re}\left\{ \lambda \right\} \right)^2  }}{1 - \gamma^2}\,,
\end{equation}
such that $\rho\left( M_{\text{Rich},\,\nu} \right) < \gamma$ for all $\nu \in (\underline{\nu},\,1)$.
\end{theorem}
\begin{proof}
We start by characterizing $\rho(M_{\text{Rich},\,\nu})$ as follow
\begin{equation}
\begin{aligned}
\rho(M_{\text{Rich},\,\nu}) &= \rho\left( I - \frac{1}{\nu}\left( I - \gamma P^{\pi} \right) \right)\\
&= \max_{\lambda \in \Lambda(P^{\pi})} \Big\vert 1 - \frac{1}{\nu}\left( 1 - \gamma \lambda \right) \Big\vert\,.
\end{aligned}
\end{equation}
Since $\rho(P^{\pi}) = 1$, the choice $\nu=1$ leads to $\rho(M_{\text{Rich}\,,1}) = 1$.
In order to obtain an accelerated rate, there has to exist $\underline{\nu},\,\bar{\nu} \in \mathbb{R}$ with $\underline{\nu} < \bar{\nu}$ such that  $\rho(M_{\text{Rich},\,\nu}) < \gamma$ for all $\nu \in (\underline{\nu},\,\bar{\nu})$. We therefore study the values of $\nu$ for which
\begin{equation}\label{eq: condition_alpha}
\max_{\lambda \in \Lambda(P^{\pi})} \Big\vert 1 - \frac{1}{\nu}\left( 1 - \gamma \lambda \right) \Big\vert < \gamma\,.
\end{equation}
\eqref{eq: condition_alpha} is equivalent to 
\begin{equation}
\Big\vert 1 - \frac{1}{\nu}\left( 1 - \gamma \lambda \right) \Big\vert < \gamma \quad \forall\, \lambda \in \Lambda(P^{\pi})\,, 
\end{equation}
which can be equivalently reformulated as follow
\begin{equation}\label{eq: condition2_alpha}
\left( \frac{\nu-1}{\nu} + \frac{\gamma}{\nu}\text{Re}\left\{ \lambda\right\} \right)^2 + \left( \frac{\gamma}{\nu} \text{Im}\left\{ \lambda\right\} \right)^2 < \gamma^2 \quad \forall\,\lambda \in \Lambda(P^{\pi})\,.
\end{equation}
After some basic algebraic manipulations,~\eqref{eq: condition2_alpha} becomes 
\begin{equation}\label{eq: quadratic_ineqs}
\nu^2 (1-\gamma^2) + 2\nu (-1 + \gamma \text{Re}\left\{ \lambda\right\}) + 1 + \gamma^2 \vert \lambda \vert^2  - 2\gamma \text{Re}\left\{ \lambda\right\} < 0 \quad \forall\,\lambda \in \Lambda(P^{\pi})\,.
\end{equation}
To solve the quadratic inequalities in~\eqref{eq: quadratic_ineqs}, we first compute their discriminant as a function of $\lambda\in \Lambda(P^{\pi})$
\begin{equation}\label{eq: discriminant}
\Delta(\lambda) = 4\gamma^2 \left((\gamma^2-1)\text{Im}^2\left\{\lambda \right\} +  \left(  1 - \gamma \text{Re}\left\{ \lambda \right\}\right)^2 \right)\,.
\end{equation}
From~\eqref{eq: discriminant} with some simple algebraic manipulations and since $\gamma\in(0,1)$, we can conclude that 
\begin{equation}\label{eq: discriminant_positive}
\Delta(\lambda) \geq 0 \quad \forall\,\lambda \in \Lambda(P^{\pi}) \iff (1-\gamma^2)\text{Im}^2\left\{\lambda \right\}\leq {(1-\gamma\text{Re}\left\{ \lambda \right\})^2} \quad \forall\,\lambda \in \Lambda(P^{\pi})\,.
\end{equation}
We know that, since $P^{\pi}$ is row-stochastic, its eigenvalues lie in the unit circle in the complex plane, \textit{i.e.},
\begin{equation}\label{eq: eig_incircle}
\text{Im}^2\left\{ \lambda \right\} + \text{Re}^2\left\{ \lambda \right\} \leq 1 \quad \forall\,\lambda \in \Lambda(P^{\pi})\,.
\end{equation}
By multiplying left and right hand-sides in~\eqref{eq: eig_incircle} by $(1-\gamma^2)$ and with some simple algebraic manipulations we obtain that the following condition holds for all $\lambda\in\Lambda(P^{\pi})$
\begin{equation}\label{eq: condition_gamma}
(1-\gamma^2) \text{Im}^2\left\{ \lambda \right\} \leq (1-\gamma^2)\left(1 -   \text{Re}^2\left\{ \lambda \right\} \right)\,.
\end{equation}
Since $(1-\gamma^2)\left(1 -   \text{Re}^2\left\{ \lambda \right\} \right) \leq \left( 1 - \gamma \text{Re}\left\{ \lambda \right\} \right)^2$, we can conclude that~\eqref{eq: condition_gamma} is more stringent that the right condition in~\eqref{eq: discriminant_positive} . Consequently, the right condition in~\eqref{eq: discriminant_positive} is always verified and therefore $\Delta(\lambda) \geq 0$ for all $\lambda \in P^{\pi}$.
We now solve the inequalities in~\eqref{eq: quadratic_ineqs} and for each $\lambda\in P^{\pi}$ we obtain an interval $\left( \nu_1(\lambda), \,\nu_2(\lambda)  \right)$ where
\begin{equation}
\nu_{1,2}(\lambda) = \frac{(1-\gamma \text{Re}\left\{ \lambda \right\}) \pm \gamma \sqrt{(\gamma^2 - 1)\text{Im}^2\left\{ \lambda \right\} + (1-\gamma \text{Re}\left\{ \lambda\right\})^2}}{(1-\gamma^2)}\,.
\end{equation}
The solution of~\eqref{eq: quadratic_ineqs} is therefore obtained by intersecting these intervals as follow
\begin{equation}
\nu \in \bigcap_{\lambda\in\Lambda(P^{\pi})} \left( \nu_1(\lambda),\,\nu_2(\lambda)\right)  = \left( \underline{\nu},\,\bar{\nu} \right)\,,
\end{equation}
where
$\underline{\nu} = \max_{\lambda\in\Lambda(P^{\pi})} \nu_1(\lambda)$ and $\bar{\nu} = \min_{\lambda\in\Lambda(P^{\pi})} \nu_2(\lambda)$.

We now show that $\bar{\nu} = 1$. We first show that $\bar{\nu} \geq 1$ with a \textit{reduction ad absurdum}. We assume that there exists $\hat{\lambda}\in\Lambda(P^{\pi})$ such that $\nu_2(\hat{\lambda})<1$. This would imply that the following inequality holds
\begin{equation}\label{eq: condition_alphabar}
\gamma \sqrt{(\gamma^2 - 1)\text{Im}^2\left\{ \hat{\lambda} \right\} + (1-\gamma \text{Re}\left\{ \hat{\lambda}\right\})^2} < \gamma\left( -\gamma + \text{Re}\left\{ \hat{\lambda} \right\} \right)\,.
\end{equation}
We now distinguish two scenarios: $\text{Re}\left\{\hat{\lambda} \right\} \leq \gamma$ and $\text{Re}\left\{\hat{\lambda} \right\} > \gamma$. In the first case, since the left hand-side in~\eqref{eq: condition_alphabar} is always non-negative, we can directly conclude that the original assumption is absurd and therefore $\bar{\nu}\geq 1$. For the second scenario, we can square both sides in~\eqref{eq: condition_alphabar} to obtain the following equivalent condition
\begin{equation}\label{eq: condition_alphabar2}
(\gamma^2 - 1) \text{Im}^2\left\{ \hat{\lambda} \right\} < (1-\gamma^2)\left( \text{Re}^2\left\{ \hat{\lambda}\right\} - 1\right)\,.
\end{equation} 
With some algebraic manipulations~\eqref{eq: condition_alphabar2} can be reformulated as follow
\begin{equation}
(1-\gamma^2) (\vert \hat{\lambda} \vert^2 - 1) > 0\,,
\end{equation} 
which is verified if and only if $\vert \hat{\lambda} \vert > 1$. Since all eigenvalue of $P^{\pi}$ have modulus not greater than 1, also for this second scenario we can conclude that the original assumption is absurd and therefore $\bar{\nu} \geq 1$. Finally, since row-stochastic matrices have always an eigenvalue at $(1, 0)$  and $\nu_2(1) = 1$, we can conclude that $\bar{\nu}=1$. With a similar argument, we can prove that $\underline{\nu} \leq 1$ and that the equality is achieved if and only if $\vert\lambda\vert^2 =1$ and $\text{Re}\left\{ \lambda \right\} \leq \gamma$ with $\lambda\in \Lambda(P^{\pi})$. While for a general MDP these conditions can be verified since $P^{\pi}$ could have multiple eigenvalues with modulus 1, for the case of regular MDPs we can conclude that $\underline{\nu}<1$ since 1 is the only eigenvalue on the unit circle and all the other eigenvalues lie strictly inside. We can therefore conclude that for a general MDP it may not be possible to select $\nu$ to obtain an accelerated rate with respect to VI since there could a policy $\pi\in\Pi$ for which $\underline{\nu}= \bar{\nu} = 1$, while for regular MDPs $(\underline{\nu},\,\bar{\nu})$ is never empty since $\underline{\nu} < \bar{\nu}=1$ and $\rho(M_{\text{Rich},\,\nu})<\gamma$ for all $\nu\in (\underline{\nu},\,\bar{\nu})$.   
\end{proof}
The results of Theorem~\ref{th: Rich_acceleration} show that value iteration for policy evaluation is not always optimal in the sense of contraction rate. For the class of regular MDPs for any $\pi\in\Pi$ the optimal $\nu$ in the sense of contraction rate is strictly less than 1. In particular, given the policy evaluation task associated to any $\pi\in\Pi$ of a regular MDP, it is always possible to achieve a faster convergence rate than $\gamma$ by setting $\nu$ in an appropriate interval which depends on the spectrum of the transition probability matrix and the discount factor. The benchmarks in Figure~\ref{fig:Richardson_policyevaluation} corroborate these results. As displayed in the figure for a regular MDP with discount factor $\gamma=0.9$, by setting $\nu=0.7$ we obtain a considerable improvement in terms of convergence rate with respect to value iteration for policy evaluation, while for $\nu>1$ we observe monotone convergence but with a slower rate. 

Richardson's iteration can also be used in combination with a non-singular preconditioning matrix $D\in\mathbb{R}^{n\times n}$. This leads to the following update
\begin{equation}\label{eq: preconditioned_Rich}
\Delta {\theta}_{i+1} = \Delta {\theta}_i - \frac{1}{\nu} D^{-1}\left( \left(I - \gamma P^{\pi} \right)\Delta {\theta}_i - g^{\pi} \right) \,.
\end{equation}
Notice that, if $\nu=1$ and $D = \text{diag}\left( I - \gamma P^{\pi} \right)$, then we recover the Jacobi iteration~\cite[Section 3.2.2]{iterative_solutions}. Finally, starting from~\eqref{eq: preconditioned_Rich} and with $D = \text{diag}\left( I - \gamma P^{\pi} \right)$, if the states are not processed all at once, but in a sequential manner making use of the interim results, then we recover the SOR iteration~\cite[Section 3.2.4]{iterative_solutions}, and, if we set $\alpha=1$, the Gauss-Seidel iteration~\cite[Section 3.2.3]{iterative_solutions}. The use of interim results may lead to a faster convergence in terms of number of iterations, as proved for instance in~\cite{gargiani_minibatch} and also confirmed by the empirical results displayed in Figure~\ref{fig:Richardson_policyevaluation}. From a computational viewpoint, SOR and Gauss-Seidel iterations are not suitable for parallel systems since they are inherently sequential, while Richardson and Jacobi iterations are fully parallelizable. The batch-versions of these latter algorithms constitute a solution to trade-off the two aspects and therefore may result into more efficient methods when a parallel system is deployed~\cite{gargiani_minibatch}.

Unfortunately, the evaluation of~\eqref{eq: lower_bound_Rich_interval} requires spectral data of the transition probability matrix whose computation is generally more expensive than the solution of the original problem. Therefore, practically, in order to enjoy the acceleration in convergence rate for regular MDPs, it is necessary to develop heuristic schemes in order to adjust $\nu$ avoiding the computation of $\Lambda(P^{\pi})$. In the following, we analyze alternative methods which adopt exact line-search strategies in order to select the step-size.      
\begin{figure}
    \centering
    \includegraphics[scale=0.6]{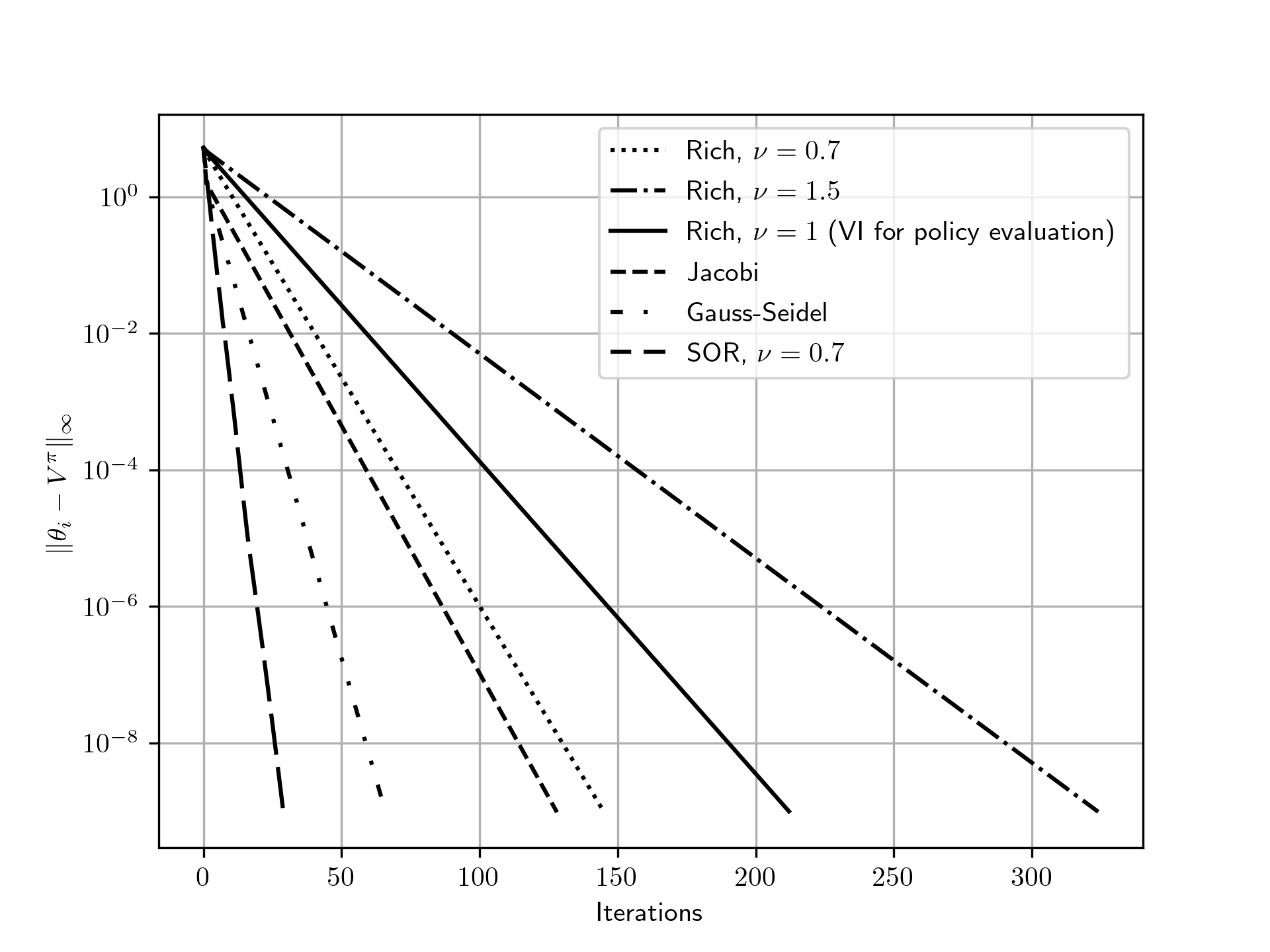}
    \caption{We consider a regular discounted MDP with $\gamma=0.9$, $n=100$ and $m=50$ and we solve the policy evaluation task for a random policy $\pi\in\Pi$ with Richardson's method with different values of $\nu$, Jacobi, Gauss-Seidel and SOR. The figure displays the distance in infinity-norm of the iterates from the solution versus number of iterations.}
    \label{fig:Richardson_policyevaluation}
\end{figure}

\subsubsection{Steepest Descent}\label{subsec:steepest_descent}
Instead of solving directly~\eqref{eq: linear_system_PE}, we can equivalently minimize a quadratic function of the following form
\begin{equation}\label{eq: quadratic_functional}
q^{\pi,\,\gamma}(\Delta \theta; \Xi,\, c) = \frac{1}{2}\langle \Xi\left(\left( I - \gamma P^{\pi}\right) \Delta \theta + r^{\pi}(\bar{V})\right),\,  \left( I - \gamma P^{\pi}\right) \Delta \theta + r^{\pi}(\bar{V})\rangle + c\,,
\end{equation} 
where $\Xi \succ 0$ and $c$ is an arbitrary constant. The minimization of~\eqref{eq: quadratic_functional} yields $V^{\pi}$ since any quadratic function of this form has a unique minimum which coincides with the solution of~\eqref{eq: linear_system_PE}~\cite[Lemma 9.3]{iterative_solutions}.
In order to minimize~\eqref{eq: quadratic_functional}, we can design an iteration of the following form
\begin{equation}\label{eq: general_iter_QP}
\Delta {\theta}_{i+1} = \Delta {\theta}_i + \eta_i\, d_i\,,
\end{equation}
where $d_i \in \mathbb{R}^n$ is a selected search-direction and $\eta_i > 0$ is the optimal step-size obtained with exact line-search. These are the main principles behind the steepest descent method, where $\Xi=I$, $c=0$ and $d_i = -\nabla \,q^{\pi,\,\gamma}(\Delta {\theta}_i; I,\, 0)$. In particular, with the choice of this quadratic function and descent-direction, iteration~\eqref{eq: general_iter_QP} becomes
\begin{equation}\label{eq: steepest_descent_iteration}
\Delta {\theta}_{i+1} = \Delta {\theta}_i - \eta_i \left(I - \gamma P^{\pi} \right)^\top \left( \left(I - \gamma P^{\pi} \right) \Delta {\theta}_i + r^{\pi}\left( \bar{V} \right) \right)\,,
\end{equation}
where $\eta_i = \frac{\langle -r^{\pi}(\bar{V}) - \left( I - \gamma P^{\pi} \right)\Delta {\theta}_i ,\, - \left( I - \gamma P^{\pi} \right) \nabla q^{\pi,\,\gamma}(\Delta {\theta}_i; I,\,0)\rangle}{\langle  \left( I - \gamma P^{\pi} \right) \nabla q^{\pi,\,\gamma}(\Delta {\theta}_i; I,\,0),\,\, \left( I - \gamma P^{\pi} \right) \nabla q^{\pi,\,\gamma}(\Delta {\theta}_i; I,\,0)  \rangle}$.
\bigskip

The following proposition characterizes the convergence properties of iteration~\eqref{eq: steepest_descent_iteration}.
\vspace{-0.5cm}
\begin{proposition}
Consider a general infinite-horizon discounted MDP with finite spaces and the linear system for the evaluation of the cost associated to a policy $\pi\in\Pi$. The sequence $\left\{ {\theta}_i \right\}$ generated by steepest descent~\eqref{eq: steepest_descent_iteration}  with $V_0 \in \mathbb{R}^n$ enjoys global Q-linear convergence to $\Delta V^* = V^{\pi} - \bar{V}$ with contraction number $\phi(\gamma, \,P^{\pi}) = \dfrac{\kappa(\left( I - \gamma P^{\pi} \right)^\top \left( I - \gamma P^{\pi} \right)) - 1}{\kappa( \left( I - \gamma P^{\pi} \right)^\top \left( I - \gamma P^{\pi} \right)) + 1}$ in the ${\left( I - \gamma P^{\pi}\right)^\top \left( I - \gamma P^{\pi}\right)}$-norm, \textit{i.e.}, for any $V_0\in \mathbb{R}^n$ and $i\geq 0$
\begin{equation}
\Vert {\theta}_{i+1} - V^{\pi}  \Vert_{\left( I - \gamma P^{\pi} \right)^\top \left( I - \gamma P^{\pi} \right)} \, \leq \, \phi(\gamma,\,P^{\pi})\Vert {\theta}_i - V^{\pi} \Vert_{\left( I - \gamma P^{\pi} \right)^\top \left( I - \gamma P^{\pi} \right)}\,.
\end{equation} 

\end{proposition}
\begin{proof}
Each iteration~\eqref{eq: steepest_descent_iteration} minimizes $q^{\pi,  \, \gamma}(\Delta \theta; I,\, 0)$ over all the vectors of the form $\Delta \theta - \eta \nabla q^{\pi,  \, \gamma}(\Delta \theta; I,\, 0)$. This is also equivalent to minimizing
over the same family of vectors the following quadratic function
\begin{equation}
f(x) = \frac{1}{2} \Vert \Delta \theta - \Delta {\theta}^* \Vert^2_{\left(  I - \gamma P^{\pi} \right)^\top \left( I - \gamma P^{\pi}\right)} \,,
\end{equation}
where $\Delta {\theta}^* = V^{\pi} - \bar{V} = -\left( I - \gamma P^{\pi} \right)^{-1} r^{\pi}(\bar{V})$. Since $\left(I - \gamma P^{\pi}  \right)$ is non-singular for any $\pi \in \Pi$, then this implies that $\left(I - \gamma P^{\pi}  \right)^\top \left(I - \gamma P^{\pi}  \right)\succ 0$. We can therefore utilize the results of Theorem 5.2 in~\cite{saad} to conclude the proof, where $A = \left( I - \gamma P^{\pi} \right)^\top \left( I - \gamma P^{\pi} \right)$ and $b = -\left(I - \gamma P^{\pi} \right)^\top r^{\pi}(\bar{V}) $.
 
\end{proof}

\subsubsection{Minimal Residual Method}\label{subsec:minres}
The minimal residual method minimizes the same quadratic function as steepest descent, but $d_i = -r^{\pi}(\bar{V}) - \left( I - \gamma P^{\pi} \right) \Delta {\theta}_i$. This results in the following iteration
\begin{equation}\label{eq:MR_iteration}
\Delta {\theta}_{i+1} = \Delta {\theta}_i + \eta_i \left( -r^{\pi}(\bar{V}) - \left( I - \gamma P^{\pi} \right)\Delta {\theta}_i \right)\,,
\end{equation}
where $\eta_i = \frac{\langle\left( I - \gamma P^{\pi} \right) \left( -r^{\pi}(\bar{V}) - \left( I - \gamma P^{\pi} \right)\Delta {\theta}_i \right),\, -r^{\pi}(\bar{V}) - \left( I - \gamma P^{\pi} \right)\Delta {\theta}_i \rangle}{\langle\left( I - \gamma P^{\pi} \right) \left( -r^{\pi}(\bar{V}) - \left( I - \gamma P^{\pi} \right)\Delta {\theta}_i \right),\,\left( I - \gamma P^{\pi} \right) \left( -r^{\pi}(\bar{V}) - \left( I - \gamma P^{\pi} \right)\Delta {\theta}_i \right)\rangle}$.
Notice that if we select a constant step-size $\eta>0$ in iteration~\eqref{eq:MR_iteration} we retrieve Richardson's iteration where $\nu = \frac{1}{\eta}$ in~\eqref{eq: RICH_iteration}. While in general the selection of the $\nu$-parameter in Richardson's method requires knowledge of the spectrum of the coefficient matrix, the minimal residual method adjusts this value on the fly based on local information using exact line search.    
Convergence results from the literature rely on the assumption that $H^{\pi,\,\gamma} \succ 0$. In particular, under this assumption, the 2-norm of the residual vectors of iteration~\eqref{eq:MR_iteration} converges to zero with contraction number at least $\kappa_{\text{MR}} = \sqrt{1 - \frac{\lambda^2_{\min}(H^{\pi\,,\gamma})}{\sigma^2_{\max}(I - \gamma P^{\pi})}}$~\cite[Th. 5.3]{saad}. As proved by the following lemma, in the considered problem setting the assumption $H^{\pi\,,\gamma}\succ 0$ is not necessarily verified for all values of $\gamma \in (0,1)$. 
\begin{figure}
    \centering
    \includegraphics[scale=0.6]{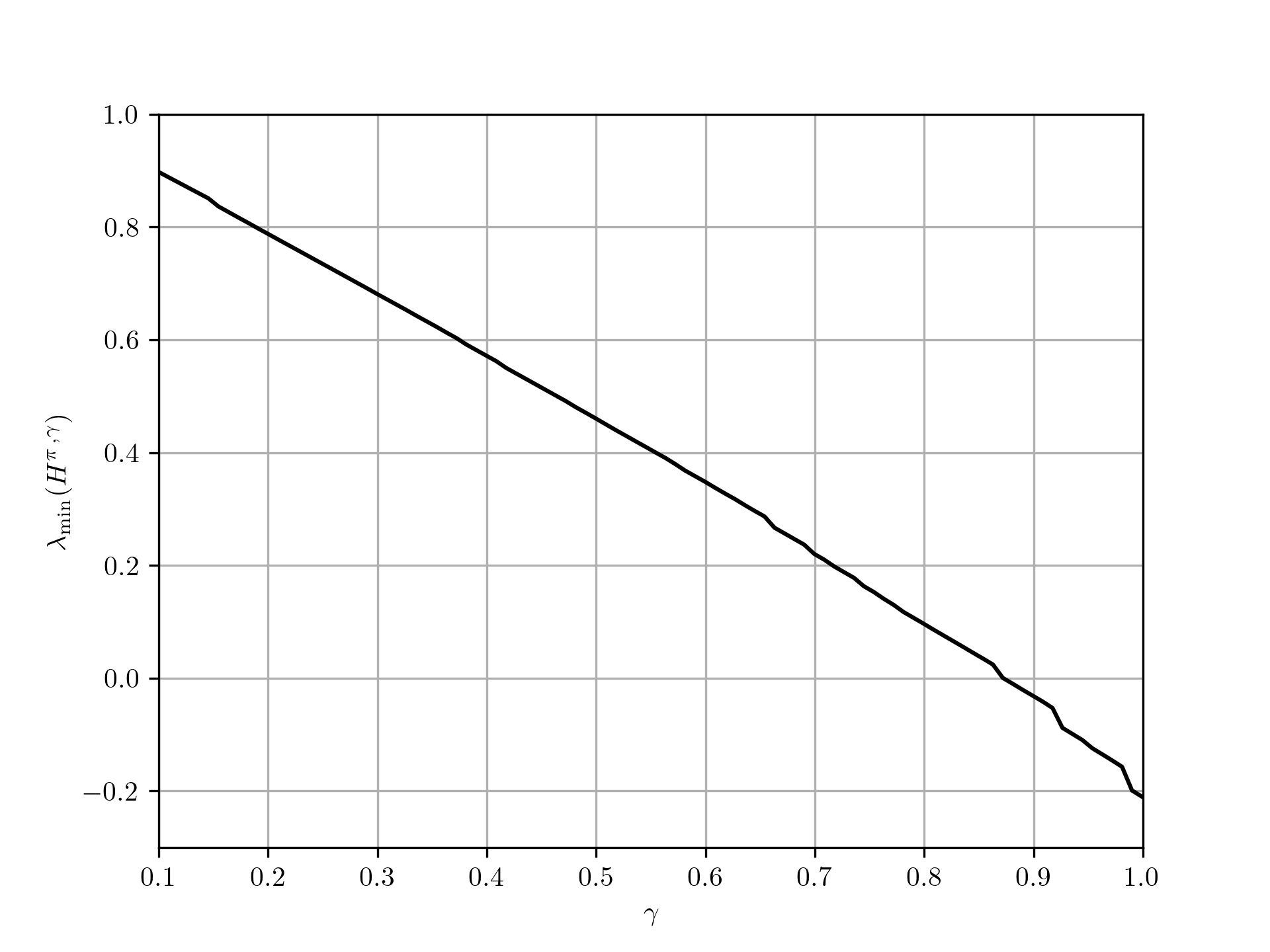}
    \caption{We consider an MDP with $m=500$, $n=40$ and different values of discount factor. We plot the minimum eigenvalue of $H^{\pi,\,\gamma}$ as a function of $\gamma$. In line with the results of Lemma~\ref{lemma: PD_H}, for high-values of the discount factor $\lambda_{\min}(H^{\pi,\,\gamma})<0$. For all tested values of $\gamma$ in the plot, the minimal residual method has achieved convergence to the solution. The convergence of some configurations in terms of 2-norm of the residuals is plotted in Figure~\ref{fig:MinRes_contraction_rate}.}
    \label{fig:MinRes_criteria}
\end{figure}

\begin{lemma}\label{lemma: PD_H}
Consider a general infinite-horizon discounted MDP with finite spaces. For any $\pi\in\Pi$, $1/\lambda_{\max}\left( P_s^{\pi} \right)\in (0,1]$ and $H^{\pi\,,\gamma}$ as defined in~\eqref{eq: symmetric_component} is positive definite if and only if $\gamma < 1/\lambda_{\max}\left( P_s^{\pi} \right)$.
\end{lemma}
\begin{proof}
Recall that, since for any $\pi\in\Pi$ $P^{\pi}$ is row-stochastic, its eigenvalues lie in the circle with radius 1 and center $(0, 0)$ in the complex plane. Consequently, $\text{Re}\left\{ \lambda \right\} \leq 1$ for all $\lambda\in\Lambda(P^{\pi})$ and $(1,0)\in \Lambda\left( P^{\pi} \right)$. 
Since $H^{\pi\,,\gamma}$ is symmetric by construction, studying its positive definiteness boils down to studying if the following condition hold
\begin{equation}\label{eq:conditions_Pd_iR}
\lambda_{\min}\left(I - \frac{\gamma}{2} \left( P^{\pi} + {P^{\pi}}^\top \right)  \right) > 0 \,.
\end{equation} 
Since
\begin{equation}
\lambda_{\min}\left(I - \frac{\gamma}{2} \left( P^{\pi} + {P^{\pi}}^\top \right)  \right) = 1 - \gamma \lambda_{\max}\left( P_s^{\pi} \right)\,,
\end{equation}
\eqref{eq:conditions_Pd_iR} can also be reformulated as follow
\begin{equation}
 1 - \gamma \lambda_{\max}\left( P_s^{\pi} \right) > 0\,. 
\end{equation}
We can therefore conclude that $H^{\pi\,,\gamma}\succ 0$ if and only if $\gamma < \frac{1}{\lambda_{\max}(P^{\pi}_s)}$. 
Finally, since $\text{Re}\left\{ \lambda \right\} \leq \lambda_{\max}\left( P^{\pi}_s \right)$ for all $\lambda \in \Lambda\left( P^{\pi} \right)$~\cite[Th. 1.20]{saad}, we can conclude that $1/\lambda_{\max}(P^{\pi}_s) \in (0, 1]$.
\end{proof}
The results of Lemma~\ref{lemma: PD_H} suggest that, in general, $H^{\pi,\,\gamma}\nsucc 0$ for large values of the discount factor. Empirical evidence though shows that iteration~\eqref{eq:MR_iteration} enjoys convergence also when $H^{\pi\,,\gamma}\nsucc 0$, indicating that the positive definite assumption can be restrictive. In particular, we conduct a benchmark using iteration~\eqref{eq:MR_iteration} to compute the cost associated with a policy $\pi$ for an MDP with $n=500$, $m=40$ and different values of discount factor. Convergence to $V^{\pi}$ is always achieved, even for high values of $\gamma$ for which, as displayed in Figure~\ref{fig:MinRes_criteria}, $\lambda_{\min}\left( H^{\pi\,,\gamma} \right)<0$.
Finally, even when $H^{\pi\,,\gamma}\succ 0$, the estimate on the contraction rate provided by the literature tends to be very loose and therefore not representative of the performance. An example of this is depicted in Figure~\ref{fig:MinRes_contraction_rate}, where we use $\hat{\kappa}_{\text{MR}}$ to denote the estimated contraction rate. 

\begin{figure}\centering
\includegraphics[scale=0.6]{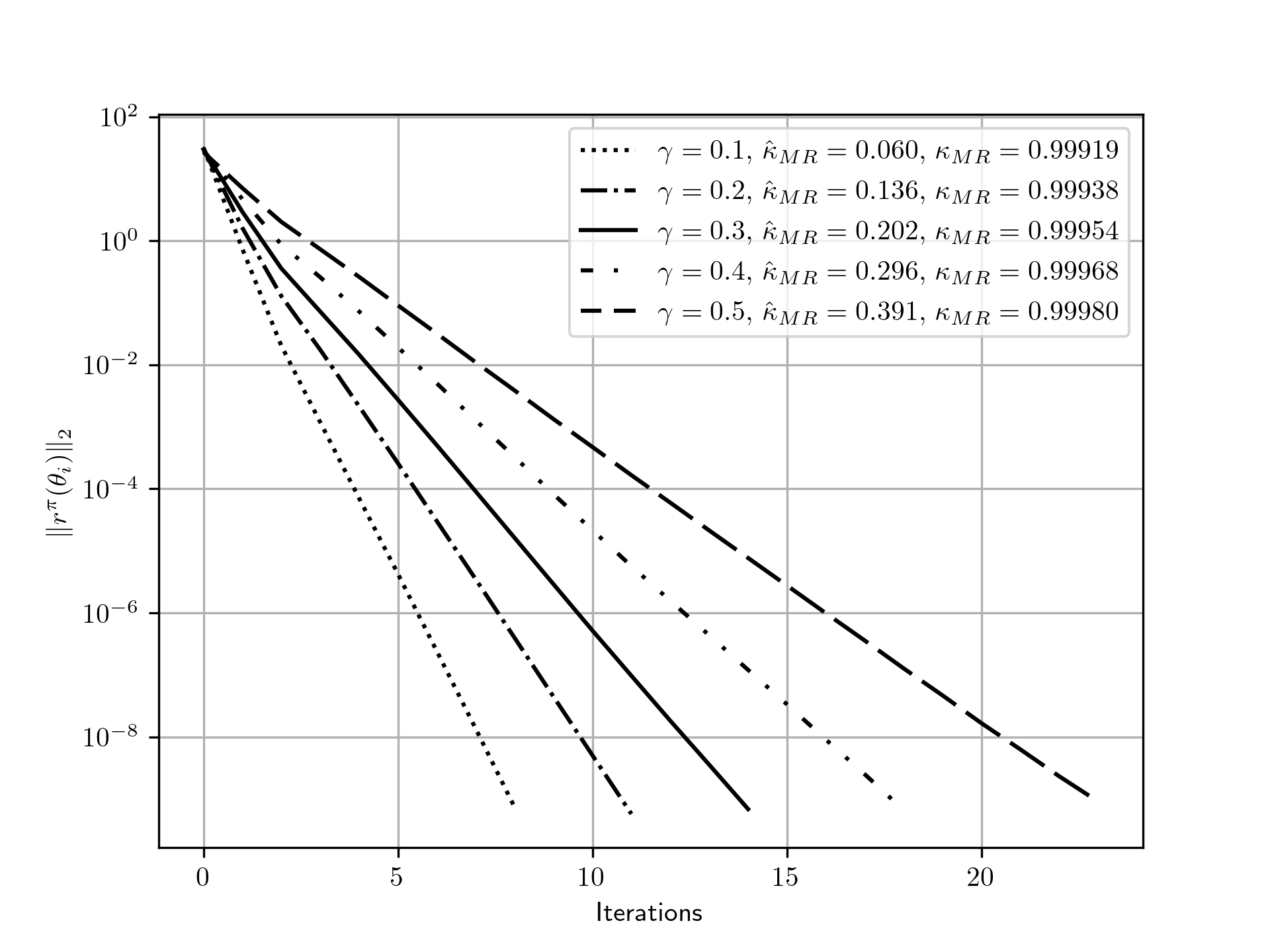}
\caption{We consider the same MDP and policy $\pi$ as in Figure~\ref{fig:MinRes_criteria} and vary the discount factor. We solve iteratively the resulting  policy evaluation tasks with the minimal residual method. For each benchmark we report in the legend $\kappa_{\text{MR}}$ and the empirically estimated contraction rate $\hat{\kappa}_{\text{MR}}$ and we plot the 2-norm of the residuals versus iteration number. }
\label{fig:MinRes_contraction_rate}
\end{figure}

\begin{figure}
    \centering
    \includegraphics[scale=0.6]{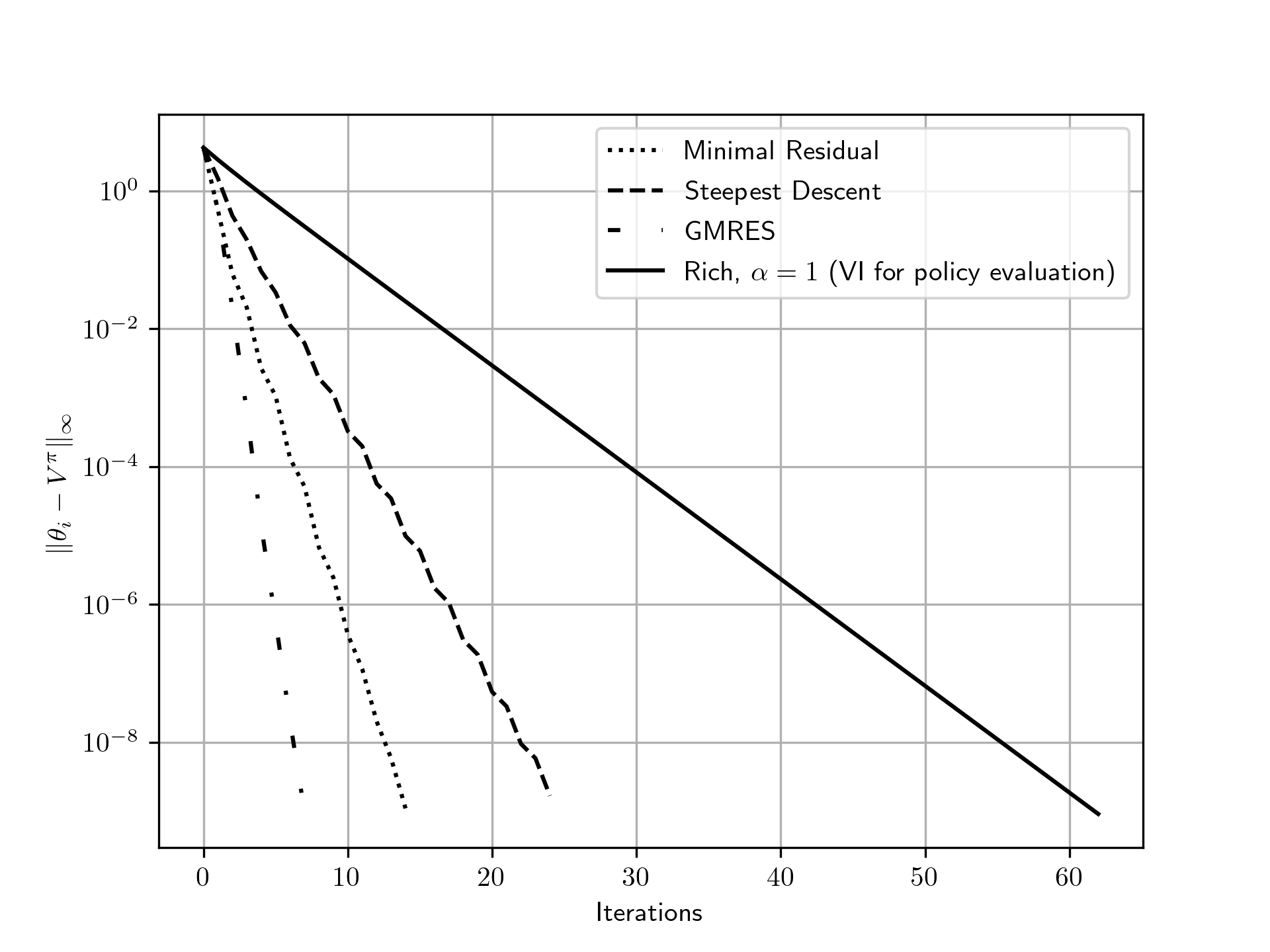}
    \caption{We consider a discounted MDP with $\gamma=0.7$, $n=100$ and $m=50$ and we solve the policy evaluation task for a random policy $\pi\in\Pi$ with Richardson's method with $\nu=1$, the minimal residual method, steepest descent and GMRES. The figure displays the distance in infinity-norm of the iterates from the solution versus number of iterations.}
    \label{fig:MinRes_criteria1}
\end{figure}

\begin{figure}
    \centering
    \includegraphics[scale=0.6]{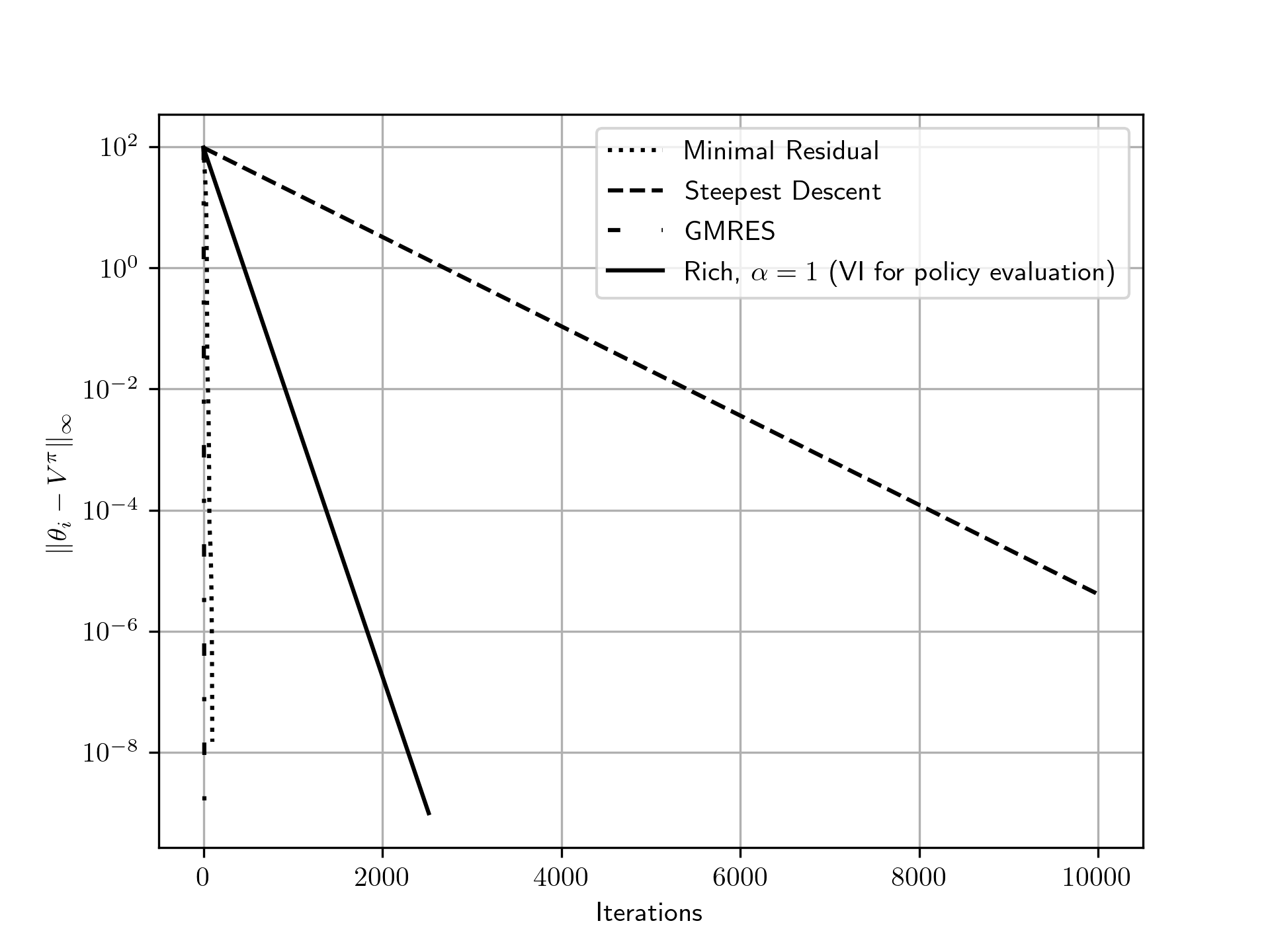}
    \caption{We consider the same  discounted MDP as in Figure~\ref{fig:MinRes_criteria1} but with $\gamma=0.99$ and we solve the policy evaluation task for a random policy $\pi\in\Pi$ with Richardson's method with $\nu=1$, the minimal residual method, steepest descent and GMRES. The figure displays the distance in infinity-norm of the iterates from the solution versus number of iterations.}
    \label{fig:MinRes_criteria2}
\end{figure}

\begin{figure}
    \centering
    \includegraphics[scale=0.6]{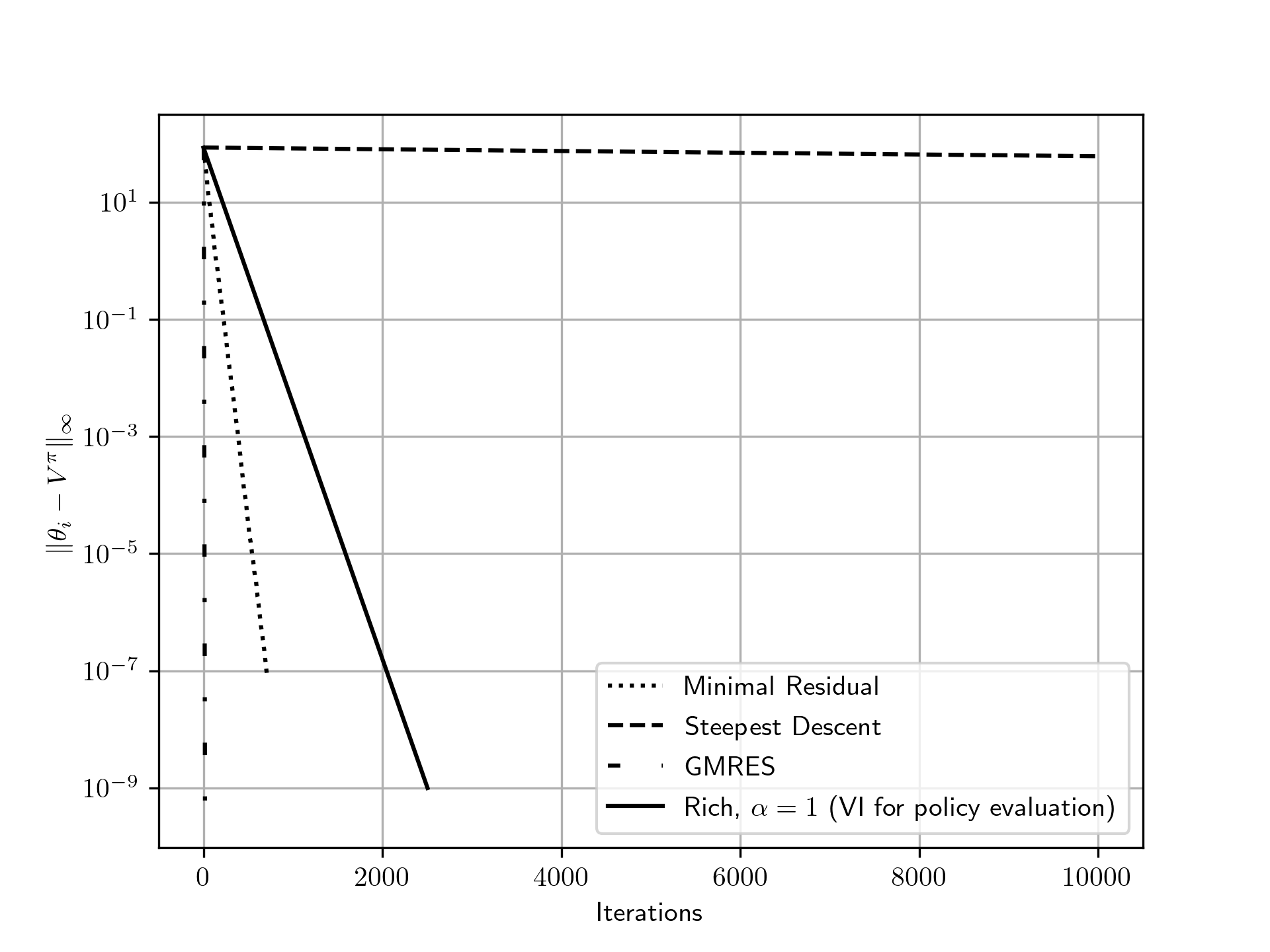}
    \caption{We consider a discounted MDP with $\gamma=0.99$, $n=100$ and $m=50$ and we solve the policy evaluation task for a random policy $\pi\in\Pi$ with Richardson's method with $\nu=1$, the minimal residual method, steepest descent and GMRES. The MDP is structurally different from the one deployed for the benchmarks in Figure~\ref{fig:MinRes_criteria2}. The figure displays the distance in infinity-norm of the iterates from the solution versus number of iterations.}
    \label{fig:MinRes_criteria3}
\end{figure}

\subsubsection{GMRES}\label{subsec:gmres}
GMRES~\cite{gmres} is an iterative method for linear systems with non-singular coefficient matrices and it belongs to the class of Krylov subspace methods~\cite[Chapter 6]{saad}. Starting from an initial guess $\Delta {\theta}_0\in\mathbb{R}^n$ with residual $\Phi^{\pi}_0 = -r^{\pi}(\bar{V}) - \left( I - \gamma P^{\pi} \right) \Delta {\theta}_0$, GMRES generates a sequence $\left\{ \Delta {\theta}_{i} \right\}$ of approximate solutions to~\eqref{eq: policy_eval} with
\begin{equation}
\Delta {\theta}_i = \arg \min_{\Delta \theta} \left\{ \Vert -r^{\pi}(\bar{V}) - \left( I - \gamma P^{\pi} \right) \Delta \theta \Vert_2 \,:\, \Delta \theta \in \Delta {\theta}_0 + \mathcal{K}_i \right\}\,,
\end{equation}
where $\mathcal{K}_i = \text{span}\left\{ \Phi^{\pi}_0,\, \left( I - \gamma P^{\pi}\right)\Phi^{\pi}_0,\, \left( I - \gamma P^{\pi}\right)^2\Phi^{\pi}_0,\, \dots, \, \left( I - \gamma P^{\pi}\right)^{i-1}\Phi^{\pi}_0 \right\}$ and it is also known as the $i$-th Krylov subspace~\cite[Chapter 6]{saad}. In other words, the $i$-th iterate is the vector that minimizes the 2-norm of the residual of~\eqref{eq: policy_eval} in the affine space $\Delta {\theta}_0 + \mathcal{K}_i$. Since working directly with $\left\{\Phi^{\pi}_0,\, \left( I - \gamma P^{\pi}\right)\Phi^{\pi}_0,\, \left( I - \gamma P^{\pi}\right)^2\Phi^{\pi}_0,\, \dots, \, \left( I - \gamma P^{\pi}\right)^{i-1}\Phi^{\pi}_0 \right\}$ is not computationally efficient and assuming that $\Vert \Phi^{\pi}_0 \Vert > 0$, the method constructs an orthonormal basis $\left\{ q_1, \dots, q_i \right\}$ where $q_1 = \Phi^{\pi}_0 / \Vert \Phi^{\pi}_0 \Vert_2$ and $\left\{ q_z \right\}_{z=2}^i$ are generated with the \textit{Arnoldi's method}~\cite{Arnoldi1951ThePO}. The output of this orthonormalization algorithm is the matrix $Q_i \in \mathbb{R}^{n \times i}$, whose columns are the vectors of the orthonormal basis, and an Hessenberg matrix $H_i\in \mathbb{R}^{(i+1)\times i}$ that satisfies the following relation
\begin{equation}\label{eq: gmres_hessenberg_eq}
A Q_i = Q_{i+1} H_i\,.
\end{equation} 
Since any $\Delta \theta \in \Delta {\theta}_0 + \mathcal{K}_i$ can be rewritten as $\Delta \theta = \Delta {\theta}_0 + Q_i y$ for some $y \in \mathbb{R}^i$ and by exploiting~\eqref{eq: gmres_hessenberg_eq}, we can rewrite the residual associated to $\Delta \theta$ as a function of $y$ as follows
\begin{equation}\label{eq: gmres_xy}
\Phi_i^{\pi} = \Phi^{\pi}_0 - \left( I - \gamma P^{\pi} \right) Q_i y =  Q_{i+1}\left( \Vert \Phi_0^{\pi} \Vert_2 e_1 - H_i y \right)\,,
\end{equation}
where $e_1 = [1,0,\dots,0]^{\top} \in\mathbb{R}^{i+1}$. The $i$-th iterate is selected as the $\Delta \theta \in \Delta {\theta}_0 + \mathcal{K}_i$ which minimizes the 2-norm of~\eqref{eq: gmres_xy}, and therefore as $\Delta {\theta}_i = \Delta {\theta}_0 + Q_i \tilde{y}$ where 
\begin{equation}
\tilde{y} = \arg\min_{y} \Vert \Vert \Phi_0^{\pi} \Vert_2 e_1 - H_i y \Vert_2\,.
\end{equation}
Unlike for the conjugate gradient method, an orthonormal basis of $\mathcal{K}_i$ can not be computed with a short recurrence. Consequently, when $i$ increases, the number of stored vectors also increases like $i$ and the number of multiplications like $0.5 i^2 n$. A practical variant of GMRES consists in restarting the algorithm after every $i$ iterations.

With exact arithmetic the sequence of iterates generated by GMRES converges to the solution of~\eqref{eq: policy_eval} in at most $\ell = \sum_{\lambda\in \Lambda(I - \gamma P^{\pi})} b(\lambda) \leq n$ iterations~\cite{Campbell1996GMRESAT}. For the large-scale cases where the minimal polynomial has a high degree then GMRES may require a significant number of iterations to reach convergence. In these scenarios it becomes particularly important to characterize its convergence properties in terms of improvement per iteration, as it is often intractable to run $\ell$-iterations. In this regard, it is known in the literature that GMRES has particularly favorable convergence properties when the eigenvalues of the coefficient matrix are clustered in a circle of center $(1,0)$ and radius $\xi < 1$. In~\cite{Campbell1996GMRESAT} the authors show that in these scenarios the contraction rate is determined by the radius $\xi$ of the cluster.
These theoretical results are confirmed by extensive numerical examples, where it is possible to observe that, while for coefficient matrices with non-clustered eigenvalues the norm of the residuals stagnates and does not notably decrease up to the very last iteration, for coefficient matrices with clustered eigenvalues significant and steady progress is observed starting from the first iteration~\cite[Figure 1]{gargiani_2023}.  

Linear systems arising from policy evaluation tasks in the considered problem setting have always a coefficient matrix with eigenvalues in the circle with center $(1,0)$ and radius $\gamma$~\cite[Lemma 1]{gargiani_2023}. Therefore, GMRES operates in the favorable convergence regime. The following proposition characterizes further the convergence properties of GMRES when deployed for policy evaluation for general, ergodic and regular MDPs. 

\begin{proposition}\label{prop: GMRES_contraction}
Consider a general infinite-horizon discounted MDP with finite spaces and the linear system for the evaluation of the cost associated to a policy $\pi \in \Pi$. The sequence $\left\{ \Phi^{\pi}({\theta}_i)\right\}$ generated by GMRES satisfies the following inequality
\begin{equation}\label{eq: GMRES_generalMDP}
\Vert\Phi^{\pi}({\theta}_i)\Vert_2 \leq C_1\gamma^{i} \Vert \Phi^{\pi}({\theta}_0)\Vert_2 \quad i=1,\dots\,,
\end{equation}
where $C_1>0$ is a constant and it is independent of $i$. If the MDP is ergodic and $h_{\pi}<n$, then
\begin{equation}\label{eq: GMRES_ergodicMDP}
\Vert\Phi^{\pi}({\theta}_{h_{\pi} + i})\Vert_2 \leq C_2\vert \lambda_{h_{\pi} + 1} \vert^{i} \Vert \Phi^{\pi}({\theta}_0)\Vert_2 \quad i = h_{\pi}+1,\dots\,,
\end{equation}
where $C_2>0$ is a constant and it is independent of $i$.
\end{proposition}

\begin{proof}
In the considered setting, the coefficient matrix of the linear system associated to a policy evaluation task has always eigenvalues contained in the circle of radius $\gamma$ and center $(1,0)$~\cite[Lemma 1]{gargiani_2023}. We can therefore directly use the results of Proposition 2.2 in~\cite{Campbell1996GMRESAT} to obtain 
Inequality~\eqref{eq: GMRES_generalMDP}. The same results can be used to prove Inequality~\eqref{eq: GMRES_ergodicMDP} for ergodic MDPs, where the first $h_{\pi}$ eigenvalues are treated as outliers. From Proposition~\ref{prop: irreducible_matrix_simple} and the fact that $1\leq b(\lambda) \leq a(\lambda)$, we can conclude that $h_{\pi}$ iterations are needed to process the first $h_{\pi}$ eigenvalues. After the first $h_{\pi}$ iterations, we can use Inequality (4) in~\cite[Proposition 2.2]{Campbell1996GMRESAT}, where the radius of the cluster is $\vert \lambda_{h_{\pi}+1}\vert<\gamma$.
\end{proof}
Proposition~\ref{prop: GMRES_contraction} suggests that for general MDPs the asymptotic convergence factor of GMRES is given by the radius $\gamma$ of the circle where the eigenvalues are clustered. For ergodic MDPs, GMRES takes $h_{\pi}$ iterations to process the eigenvalues of $P^{\pi}$ with modulus 1. For this phase the improvement per iteration is still described by~\eqref{eq: GMRES_generalMDP}. Then the asymptotic convergence factor is $\vert \lambda_{h_{\pi} + 1} \vert < \gamma$. Because of Proposition~\ref{prop: rhoA_greater1}, the results of Proposition~\ref{prop: GMRES_contraction} further imply that for the subclass of regular MDPs only one iteration is required to process the eigenvalue $(1,0)$ of $P^{\pi}$ and then for $i > 1$
\begin{equation}
\Vert \Phi^{\pi}({\theta}_i) \Vert_2 \leq C_2 \vert \lambda_2 \vert^i \Vert \Phi^{\pi}({\theta}_0) \Vert_2\,.
\end{equation}

As shown in Proposition~\ref{prop: GMRES_contraction}, for general and ergodic MDPs Richardson's method can not always achieve an asymptotic convergence rate faster than $\gamma$. The latter is attained with the choice $\nu=1$, which leads to value iteration for policy evaluation. These considerations and the results discussed in Proposition~\ref{prop: GMRES_contraction} advocate in favor of GMRES for ergodic MDPs in terms of better asymptotic contraction rate, especially when $h_{\pi} << n$ and $\vert \lambda_{h_{\pi} + 1} \vert << \gamma$.
\section{Numerical Evaluation}\label{sec6}
This section is dedicated to numerically evaluating the performance of iPI methods. For that we consider an MDP arising from a compartmental model commonly deployed in epidemiology to mathematically model infectious diseases~\cite{hethcote}. Details on the model are given in~\ref{subsec: SIS_model}, while in~\ref{subsec: benchmarks} we comment on the convergence and scalability performances of iPI methods when used to compute health policies for a population with large size. 
\subsection{Infectious Disease Model}\label{subsec: SIS_model}
We consider a Susceptible-Infectious-Susceptible (SIS) model to describe the evolution of an infectious disease with no immunity conferred by the previous infection~\cite{hethcote}. This model is characterized by two classes: the susceptible and the infective class. Individuals move from the susceptible class to the infective class and then back to the susceptible class upon recovery.
As in~\cite{sis_model}, we assume that the infectious period is fixed and equal to 1; that is, an individual that is infective at time $t$ remains infectious over the interval $[t, t+1]$, but she/he will recover upon diagnosis and effective treatment at time $t+1$, re-entering therefore in the susceptible class. We denote with $s(t)$ and $i(t)$ the number of individuals in the susceptible and infective class at time $t$, respectively.   Since we consider populations with a fixed-size, then
\begin{equation}\label{eq: SIS_dynamics}
s(t) + i(t) = N \quad \forall \, t\geq 0\,.
\end{equation}
From~\eqref{eq: SIS_dynamics} we obtain that only the information of a single class are needed to describe the state of the disease and we choose $s(t)$ to be the state of the model. From now on, we drop the dependency on time, unless needed.
We extend the static model in~\cite{sis_model} by adding a dynamic component which allows one to derive health policies. In particular, we design actions, introduce an action-state dependent stage-cost and make also the distribution of the driving event in~\cite{} action-dependent. For the actions, we propose $\mathcal{A} = \mathcal{A}_1 \times \mathcal{A}_2$ as action set, where $\mathcal{A}_1 = \left\{ 0,1,2,3,4 \right\}$ are the levels of hygiene measures and $\mathcal{A}_2 = \left\{ 0, 1, 2, 3 \right\}$ are the levels of social distancing imposed by the public health authorities. Considering influenza as example, levels of actions in increasing order in the first set could correspond to \textit{no measures}, \textit{frequent hand washing and disinfection}, \textit{mandatory surgical masks}, \textit{mandatory FFP2 masks} and \textit{mandatory full body protection}, respectively. Levels of actions in increasing order in the second set could correspond to \textit{no social restrictions}, \textit{mandatory social distancing}, \textit{restaurants and stores closure} and \textit{full lockdown}, respectively. As stage-cost, we consider the following multi-objective cost function
\begin{equation}
g(s, a) = w_f c_f(a) - w_q c_q(a) + w_h c_h(s) \,,
\end{equation}
where $w_f,\,w_q,\,w_h \geq 0$, $c_f:\mathcal{A}\rightarrow \mathbb{R}$ captures the financial costs and losses determined by the hygiene and social measures that are put in place, $c_q :\mathcal{A}\rightarrow [0,1]$ assigns to each action a quality of life score, and $c_h:\mathcal{S}\rightarrow \mathbb{R}$ maps the number of infected people to the medical cost incurred for their treatments. 

The driving event from the susceptible class to the infective class is the random variable $I(t)$, which represents the number of new infections occurring during the interval $[t, \,t+1]$. As in~\cite{sis_model}, we model the probability mass function for the driving event as a binomial distribution. Hence, the probability mass function for the driving event for any $a\in\mathcal{A}$ and $s\in\mathcal{S}$ is 
\begin{equation}
P\left(I(t) = i \,\vert\, s\right)=
\begin{cases}
\binom{s}{i} q(s,a)^{i} (1-q(s,a))^{s-i} ,\quad 0\leq i \leq s  \\
\,\, 0,\quad \text{else},
\end{cases}
\end{equation}
where $q : \mathcal{S}\times\mathcal{A} \rightarrow [0,1]$ with $q(s,a) = 1 - \exp\left({-\lambda(a)\beta(s)\psi(a)}\right)$ is the overall probability that a susceptible person becomes infected. The latter is a function of $\psi(a)$, $\beta(s)$ and $\lambda(a)$, which are, respectively, the probability that a susceptible person becomes infected upon contact with an infectious individual, the probability that the next interaction of a random susceptible person is with an infectious person, and the contact rate. These parameters are specific of the considered infectious disease. 

Finally, we consider $s(t) = N$ as an absorbing state since no infected individuals remain in the population. 
\subsection{Benchmarks}\label{subsec: benchmarks}
We study the empirical convergence properties of iPI methods with different inner solvers when they are used to solve infinite-horizon discounted MDPs arising from dynamic SIS-models.
For the benchmarks, we consider a dynamic SIS model with $N=10000$. Fig.~\ref{fig:SIS_transition_prob} displays the sparsity structure of the resulting transition probability matrices.

We implemented PI and iPI methods in Python, relying on numpy and scipy libraries for the dense and sparse linear algebra operations, respectively. We refer to \url{https://gitlab.ethz.ch/gmatilde/SIS_iPI} for more details on the selection of the other parameters of the model as well as for the code. All benchmarks were run on a Core(TM) i7-10750H CPU @ 2.60GHz. 

After the selection of the inner solver, we call the resulting iPI method {inexact \textit{inner solver name}-policy iteration} (i\textit{innersolvername}-PI). For instance, if Richardson's method is used as inner solver for the inexact policy evaluation step, the resulting scheme is called iRich-PI. To guarantee a finite-time execution of the code, in the case the stopping-condition is not met after $500$ inner iterations, we terminate the inner loop and update the outer solution. In addition, the overall algorithm is terminated if the suboptimality falls below a given threshold or in case the execution time exceeds 500 seconds.

As underlined by our theoretical analysis, the disocunt factor plays a fundamental role in determining the convergence properties of the inner solver and, therefore, the overall performance of an iPI scheme. Consequently, we study two different scenarios: small ($\gamma=0.1$) and large ($\gamma=0.9$) values of discount factor. In both cases and in line with the theoretical analysis, PI is the method converging in the smallest number of iterations. The inexactness of iPI methods results in a higher number of iterations required to achieve convergence to the solution, where convergence is intended within a certain tolerance. On the other hand, the CPU-time can be dramatically reduced thanks to the inexactness and if an efficient inner solver is deployed. iGMRES-PI is the fastest converging method in terms of CPU-time and it is $\times$ 1.46 and $\times$ 1.56 faster than PI in the small and large discount factor scenarios, respectively. iMinRes-PI shows also a good and stable performance across the two scenarios, while iSD-PI is competitive only in the small discount factor scenario. In the large discount factor scenario it indeed fails to achieve convergence within the given time-budget, and this is attributable to the exacerbating slow convergence rate of SD as inner solver for this specific scenario. iRich-PI performance is also not competitive since, in both scenarios, the method takes longer than PI to achieve convergence in terms of CPU-time. Similarly as for the iSD-PI method, this is ascribable to the slow convergence rate of the inner solver.
In Figure~\ref{fig:iPI_sub3} we study the impact of the discount factor on convergence in terms CPU-time of iPI methods with $\alpha = 0.1$ and PI when used to solve a dynamic SIS MDP with $N=10000$. As depicted in the plot, for low-values of discount factor iMinRes-PI is the fastest converging method, followed by iGMRES-PI, while for high-values of discount factor iGMRES-PI attains the best performance. As in the previous benchmarks, iRich-PI is not competitive and for high-values of discount 
factor it even fails to achieve convergence within the given time and inner-iterations limits. A similar consideration holds for iSD-PI, but for a larger interval of discount factor values. 
In Figure~\ref{fig:iPI_sub4} we benchmark iGMRES-PI wth $\alpha=0.1$ and PI on dynamic SIS MDPs with different population sizes and $\gamma=0.9$. As to be expected, for small-scale regimes PI tends to be the fastest converging method, but its performance does not scale well. For large-scale settings, indeed, iGMRES-PI converges up to more than $\times 3.7$ faster than PI.

From the conducted benchmarks we can conclude that, in line with the theoretical results, inexactness of iPI methods results in a slower convergence rate in terms of iterations, but a faster method in terms of CPU-time, provided an appropriate choice of the inner solver. The selection of the latter is indeed fundamental to obtain a competitive method. 

\begin{figure}
\centering
\begin{subfigure}{.5\textwidth}
  \centering
  \includegraphics[width=\linewidth]{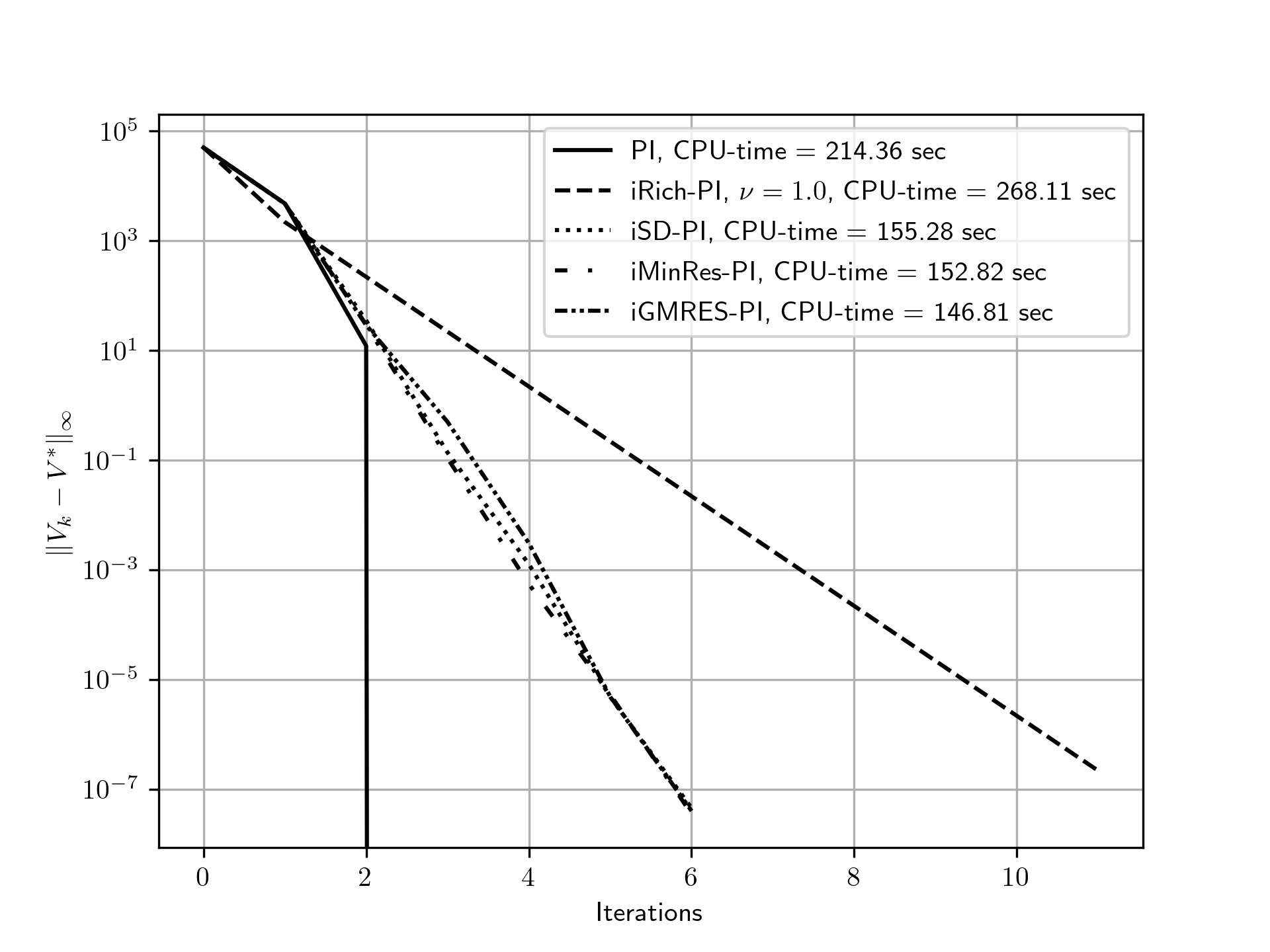}
  \caption{$\gamma=0.1$}
  \label{fig:iPI_sub1}
\end{subfigure}%
\begin{subfigure}{.5\textwidth}
  \centering
  \includegraphics[width=\linewidth]{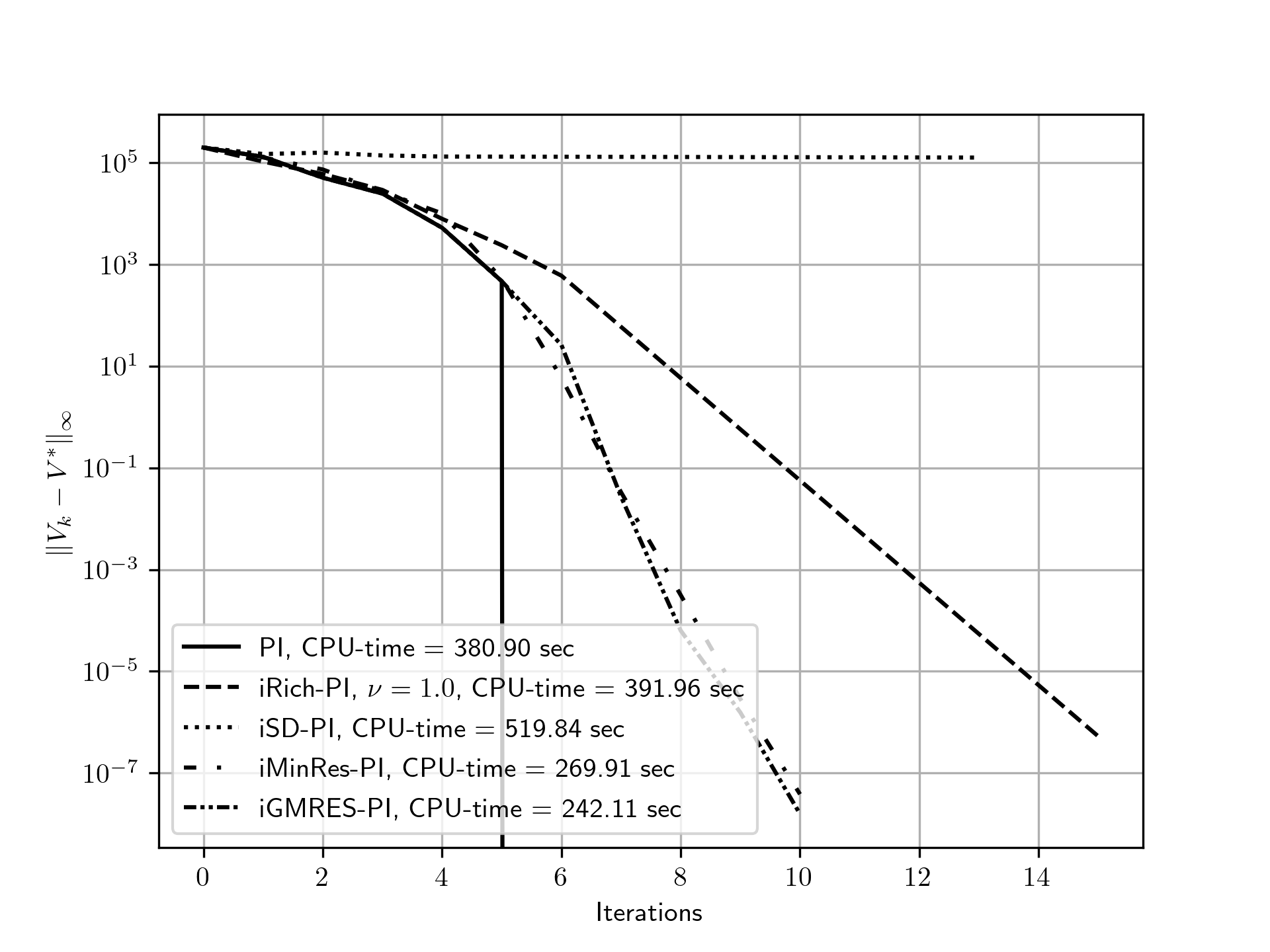}
  \caption{$\gamma=0.9$}
  \label{fig:iPI_sub2}
\end{subfigure}
\caption{We consider a SIS model with actions and population size of 10000. We solve the associated infinite-horizon discounted MDP with PI and different iPI methods with $\alpha = 0.1$. In particular, as inner solvers, we consider Richardson with $\nu=1$, steepest descent, minimal residual method and GMRES. We study the empirical convergence properties of these methods for two different values of discount factor and plot the distance in infinity-norm of the iterates from the solution versus number of iterations. In the legend, we also report the overall CPU-time in seconds.}
\label{fig:iPI_benchmarks}
\end{figure}

\begin{figure}
\centering
\begin{subfigure}{.8\textwidth}
  \centering
  \includegraphics[width=\linewidth]{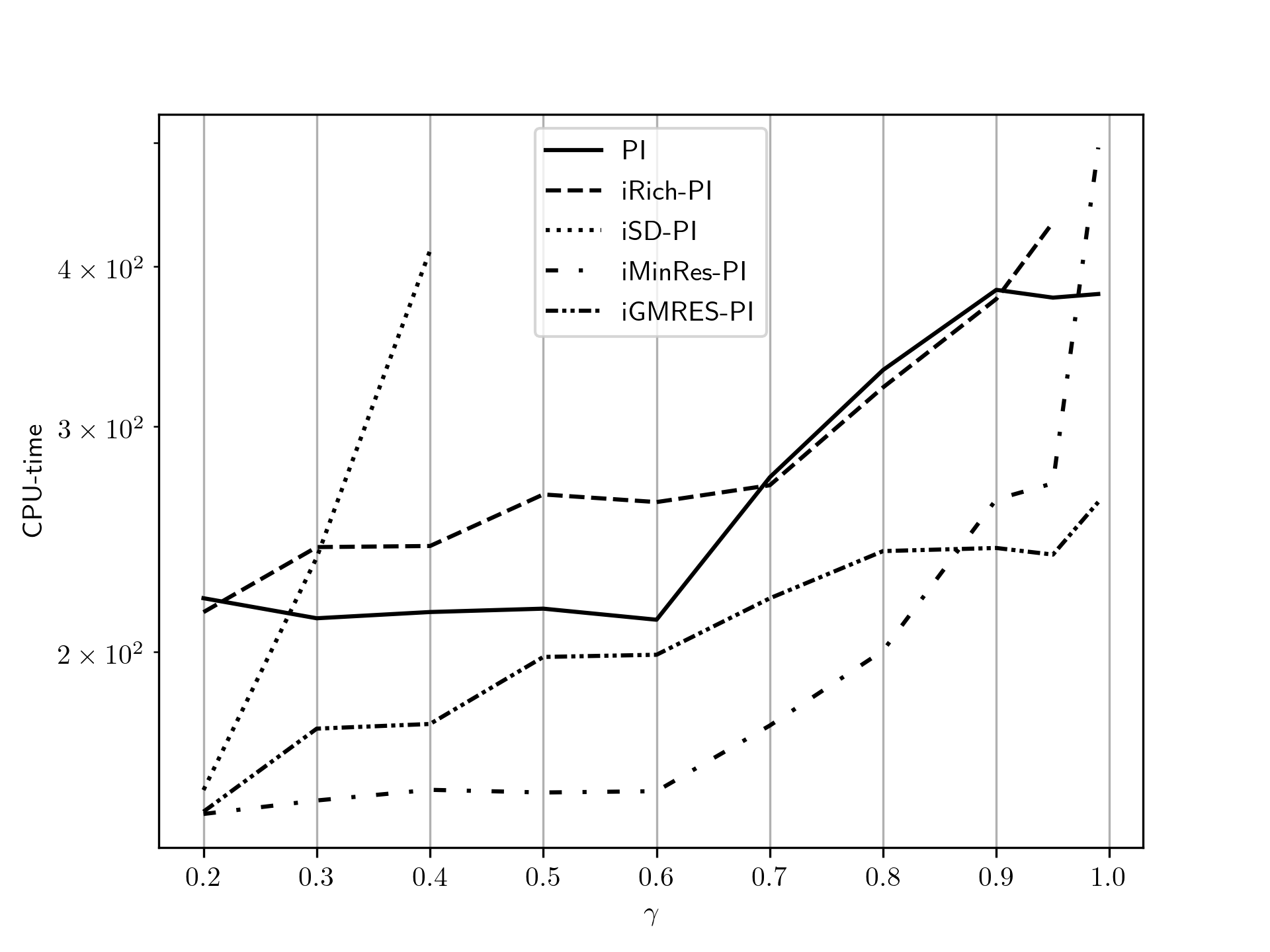}
  \caption{CPU-time versus discount factor values for a dynamic SIS MDP with $N=10000$. For the iPI methods we set $\alpha=0.1$.}
  \label{fig:iPI_sub3}
\end{subfigure}%

\begin{subfigure}{.8\textwidth}
  \centering
  \includegraphics[width=\linewidth]{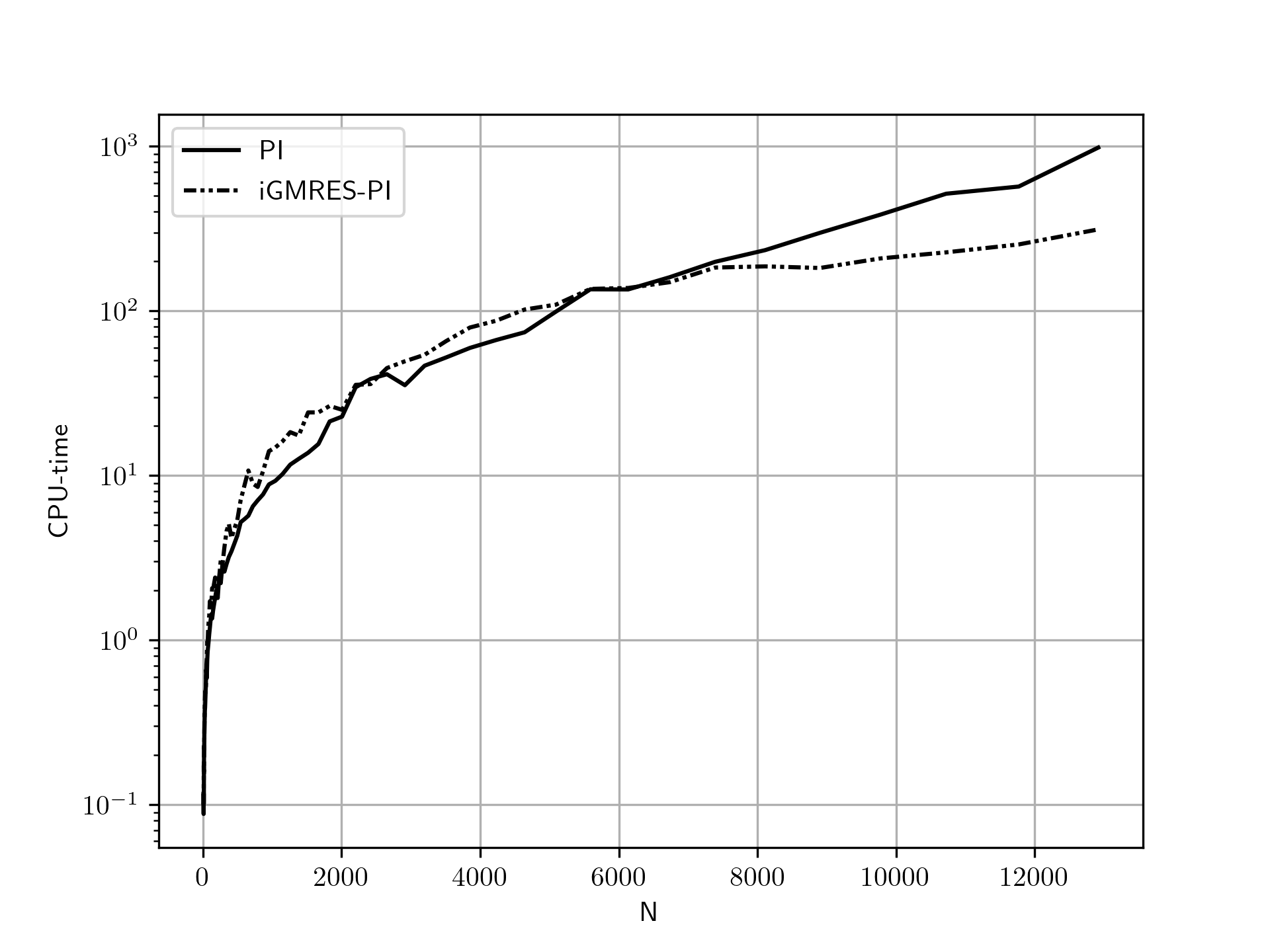}
  \caption{CPU-time versus population sizes of dynamic SIS MDPs with $\gamma=0.9$. For the iGMRES-PI we set $\alpha=0.1$.}
  \label{fig:iPI_sub4}
\end{subfigure}
\caption{}
\label{fig:iPI_cpu}
\end{figure}

\begin{figure}
    \centering
    \includegraphics[scale=0.25]{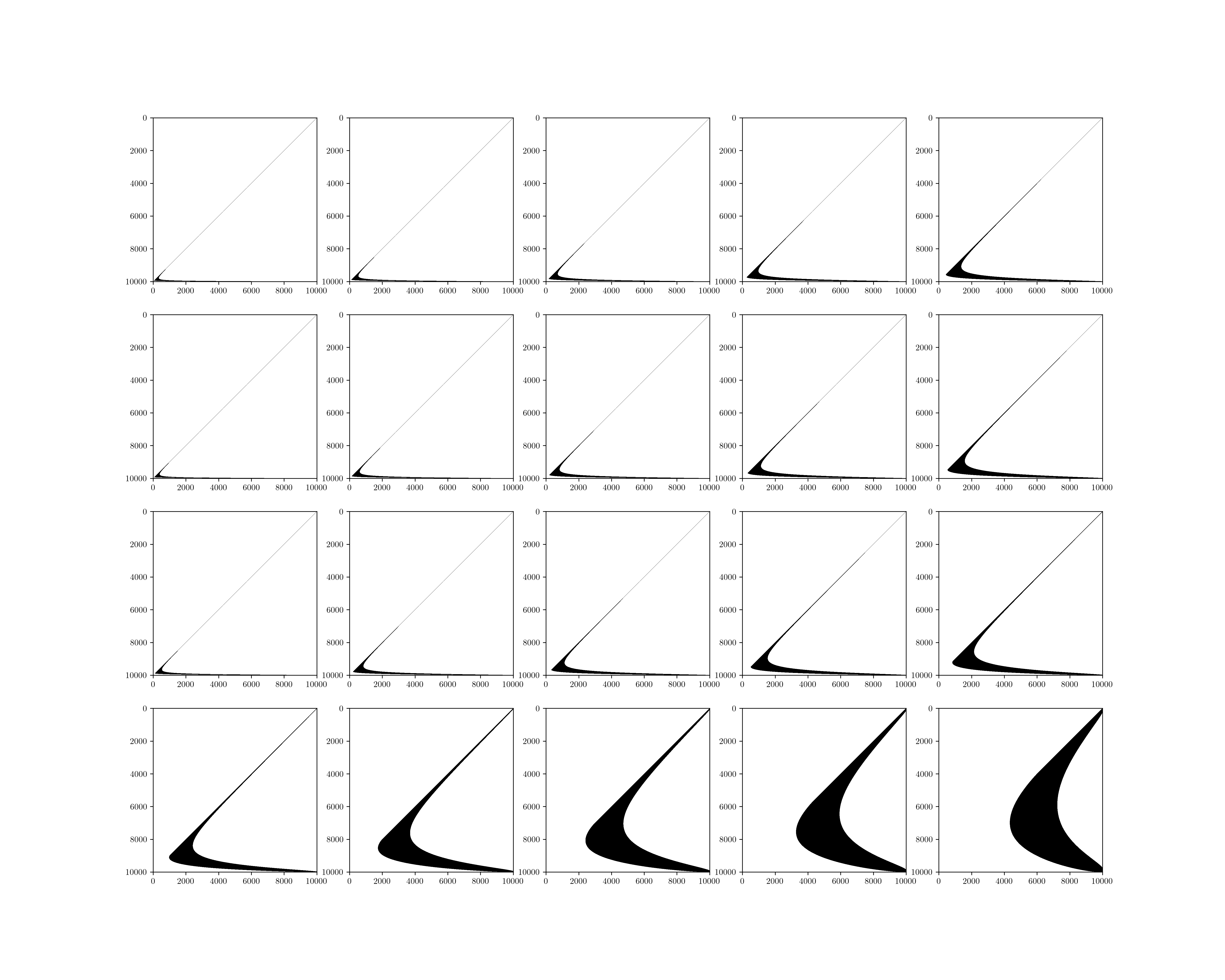}
    \caption{Graphical representation of the sparsity structure of the transition probability matrices for the dynamic SIS model deployed to generate the benchmarks in Fig.~\ref{fig:iPI_benchmarks}. White color is used for the zero entries and black color for the non-zero entries. The $(i,j)$-plot displays the transition probability matrix associated to the $(i,j)$-action pair.}
    \label{fig:SIS_transition_prob}
\end{figure}

\section{Conclusions}
We analyse a novel class of DP methods, inexact policy iteration methods, which are based on an approximate solution of the policy evaluation. After discussing their convergence properties and the importance of a good inner solver selection, we consider different iterative methods for policy evaluation and characterize their convergence behavior for the specific task considered. We conclude by showcasing the competitive performance of inexact policy iteration methods on large-scale MDPs arising from dynamic SIS models. In line with the derived theoretical results, our numerical evidence shows that value iteration as inner solver may result into a slow converging overall scheme. A much better choice is represented by inner solvers that adjust step and search direction based on locally available information, such as the minimal residual method and GMRES. 

Future extensions of the current work include an improved global convergence analysis, which should encompass the need for extra assumptions on the structural properties of the underlying MDP, as well as a distributed implementation of inexact policy iteration schemes, which would allow to overcome the memory limitations of a single machine and, consequently, to scale to potentially gigantic problem sizes making use of modern high-performance computer clusters.

\section*{Acknowledgement}
This work was supported by the European Research Council under
the Horizon 2020 Advanced under Grant 787845 (OCAL).


\bigskip
\bibliographystyle{abbrv}
\bibliography{sn-article}

\begin{thebibliography}{10}

\bibitem{DEUFLHARD1991366}
Global inexact {N}ewton methods for very large scale nonlinear problems.
\newblock {\em IMPACT of Computing in Science and Engineering}, 3(4):366--393, 1991.

\bibitem{inexactnewtonMR}
Newton-{MR}: {I}nexact {N}ewton method with minimum residual sub-problem solver.
\newblock {\em EURO Journal on Computational Optimization}, 10:100035, 2022.

\bibitem{Arnoldi1951ThePO}
W.~E. Arnoldi.
\newblock The principle of minimized iterations in the solution of the matrix eigenvalue problem.
\newblock {\em Quarterly of Applied Mathematics}, 9:17--29, 1951.

\bibitem{Bellman:1957}
R.~Bellman.
\newblock {\em Dynamic Programming}.
\newblock Princeton University Press, Princeton, NJ, USA, 1 edition, 1957.

\bibitem{nonnegative_matrices}
A.~Berman and R.~J. Plemmons.
\newblock {\em Nonnegative Matrices in the Mathematical Sciences}.
\newblock Society for Industrial and Applied Mathematics, 1994.

\bibitem{DB_book}
D.~P. Bertsekas.
\newblock {\em Dynamic Programming and Optimal Control, Vol. II}.
\newblock Athena Scientific, Belmont, Massachusets, 4th edition, 2012.

\bibitem{finance_MDP}
N.~Bäuerle and U.~Rieder.
\newblock {\em Markov Decision Processes with Applications to Finance}.
\newblock Springer Berlin Heidelberg, 2011.

\bibitem{Campbell1996GMRESAT}
S.~L. Campbell, I.~C.~F. Ipsen, C.~T. Kelley, and C.~D. Meyer.
\newblock {GMRES} and the minimal polynomial.
\newblock {\em BIT Numerical Mathematics}, 36:664--675, 1996.

\bibitem{inexact_Newton}
R.~S. Dembo, S.~C. Eisenstat, and T.~Steihaug.
\newblock Inexact newton methods.
\newblock {\em SIAM Journal on Numerical Analysis}, 19(2):400--408, 1982.

\bibitem{facchinei_bookI}
F.~Facchinei and J.-S. Pang.
\newblock {\em Finite-Dimensional Variational Inequalities and Complementarity Problems, Vol. I}.
\newblock Springer New York, NY, 1st edition, 2003.

\bibitem{facchinei_bookII}
F.~Facchinei and J.-S. Pang.
\newblock {\em Finite-Dimensional Variational Inequalities and Complementarity Problems, Vol. II}.
\newblock Springer New York, NY, 1st edition, 2003.

\bibitem{inexactgmresnewton}
A.~Flueck and H.-D. Chiang.
\newblock Solving the nonlinear power flow equations with an inexact {N}ewton method using {GMRES}.
\newblock {\em IEEE Transactions on Power Systems}, 13(2):267--273, 1998.

\bibitem{gargiani_2023}
M.~Gargiani, D.~Liao-McPherson, A.~Zanelli, and J.~Lygeros.
\newblock Inexact {GMRES} policy iteration for large-scale {M}arkov decision processes.
\newblock {\em IFAC-PapersOnLine}, 56(2):11249--11254, 2023.
\newblock 22nd IFAC World Congress.

\bibitem{gargiani_minibatch}
M.~Gargiani, A.~Martinelli, M.~R. Martinez, and J.~Lygeros.
\newblock Parallel and flexible dynamic programming via the mini-batch {B}ellman operator.
\newblock {\em IEEE Transactions on Automatic Control}, 69(1):455--462, 2024.

\bibitem{gargiani_2022}
M.~Gargiani, A.~Zanelli, D.~Liao-McPherson, T.~Summers, and J.~Lygeros.
\newblock Dynamic programming through the lens of semismooth {N}ewton-type methods.
\newblock {\em IEEE Control Systems Letters}, 6:1--1, 01 2022.

\bibitem{gargiani_extended_2022}
M.~Gargiani, A.~Zanelli, D.~Liao-McPherson, T.~Summers, and J.~Lygeros.
\newblock Dynamic programming through the lens of semismooth {N}ewton-type methods (extended version).
\newblock {\em https://arxiv.org/abs/2203.08678}, 2022.

\bibitem{gerschgorin_31}
S.~Gerschgorin.
\newblock Uber die {A}bgrenzung der {E}igenwerte einer {M}atrix.
\newblock {\em Izvestija Akademii Nauk SSSR, Serija Matematika}, 7(3):749--754, 1931.

\bibitem{iterative_solutions}
W.~Hackbusch.
\newblock {\em Iterative Solution of Large Sparse Systems of Equations}.
\newblock Springer International Publishing, 2nd edition, 2016.

\bibitem{hethcote}
H.~W. Hethcote.
\newblock The mathematics of infectious diseases.
\newblock {\em SIAM Review}, 42(4):599--653, 2000.

\bibitem{izmailov_book}
A.~F. Izmailov and M.~V. Solodov.
\newblock {\em Newton-Type Methods for Optimization and Variational Problems}.
\newblock Springer Cham, 1st edition, 2014.

\bibitem{yuchao_phd}
Y.~Li.
\newblock Approximate methods of optimal control via dynamic programming models.
\newblock {\em Doctoral Thesis in Electrical Engineering, KTH Royal Institute of Technology}, 2023.

\bibitem{MARTINEZ1995127}
J.~Martínez and L.~Qi.
\newblock Inexact {N}ewton methods for solving nonsmooth equations.
\newblock {\em Journal of Computational and Applied Mathematics}, 60(1):127--145, 1995.
\newblock Proceedings of the International Meeting on Linear/Nonlinear Iterative Methods and Verification of Solution.

\bibitem{agriculture_MDP}
L.~Nielsen, E.~Jørgensen, and S.~Højsgaard.
\newblock Embedding a state space model into a {M}arkov decision process.
\newblock {\em Annals OR}, 190:289--309, 10 2011.

\bibitem{puterman_mdp}
M.~L. Puterman.
\newblock {\em Markov Decision Processes: Discrete Stochastic Dynamic Programming}.
\newblock John Wiley \& Sons, Inc., USA, 1st edition, 1994.

\bibitem{inexactnewtoncg}
A.~Rieder.
\newblock Inexact {N}ewton regularization using conjugate gradients as inner iteration.
\newblock {\em SIAM Journal on Numerical Analysis}, 43(2):604--622, 2005.

\bibitem{rockafellar_1976}
R.~T. Rockafellar.
\newblock Monotone operators and the proximal point algorithm.
\newblock {\em SIAM Journal on Control and Optimization}, 14(5):877--898, 1976.

\bibitem{saad}
Y.~Saad.
\newblock {\em Iterative methods for sparse linear systems}.
\newblock SIAM, 2003.

\bibitem{gmres}
Y.~Saad and M.~H. Schultz.
\newblock {GMRES}: {A} generalized minimal residual algorithm for solving nonsymmetric linear systems.
\newblock {\em SIAM Journal on Scientific and Statistical Computing}, 7(3):856--869, 1986.

\bibitem{White1993ASO}
D.~J. White.
\newblock A survey of applications of {M}arkov decision processes.
\newblock {\em Journal of the Operational Research Society}, 44:1073--1096, 1993.

\bibitem{sis_model}
R.~Yaesoubi and T.~Cohen.
\newblock Generalized {M}arkov models of infectious disease spread: {A} novel framework for developing dynamic health policies.
\newblock {\em European Journal of Operational Research}, 215(3):679--687, 2011.

\end{thebibliography}


\end{document}